\documentclass{amsart}
\usepackage{palatino}
\usepackage{amsfonts}
\usepackage{amssymb}
\usepackage{amscd}
\usepackage{xypic}

\usepackage{ytableau}

\usepackage{amsmath}
\usepackage{mathrsfs}
\usepackage{graphicx}
\usepackage{epstopdf}
\usepackage{color}
\newcommand\numberthis{\addtocounter{equation}{1}\tag{\theequation}}

\usepackage{multicol}
\usepackage[shortlabels]{enumitem}

\usepackage{tikz}
\usetikzlibrary{positioning}
\usetikzlibrary{matrix}
\usepackage{float}
\usepackage{comment}

\tikzstyle{bsq}=[rectangle, draw, thick, minimum width=.5cm, minimum height=.5cm]
\tikzstyle{bver}=[rectangle, draw, thick, minimum width=1cm, minimum height=2cm]
\tikzstyle{bhor}=[rectangle, draw, thick, minimum width=2cm, minimum height=1cm]

 \usepackage[margin=1.2in]{geometry}

\newtheorem{theorem}{Theorem}[section]
\newtheorem{proposition}[theorem]{Proposition}
\newtheorem{corollary}[theorem]{Corollary}
\newtheorem{lemma}[theorem]{Lemma}
\theoremstyle{definition}
\newtheorem{definition}[theorem]{Definition}
\newtheorem{remark}[theorem]{Remark}
\newtheorem{example}[theorem]{Example}

\newtheorem{conjecture/question}[theorem]{Conjecture/Question}

\newtheorem{remark/definition}[theorem]{Remark/Definition}
\newtheorem{terminology/notation}[theorem]{Terminology/Notation}

\def\OO{\mathcal{O}}
\def\cT{\mathcal{T}}
\def\cN{\mathcal{N}}
\def\cD{\mathcal{D}}
\def\cB{\mathcal{B}}
\def\cA{\mathcal{A}}
\def\F{\mathcal{F}}
\def\P{\mathcal{P}}

\def\E{\mathcal{E}}
\def\G{\mathcal{G}}

\def\L{\mathcal{L}}
\def\I{\mathcal{I}}

\def\cM{\mathcal{M}}

\def\cR{\mathcal{R}}
\def\rr{\overline{\mathcal{R}}}
\def\cZ{\mathcal{Z}}

\def\cX{\mathcal{X}}

\def\mm{\overline{\mathcal{M}}}

\def\rr{\overline{\mathcal{R}}}

\def\pm{\widetilde{\mathcal{M}}}

\def\ttem{\overline{\mathcal{M}}^{\circ}}

\def\tsu{\mathcal{SU}_{13}^{\sharp}(2,\omega,8)}
\def\tm13{\cM_{13}^{\sharp}}
\def\tr13{\cR_{13}^{\sharp}}
\def\tc13{\mathcal{C}^{\sharp}_{13}}
\def\tcr13{\mathcal{RC}^{\sharp}_{13}}
\def\tms13{\mathfrak{M}_{13}^{\sharp}}
\def\trs13{\mathfrak{R}_{13}^{\sharp}}
\def\rtsu{\mathfrak{SU}_{13}^{\sharp}(2,\omega,8)}
\def\rsu{\mathfrak{R}\mathcal{SU}_{13}^{\sharp}(2,\omega,8)}

\def\PPP{{\textbf P}}

\newcommand{\PP}{\mathbb{P}}
\newcommand{\RR}{\mathbb{R}}
\newcommand{\fM}{\mathfrak{M}}
\newcommand{\M}{\overline{M}}
\newcommand{\MM}{\overline{\mathfrak{M}}}
\newcommand{\Grd}{\widetilde{\mathfrak{G}}^r_d}
\newcommand{\fU}{\mathfrak{U}}

\newcommand{\cO}{\mathcal{O}}

\newcommand{\Trop}{\operatorname{Trop}}
\newcommand{\trop}{\operatorname{trop}}
\newcommand{\ddiv}{\operatorname{div}}

\newcommand{\PL}{\operatorname{PL}}

\newcommand{\Pic}{\operatorname{Pic}}
\newcommand{\Sym}{\operatorname{Sym}}

\newcommand{\vir}{\mathrm{virt}}

\newcommand{\an}{\mathrm{an}}

\begin{document}
\title{The non-abelian Brill-Noether divisor
 on $\mm_{13}$ and the Kodaira dimension of $\rr_{13}$}
\author[G. Farkas]{Gavril Farkas}

\address{Gavril Farkas: Humboldt-Universit\"at zu Berlin, Institut f\"ur Mathematik
\hfill \newline\texttt{}
\indent Unter den Linden 6, 10099 Berlin, Germany} \email{{\tt farkas@math.hu-berlin.de}}

\author[D. Jensen]{David Jensen}

\address{David Jensen: Department of Mathematics, University of Kentucky \hfill
\hfill \newline\texttt{}
 \indent 733 Patterson Office Tower, Lexington, KY 40506–0027, USA}
 \email{{\tt dave.jensen@uky.edu}}

\author[S. Payne]{Sam Payne}

\address{Sam Payne: Department of Mathematics, University of Texas at Austin
\hfill \newline\texttt{}
 \indent 2515 Speedway, RLM 8.100, Austin, TX 78712, USA} \email{{\tt sampayne@utexas.edu}}

\begin{abstract}
The paper is devoted to highlighting several novel aspects of the moduli space of curves of genus $13$, the first genus $g$ where phenomena related to $K3$ surfaces no longer govern the birational geometry of $\mm_g$. We compute the class of the non-abelian Brill-Noether divisor on $\mm_{13}$ of curves that have a stable rank $2$ vector bundle with canonical determinant and many sections. This provides the first example of an effective divisor on $\mm_g$ with slope less than $6+\frac{10}{g}$. Earlier work on the Slope Conjecture suggested that such divisors may not exist. The main geometric application of our result is a proof that the Prym moduli space $\rr_{13}$ is of general type. Among other things, we also prove the Bertram-Feinberg-Mukai and the Strong Maximal Rank Conjectures on $\mm_{13}$.
\end{abstract}

\maketitle

\vskip 5pt

\setcounter{tocdepth}{1}
\tableofcontents

\section{Introduction}

One of the defining achievements of modern moduli theory is the result of Harris, Mumford and Eisenbud  that $\mm_g$ is of general type for $g\geq 24$ \cite{HM, EH2}. An essential step in their proof is the calculation of the class of the \emph{Brill-Noether divisor}
$\mm_{g,r}^d$ consisting of those curves $X$ of genus $g$ such that $G^r_d(X)\neq \emptyset$ in the case $\rho(g,r,d):=g-(r+1)(g-d+r)=-1$.  Recall that the \emph{slope} of an effective divisor $D$ on $\mm_g$ not containing any of the boundary divisors $\Delta_i$ in its support is defined as the quantity $s(D):=a/\mathrm{min}_i\ b_i$, where $[D]=a\lambda-b_0\delta_0-\cdots -b_{\lfloor \frac{g}{2}\rfloor} \delta_{\lfloor \frac{g}{2}\rfloor}\in CH^1(\mm_g)$. Eisenbud and Harris showed that the slope of $\mm_{g,r}^d$ is $\frac{a}{b_0} = 6+\frac{12}{g+1}$ \cite{EH2}. After these seminal results from the 1980s, the fundamental question arose whether one can construct effective divisors $D$ on $\mm_g$ of slope $s(D)<6+\frac{12}{g+1}$ by using conditions defined in terms of  \emph{higher rank} vector bundles on curves.

Each effective divisor $D$ on $\mm_g$ of slope $s(D)<6+\frac{12}{g+1}$ must contain the locus $\mathcal{K}_g\subseteq  \cM_g$ of curves lying on a $K3$ surface \cite{FP}. Since curves on $K3$ surfaces possess stable rank two vector bundles with canonical determinant and unexpectedly many sections \cite{Laz, Mu, Vo}, it is then natural to focus on conditions defined in terms of rank two vector bundles with canonical determinant.

\vskip 4pt

For a smooth curve $X$ of genus $g$, let  $SU_X(2, \omega)$ be the moduli space of semistable rank $2$ vector bundles $E$ on $X$ with $\mbox{det}(E)\cong \omega_X$. For  $k\geq 0$, Bertram-Feinberg \cite[Conjecture, p.~2]{BF1} and Mukai \cite[Problem~4.8]{Mu} conjectured that for a general curve $X$ the  rank $2$  Brill-Noether locus
\[
SU_X(2,\omega, k):=\bigl\{E\in SU_X(2, \omega_X): h^0(X,E)\geq k\bigr\}
\]
has dimension $\beta(2,g,k):=3g-3-{k+1\choose 2}$. For a general curve $X$ the \emph{Mukai-Petri map} 
\begin{equation}\label{eq:mukaipetri}
\mu_E\colon \mbox{Sym}^2 H^0(X,E)\rightarrow H^0\bigl(X, \mathrm{Sym}^2(E)\bigr)
\end{equation}
is injective for each $E\in SU_X(2,\omega)$ \cite{Te}.  As a consequence, $SU_X(2,\omega,k)$ has the expected dimension $\beta(2,g,k)$, if it is nonempty.
There are numerous partial results on the non-emptiness of $SU_X(2, \omega, k)$ \cite{LNP, T2, Zh}, although still no proof in full generality.

\vskip 3pt

Assume now  $3g-3={k+1\choose 2}$. Then generically  $SU_X(2,\omega,k)$ consists of finitely many vector bundles, if it is nonempty.
We consider the \emph{non-abelian} Brill-Noether divisor $\mathcal{M}\mathcal{P}_g$ on $\cM_g$ consisting of curves $[X]$ for which there exists
$E\in SU_X(2, \omega_X,k)$ such that the Mukai-Petri map $\mu_E$ is not an isomorphism.  In this paper, we focus on the first genuinely interesting case\footnote{It is left to the reader to show that in the previous cases $k=5, 6$, the corresponding divisors
$\mathcal{M}\mathcal{P}_6$ and $\mathcal{M}\mathcal{P}_8$ are supported on the loci, in $\cM_6$ and $\cM_8$ respectively, of curves failing the Petri Theorem.}
\[
g=13 \ \ \mbox{ and }\ \ k=8.
\]
Our first main result proves this case of the Bertram-Feinberg-Mukai Conjecture and computes the class of the closure of the non-abelian Brill-Noether divisor.

\begin{theorem}\label{thm:mp}
A general curve $X$ of genus $13$ carries precisely three stable vector bundles $E\in SU_X(2,\omega, 8)$. The closure in $\mm_{13}$ of the non-abelian Brill-Noether divisor on $\cM_{13}$
\[
\mathcal{M}\mathcal{P}_{13}:=\Bigl\{[X]\in \cM_{13}: \exists E\in \mathcal{SU}_X(2, \omega,8)\ with \ \mu_E\colon \mathrm{Sym}^2 H^0(E)\stackrel{\ncong}\rightarrow H^0\bigl(\mathrm{Sym}^2(E)\bigr)\Bigr\}
\]
has slope equal to
\[
s\bigl([\overline{\mathcal{M}\mathcal{P}}_{13}]\bigr)=\frac{4109}{610}=6.735...<6+\frac{10}{13}=6.769\ldots
\]
\end{theorem}

To explain the significance of this result, we recall that several infinite series of examples of divisors on $\mm_g$ for $g\geq 10$ with slope less than $6+\frac{12}{g+1}$ have been constructed in  \cite{FP, F, Kh, FJP}, using syzygies on curves. Quite remarkably,  the slopes $s(D)$ of all these divisors $D$ on $\mm_g$
satisfy
\[
6+\frac{10}{g}\leq s(D)<6+\frac{12}{g+1}.
\]
The slope $6+\frac{12}{g+1}$ appears as both the slope of the Brill-Noether divisors $\mm_{g,r}^d$, and as the slope of a Lefschetz pencil of curves of genus $g$ on a $K3$ surface.  Similarly, $6+\frac{10}{g}$ is the slope of the family of curves $\{X_t'\}_{t\in \PP^1}$ in $\Delta_0\subseteq \mm_g$ obtained from a Lefschetz pencil $\{X_t\}_{t\in \PP^1}$ of curves of genus $g-1$ on a $K3$ surface $S$ by identifying two sections corresponding to base points of the pencil. The natural question has been therefore raised in \cite[p.~2]{CFM}, whether a slight weakening of the Harris-Morrison Slope Conjecture \cite{HMo} remains true and the inequality
\begin{equation}\label{eq:ineq1}
s(D)\geq 6+\frac{10}{g}
\end{equation}
holds for every effective divisor $D$ on $\mm_g$. Results from \cite{FP, Ta} imply that inequality \eqref{eq:ineq1} holds for all $g\leq 12$. In particular, the divisor $\overline{\mathcal{K}}_{10}$ on $\mm_{10}$ consisting of curves lying on $K3$ surfaces,  which was shown in \cite{FP} to be the original counterexample to the Slope Conjecture, satisfies $s(\overline{\mathcal{K}}_{10})=7=6+\frac{10}{g}$. On $\mm_{12}$, since a general curve of genus $11$ lies on a $K3$ surface, it follows that the pencils $\{X_t'\}_{t\in \PP^1}$ cover the boundary divisor $\Delta_0\subseteq \mm_{12}$, and consequently the inequality (\ref{eq:ineq1}) holds. Therefore $13$ is the smallest genus where inequality (\ref{eq:ineq1}) can be tested, and Theorem \ref{thm:mp} provides a negative answer to the question posed in \cite{CFM}.

\vskip 3pt

\subsection{The Kodaira dimension of the Prym moduli space $\rr_{13}$.}

One application of Theorem \ref{thm:mp} concerns the birational geometry of the moduli space $\rr_g$ of Prym curves of genus $g$.
The Prym moduli space $\cR_g$ classifying pairs $[X, \eta]$, where $X$ is a smooth curve of genus $g$ and $\eta$ is a $2$-torsion point in $\mbox{Pic}^0(X)$, has been classically used to parametrize moduli of abelian varieties via the Prym map $\cR_g\rightarrow \cA_{g-1}$
\cite{B}. The Deligne-Mumford compactification $\rr_g$ is uniruled for $g \leq 8$ (see \cite{FV} and references therein), and was previously known to be of general type for $g \geq 14$ and $g \neq 16$ \cite{FL, Br}\footnote{The problem of determining the Kodaira dimension of $\rr_{16}$ remains open. It was proven in \cite{FL} that the Prym-Green Conjecture on $\rr_{16}$ implies that $\rr_{16}$ is of general type. However, as shown in \cite[Proposition 4.4]{CEFS} there is strong indication that the Prym-Green Conjecture fails in genus $16$.}.

\begin{theorem}\label{thm:r13}
The Prym moduli space $\rr_{13}$ is of general type.
\end{theorem}

\noindent In particular, $13$ is  the smallest genus $g$ for which it is known that $\rr_g$ is of general type. The proof of Theorem \ref{thm:r13} takes full advantage of Theorem \ref{thm:mp}. It also uses the \emph{universal theta divisor} $\Theta_{13}$, defined as the locus of Prym curves $[X, \eta]\in \cR_{13}$ for which there exists a vector bundle $E\in SU_X(2, \omega, 8)$ such that
$H^0(X, E\otimes \eta)\neq 0.$  In an indirect way (to be explained later), we calculate the class $[\overline{\Theta}_{13}]$ of the closure of $\Theta_{13}$ inside $\rr_{13}$ and show that
\begin{equation}\label{eq:cond13}
K_{\rr_{13}} \in \mathbb Q_{>0}\Bigl\langle \lambda, [\overline{\Theta}_{13}], [\overline{D}_{13:2}], \mbox{ boundary divisors}\Bigr\rangle,
\end{equation}
where $D_{13:2}$ is the effective divisor on $\cR_{13}$ introduced in \cite{FL}  consisting of Prym curves $[X, \eta]$ for which $\eta$ can be written as the difference of two effective divisors of degree $6$ on $X$.
Since $\lambda$ is big, it follows that $K_{\rr_{13}}$ is also big. Theorem \ref{thm:r13} follows, since the singularities of $\rr_g$ do not impose adjunction conditions \cite{FL}.

\vskip 3pt

\subsection{The Strong Maximal Rank Conjecture on $\mm_{13}$.}

The proofs of both Theorems \ref{thm:mp} and \ref{thm:r13} are indirect and proceed through a study of the failure locus of the Strong Maximal Rank Conjecture \cite{AF} on $\mm_{13}$. For a general curve $X$ of genus $13$ the Brill-Noether locus $W^5_{16}(X)$ is $1$-dimensional, and $W^6_{16}(X)=\emptyset$.
Counting dimensions shows that the multiplication map
\[
\phi_L\colon \mbox{Sym}^2 H^0(X,L)\rightarrow H^0(X,L^{\otimes 2})
\]
has at least a one-dimensional kernel, since $h^0\bigl(X,L^{\otimes 2}\bigr)=2 \deg(L)+1-g=20$.  The space of pairs $[X,L]$ such that $\mbox{Ker}(\phi_L)$ is at least 2-dimensional therefore has expected codimension 2 in the parameter space $\mathfrak{G}^5_{16}$ of all such pairs $[X, L]$.
 Since the fibers of the map $\sigma\colon \mathfrak{G}^5_{16}\rightarrow \cM_{13}$ are in general $1$-dimensional, the push-forward of this locus is expected to be a divisor on $\cM_{13}$.

Our next result verifies this case of the Strong Maximal Rank Conjecture and computes the class of the closure of the divisorial part of the failure locus.  This is essential input for the calculation of the non-abelian Brill-Noether divisor class in Theorem~\ref{thm:mp} and hence for the proof of Theorem~\ref{thm:r13}.

\begin{theorem}\label{thm:main13}
The locus of curves $[X]\in \cM_{13}$ carrying a line bundle $L\in W^5_{16}(X)$ such that
the multiplication map
$\phi_L\colon \mathrm{Sym}^2 H^0(X,L)\rightarrow H^0(X,L^{\otimes 2})$ is not surjective
is a proper subvariety of $\cM_{13}$, having a divisorial part $\mathfrak{D}_{13}$, whose closure in $\mm_{13}$ has slope
\[
s\bigl(\overline{\mathfrak{D}}_{13}\bigr)=\frac{5059}{749}=6.754\ldots <6+\frac{10}{13}.
\]
\end{theorem}

\vskip 4pt

\noindent The proof of Theorem \ref{thm:main13}  takes full advantage of the techniques we developed in \cite{FJP} in the course of our work on $\mm_{22}$ and $\mm_{23}$. To that end, we split Theorem \ref{thm:main13} in two parts.

\vskip 3pt

Recall that a curve is \emph{tree-like} if its dual graph becomes a tree after deleting all loop edges \cite[p.~364]{EH2}.  We consider a proper moduli stack of generalized limit linear series $\sigma\colon \widetilde{\mathfrak{G}}^5_{16}\rightarrow \widetilde{\mathfrak{M}}_{13}$, where $\widetilde{\fM}_{13}$ is a suitable moduli stack of tree-like curves of genus $13$ equal to $\mathfrak{M}_{13}\cup \Delta_0\cup \Delta_1$ in codimension one (see \S \ref{virtualis_divizorok} for a precise definition). We then construct a morphism of vector bundles over $\widetilde{\mathfrak{G}}^5_{16}$ globalizing the multiplication maps $\phi_L$ considered before. The degeneracy locus $\fU$ of this morphism, due to its determinantal nature, carries a virtual class $[\fU]^{\vir}$ of codimension $2$ inside $\widetilde{\mathfrak{G}}^5_{16}$. Set
\[
[\widetilde{\mathfrak{D}}_{13}]^{\vir}:=\sigma_*\bigl([\fU]^{\vir})\in CH^1(\pm_{13}).
\]

\begin{theorem}\label{rho1virtual}
The following relation for the virtual class  $[\widetilde{\mathfrak{D}}_{13}]^{\vir}$ holds:
\[
[\widetilde{\mathfrak{D}}_{13}]^{\vir}=3\bigl(5059\ \lambda-749\ \delta_0-3929\ \delta_1\bigr)\in CH^1(\pm_{13}).
\]
\end{theorem}

That the degeneracy locus $\fU$ does not map onto $\cM_{13}$ is a particular case of the \emph{Strong Maximal Rank Conjecture} of \cite{AF}. We prove this case, along with a stronger result that guarantees that the virtual class $[\widetilde{\mathfrak{D}}_{13}]^{\vir}$ is effective, using tropical geometry. In particular, we use the method of tropical independence on chains of loops, as introduced in \cite{JP1, JP2}.  Our construction of the required tropical independences is similar to the one used in our proof that $\mm_{22}$ and $\mm_{23}$ are of general type, with one important innovation. In \cite{FJP}, we were able to ignore certain loops called lingering loops.  Here, this seems  impossible; there are too few non-lingering loops. This difficulty shows up already in the simplest combinatorial case, which we call the vertex avoiding case; for a discussion of how we resolve this difficulty, see Remarks~\ref{rem:differences} and \ref{rem:newidea}.

\begin{theorem}\label{thm:strongmrc}
For a general curve $[X]\in \cM_{13}$ the map $\phi_L\colon \mathrm{Sym}^2 H^0(X,L)\rightarrow H^0(X, L^{\otimes 2})$ is surjective for all $L\in W^5_{16}(X)$. Furthermore, there is no component of the degeneracy locus $\fU$ mapping with positive dimensional fibers onto a divisor in $\pm_{13}$.
\end{theorem}

\noindent  Theorem \ref{thm:strongmrc} implies that $\widetilde{\mathfrak{D}}_{13}$, defined as the divisorial part of $\sigma(\fU)$,  represents the class  $[\widetilde{\mathfrak{D}}_{13}]^{\vir}$. Together with Theorem \ref{rho1virtual}, this completes the proof of Theorem \ref{thm:main13}.

The existence of effective divisors of exceptionally small slope on $\mm_{13}$ has direct applications  to the birational geometry of the moduli space $\mm_{13,n}$ of $n$-pointed stable curves of genus $13$.

\begin{theorem}\label{m13n}
The moduli space $\mm_{13,n}$  is of general type for $n\geq 9$.
\end{theorem}

\noindent This improves on Logan's result that $\mm_{13,n}$ is of general type for $n \geq 11$ \cite{Log}. It is known that $\mm_{13,n}$ is uniruled for $n\leq 4$; see \cite{AB} and references therein.

\vskip 4pt

\subsection{The divisor $\mathfrak{D}_{13}$ and rank two Brill-Noether loci.}

\vskip 4pt

The link between Theorems \ref{thm:mp} and \ref{thm:main13} involves a reinterpretation of the divisor $\mathfrak{D}_{13}$ in terms of rank $2$ Brill-Noether theory. Let $\mathcal{SU}_{13}(2, \omega, 8)$ denote the moduli space of pairs $[X,E]$, where $[X]\in \cM_{13}$ and $E\in SU_X(2, \omega, 8)$. Consider the forgetful map
\[
\vartheta\colon \mathcal{SU}_{13}(2, \omega, 8)\rightarrow \cM_{13}, \ \ \  [X, E]\mapsto [X].
\]
We will show that $\vartheta$ is a generically  finite map of degree $3$ (Theorem \ref{thm:3bundles}) and that $\mathcal{SU}_{13}(2, \omega, 8)$ is unirational
(Corollary \ref{cor:unirational}). The fact that $\mm_{13}$ possesses a modular cover $\vartheta$ of such small degree is surprising; we do not know of parallels for other moduli spaces $\mm_g$.

\vskip 4pt

We now fix a pair $[X,E]\in \mathcal{SU}_{13}(2, \omega, 8)$ and consider the determinant map
\[
d\colon \bigwedge ^2 H^0(X,E)\rightarrow H^0(X,\omega_X).
\]

It turns out that for a general $[X,E]$ as above, $E$ is globally generated and the map $d$ is surjective.  In particular, $\PP\bigl(\mbox{Ker}(d)\bigr)\subseteq \PP\bigl(\bigwedge^2 H^0(X,E)\bigr)\cong \PP^{27}$ is a $14$-dimensional linear space. Since $h^0(X,\omega_X)=2h^0(X,E)-3$, it follows that the set of pairs $[X,E]$ satisfying the condition
\begin{equation}\label{eq:grassint}
\PP\bigl(\mathrm{Ker}(d)\bigr)\cap G\bigl(2, H^0(X,E)\bigr)\neq \emptyset
\end{equation}
(the intersection being taken inside $\PP\bigl(\bigwedge^2 H^0(X,E)\bigr)$) is expected to be a divisor on $\mathcal{SU}_{13}(2, \omega, 8)$, and its image under projection by the generically finite map $\vartheta$ is expected to  be also a divisor on $\cM_{13}$. We refer to this locus as the \emph{resonance divisor} $\mathfrak{Res}_{13}$, inspired by the algebraic  definition of the \emph{resonance variety}, see \cite[Definition 2.4]{AFPRW} and references therein.

\begin{theorem}\label{thm:bn13}
The closure of the resonance divisor in $\cM_{13}$
\[
\mathfrak{Res}_{13}:=\Big\{[X]\in \cM_{13}: \exists E\in \mathcal{SU}_X(2, \omega, 8) \mbox{ with } \mathbb P\bigl(\mathrm{Ker}(d)\bigr)\cap G\bigl(2, H^0(X,E)\bigr)\neq \emptyset\Bigr\}
\]
is an effective divisor in $\mm_{13}$. One has the following equality of divisors on $\mm_{13}$
\[
\overline{\mathfrak{Res}}_{13}=\overline{\mathfrak{D}}_{13}+3\cdot \mm_{13,7}^1.
\]
\end{theorem}

Here, we recall that $\mm_{13,7}^1$ is the Hurwitz divisor of heptagonal curves on $\mm_{13}$ whose class is computed in \cite{HM}. The set-theoretic inclusion $\cM_{13,7}^1\subseteq \mathfrak{Res}_{13}$ is relatively straightforward. The multiplicity $3$ with which $\mm_{13,7}^1$ appears in $\mathfrak{Res}_{13}$ is explained by an excess intersection calculation carried out in \S \ref{sec:nonab} and confirms once more that the degree of the map $\vartheta\colon \mathcal{SU}_{13}(2, \omega, 8)\rightarrow \cM_{13}$ is $3$.

\vskip 4pt

We conclude this introduction by explaining the connection between the resonance divisor $\mathfrak{Res}_{13}$ and Theorems \ref{thm:mp} and \ref{thm:main13}. On the one hand, using \cite{FR} the class $[\widetilde{\mathfrak{Res}}_{13}]$ of the closure of $\mathfrak{Res}_{13}$ in $\widetilde{\cM}_{13}$ can be computed in terms of the generators of $CH^1(\widetilde{\cM}_{13})$ and a tautological class $\vartheta_*(\gamma)$, where $\gamma$ is the push-forward of the second Chern class of the (normalized) universal rank $2$ vector bundle on the universal curve over a suitable compactification of $\mathcal{SU}_{13}(2, \omega, 8)$; see Definition \ref{def:tautclass} for details. On the other hand, Theorem \ref{thm:bn13} yields an explicit description of $\widetilde{\mathfrak{Res}}_{13}$. By combining this description  with Theorem \ref{thm:main13}, we obtain a \emph{second} calculation for the class $[\widetilde{\mathfrak{Res}}_{13}]$. In this way, we indirectly determine the tautological class $\vartheta_*(\gamma)$; see Proposition \ref{prop:gamma}. Once the class of $[\widetilde{\mathfrak{Res}}_{13}]$ is known, the calculation of the class of the non-abelian Brill-Noether divisor $[\widetilde{\mathcal{M}\mathcal{P}}_{13}]$ (Theorem \ref{thm:mp}) and that of the universal Theta divisor $[\widetilde{\Theta}_{13}]$ on $\rr_{13}$ (Theorems \ref{thm:r13} and \ref{thm:univtheta13}) follow from Grothendieck-Riemann-Roch calculations, after checking suitable transversality assumptions.

\vskip 5pt

{\small{\noindent {{\bf{Acknowledgments:}}}
We had interesting discussions with P. Newstead and A. Verra related to this circle of ideas. Farkas was partially supported by the DFG Grant \emph{Syzygien und Moduli} and by the ERC Advanced Grant SYZYGY.  Jensen was partially supported by NSF grant DMS--2054135.  Payne was partially supported by NSF grants DMS--2001502 and DMS--2053261. This project has received funding from the European Research Council (ERC) under the European Union Horizon 2020 research and innovation program (grant agreement No. 834172).}}

\section{The failure locus of the Strong Maximal Rank Conjecture on $\pm_{13}$} \label{virtualis_divizorok}

We denote by $\MM_g$ the moduli stack of stable curves of genus $g\geq 2$ and by $\mm_g$ the associated coarse moduli space. We work throughout over an algebraically closed field $K$ of characteristic $0$ and the Chow groups that we consider are with rational coefficients. Via the isomorphism $CH^*(\MM_g)\cong CH^*(\mm_g)$, we routinely identify cycle classes on $\MM_g$ with their push forward to $\mm_g$.  Recall that for $g\geq 3$ the group $CH^1(\mm_g)$  is freely generated by the Hodge class $\lambda$ and by the classes of the boundary divisors $\delta_i=[\Delta_i]$, for $i=0,  \ldots, \lfloor \frac{g}{2} \rfloor$.

\vskip 3pt

In this section, we realize the virtual divisor class $[\widetilde {\mathfrak{D}}_{13}]^\vir$ as the push forward of the virtual class of a codimension $2$ determinantal locus inside the moduli space $\widetilde{\mathfrak{G}}^5_{16}$ of limit linear series of type $\mathfrak g^5_{16}$ over an open substack $\widetilde{\mathfrak{M}}_{13}$  of $\overline \fM_{13}$, that agrees with $\fM_{13}\cup \Delta_0\cup \Delta_1$ outside a subset of codimension $2$. We will use standard terminology from the theory of limit linear series \cite{EH1}, and begin by recalling a few of the basics.

\begin{definition}\label{def:bnnumber}
Let  $X$ be a smooth curve of genus $g$ with $\ell = (L, V) \in G^r_d(X)$ a linear series. The \emph{ramification sequence} of $\ell$ at a point $q \in X$ is denoted
\[
\alpha^{\ell}(q) : \alpha_0^{\ell}(q) \leq \cdots \leq \alpha_r^{\ell}(q).
\] This is obtained from the \emph{vanishing sequence} $a^{\ell}(q) : a_0^{\ell}(q) < \cdots < a_r^{\ell}(q) \leq d$ of $\ell$ at $q$,
by setting $\alpha^{\ell}_i(q) := a^{\ell}_i(q) - i$, for $i=0, \ldots, r$.   The \emph{ramification weight} of $q$ with respect to $\ell$ is $\mathrm{wt}^{\ell}(q) := \sum_{i=0}^r \alpha^{\ell}_i(q)$.  We define $\rho(\ell, q) := \rho(g,r,d) - \mathrm{wt}^{\ell}(q)$.
\end{definition}

A \emph{generalized limit linear series} on a tree-like curve $X$ of genus $g$  consists of a collection
\[
\ell=\{(L_C, V_C) : C\mbox{ is a component of } X\},
\]
 where $L_C$ is a rank $1$ torsion free sheaf of degree $d$ on $C$ and $V_C\subseteq H^0(C,L_C)$ is an $(r+1)$-dimensional space of sections satisfying compatibility conditions on the vanishing sequences at the nodes of $X$, see \cite[p.~364]{EH2}. Let $\overline{G}^r_d(X)$ be the variety of generalized limit linear series of type $\mathfrak{g}^r_d$ on $X$.

\vskip 4pt

In this section we set
\begin{equation}\label{sit1}
g=13, \ r=5, \ d=16.
\end{equation}
Although we are mainly interested in the case $g=13$, some of the constructions are set up for an arbitrary genus $g$, making it easier to refer to results from \cite{FJP}.

\vskip 4pt

We denote by $\cM_{13, 15}^5$ the subvariety of $\cM_{13}$
parametrizing curves $X$  such that $W^5_{15}(X)\neq \emptyset$. As explained in \cite[\S 3]{FJP}, we have $\mathrm{codim}(\cM_{13, 5}^5, \cM_{13})\geq 2$.
\vskip 4pt

Let $\Delta_1^{\circ}\subseteq \Delta_1\subseteq \mm_{g}$ be the locus of curves
$[X\cup_y E]$, where $X$ is a smooth curve of genus $g-1$ and $[E,y]\in \mm_{1,1}$ is an arbitrary elliptic curve.  The point of attachment $y \in X$ is chosen arbitrarily.  Furthermore, let  $\Delta_0^{\circ} \subseteq \Delta_0\subseteq \mm_{g}$ be the locus of curves $[X_{yq}:=X/y\sim q]\in \Delta_0$, where $[X, q]$  is a smooth curve of genus $g-1$ and $y \in X$ is an arbitrary point, together with their degenerations $[X \cup_q E_{\infty}]$, where
$E_{\infty}$ is a rational nodal curve (that is, $E_{\infty}$ is a nodal elliptic curve and $j(E_{\infty})=\infty$). Points of this form comprise the intersection
$\Delta_0^{\circ}\cap \Delta_1^{\circ}$. We define the following open subset of $\mm_g$:
\[
\mm_g^{\circ}:=\cM_g\cup \Delta_0^{\circ}\cup \Delta_1^{\circ}.
\]

\vskip 3pt

Along the lines of \cite[\S 3]{FJP}, we introduce an even smaller open subspace of $\mm_g$ over which the calculation of  $[\widetilde{\mathfrak{D}}_{13}]^{\vir}$ can be completed. Let  $\mathcal{T}_0 \subset \Delta_0^{\circ}$ be the locus of curves $[X_{yq}:=X/y\sim q]$ where either $\overline{G}^{r+1}_d(X)\neq \emptyset$ or $\overline{G}^{r}_{d-2}(X)\neq \emptyset$.  Similarly, let $\mathcal{T}_1\subseteq  \Delta_1^{\circ}$ be the locus of curves $[X\cup_y E]$, where $X$ is a smooth curve of genus $g-1$ such that $G^{r+1}_d(X)\neq \emptyset$ or $G^r_{d-2}(X)\neq \emptyset$.
We set
\begin{equation*}
\pm_{g}:=\mm_g^{\circ} \smallsetminus \Bigl( \mm_{g,d-1}^r   \cup \mathcal{T}_0 \cup \mathcal{T}_1 \Bigr).
\end{equation*}
We define $\widetilde{\Delta}_0:=\pm_g\cap \Delta_0\subseteq \Delta_0^{\circ}$ and $\widetilde{\Delta}_1:=\pm_g\cap \Delta_1\subseteq \Delta_1^{\circ}$.
Note that $\pm_g$ and $\cM_g\cup \Delta_0\cup \Delta_1$ agree away from a set of codimension $2$ in each. We identify  $CH^1(\pm_{g}) \cong \mathbb Q \langle \lambda, \delta_0, \delta_1 \rangle$, where $\lambda$ is the Hodge class, $\delta_0:=[\widetilde{\Delta}_0]$ and $\delta_1:=[\widetilde{\Delta}_1]$.

\subsection{Stacks of limit linear series.}

Let $\widetilde{\mathfrak{G}}^r_{d}$ be the stack of pairs $[X, \ell]$, where $[X]\in \widetilde{\mathcal{M}}_{g}$ and $\ell$ is a (generalized) limit linear series of type $g^r_d$ on the tree-like curve $X$. We consider the proper projection
\[
\sigma\colon \widetilde{\mathfrak{G}}_{d}^r  \rightarrow  \widetilde{\mathfrak{M}}_g.
\]

Over a curve $[X\cup_y E]\in \widetilde{\Delta}_1$, we identify $\sigma^{-1}([X\cup_y E])$ with the variety of (generalized) limit linear series $\ell = (\ell_X, \ell_E) \in \overline{G}^r_d(X\cup _y E)$.    The fibre $\sigma^{-1}([X_{yq}])$ over  an irreducible curve $[X_{yq}] \in \widetilde{\Delta}_0\smallsetminus \widetilde{\Delta}_1$, is canonically identified with the variety $\overline{W}^r_d(X_{yq})$ of rank $1$ torsion free sheaves $L$ on $X_{yq}$ having degree $d(L) = d$ and $h^0(X_{yq},L) \geq r+1$.

\vskip 4pt

Let $\widetilde{\mathfrak{C}}_{g} \rightarrow \widetilde{\mathfrak{M}}_{g}$ be the universal curve, and let $p_2\colon \widetilde{\mathfrak{C}}_{g} \times_{\widetilde{\mathfrak{M}}_{g}} \widetilde{\mathfrak{G}}^r_{d} \rightarrow \widetilde{\mathfrak{G}}^r_{d}$ be the projection map.  We denote by $\mathfrak{Z}\subseteq \widetilde{\mathfrak{C}}_{g} \times_{\widetilde{\mathfrak{M}}_{g}} \widetilde{\mathfrak{G}}^r_{d}$ the codimension $2$ substack consisting
of pairs $[X_{yq}, L, z]$, where $[X_{yq}]\in \Delta_0^{\circ}$, the point $z$ is the node of $X_{yq}$  and $L\in \overline{W}^r_d(X_{yq})\smallsetminus W^r_d(X_{yq})$ is a \emph{non-locally free} torsion free sheaf.  Let
\[
\epsilon \colon \widehat{\mathfrak{C}}_g:=\mathrm{Bl}_{\mathfrak{Z}}\Bigl(\widetilde{\mathfrak{C}}_{g} \times_{\widetilde{\mathfrak{M}}_{g}} \widetilde{\mathfrak{G}}^r_{d}\Bigr)\rightarrow \widetilde{\mathfrak{C}}_{g} \times_{\widetilde{\mathfrak{M}}_{g}} \widetilde{\mathfrak{G}}^r_{d}
\]
be the blow-up of this locus, and we denote the induced universal curve by
\[
\wp:=p_2\circ \epsilon: \widehat{\mathfrak{C}}_g\rightarrow \widetilde{\mathfrak{G}}^r_{d}.
\]

The fibre of $\wp$ over a point $[X_{yq}, L] \in \widetilde{\Delta}_0$, where $L\in \overline{W}^r_d(X_{yq})\smallsetminus W^r_d(X_{yq})$, is the semistable curve $X\cup_{\{y,q\}} R$ of genus $g$, where $R$ is a smooth rational curve meeting $X$ transversally at $y$ and $q$.

\subsection{A degeneracy locus inside $\widetilde{\mathfrak G}^5_{16}$.}

In order to define the degeneracy locus on $\widetilde{\mathfrak{G}}^5_{16}$  whose push-forward produces $[\widetilde{\mathfrak{D}}_{13}]^\vir$, we first choose a Poincar\'e line bundle $\L$ over the universal curve $\widehat{\mathfrak{C}}_g$ with the following properties:
\begin{enumerate}
\item If $[X\cup_y E]\in \widetilde{\Delta}_1$ and $\ell = (\ell_X, \ell_E) \in \overline{G}^r_d(X\cup E)$ is a limit linear series, then
\[
\L_{| [X\cup_y E, \ell]} \in \Pic^{d}(X) \times \Pic^0(E).
\]

\item For a point $t=[X_{yq}, L]$, where $[X_{yq}] \in \widetilde{\Delta}_0$ and $L \in \overline{W}^r_d(X_{yq}) \smallsetminus W^r_d(X_{yq})$, thus $L=\nu_*(A)$ for some $A\in W^r_{d-1}(X)$, we have $\mathcal{L}_{|X}\cong A$ and $\mathcal{L}_{|R}\cong \OO_R(1)$.  Here, $\wp^{-1}(t) = X \cup R$, whereas $\nu\colon X\rightarrow X_{yq}$ is the normalization map.
\end{enumerate}

We now introduce the following two sheaves over $\widetilde{\mathfrak{G}}^r_d$
\begin{equation*}
\E:=\wp_*(\mathcal{L})  \ \mbox{ and }
\F:=\wp_*(\mathcal{L}^{\otimes 2}).
\end{equation*}

\noindent Both $\E$ and $\F$ are locally free; the proof by local analysis in \cite[Proposition 3.6]{FJP} goes through essentially without change.
\vskip 4pt

There is a sheaf morphism over $\widetilde{\mathfrak G}^5_{16}$ globalizing the multiplication of sections
\begin{equation}\label{morphism1}
\phi\colon \mbox{Sym}^{2}(\E)\rightarrow \F.
\end{equation}
We denote by  $\fU \subseteq  \widetilde{\mathfrak{G}}^5_{16}$ the locus where $\phi$ is not surjective (equivalently, where $\phi^{\vee}$ is not injective). Due to its determinantal nature, $\fU$ carries a virtual class in the expected codimension $2$.

\begin{definition}\label{def:virtclass}
We define the virtual divisor class  $[\widetilde{\mathfrak{D}}_{13}]^{\mathrm{virt}}:=\sigma_*([\fU]^\vir)$ as
\[
[\widetilde{\mathfrak{D}}_{13}]^{\mathrm{virt}}:=\sigma_*\Bigl(c_2(\mathrm{Sym}^2(\E)^{\vee}-\F^{\vee})\Bigr)\in CH^1(\widetilde{\mathfrak{M}}_{13}).
\]
\end{definition}

\noindent If $\fU$ has pure codimension $2$, then $\widetilde{\mathfrak{D}}_{13}$ is a divisor on $\pm_{13}$ and $[\widetilde{\mathfrak{D}}_{13}]^{\vir}=[\widetilde{\mathfrak{D}}_{13}]$.  The following  corollary provides a local description of the morphism $\phi$.

\begin{corollary}
The morphism $\phi\colon \mathrm{Sym}^2(\E)\rightarrow \F$ has the following description on fibers:

\noindent $(i)$  For $[X, L]\in \widetilde{\mathfrak{G}}^r_{d}$, with $[X]\in \cM_g\smallsetminus \cM_{g, d-1}^r$ smooth, the fibers are
\[
\E_{(X, L)}=H^0(X, L)\ \mbox{ and } \ \F_{(X, L)}=H^0(X, L^{\otimes 2}),
\]
and $\phi_{(X,L)}: \mathrm{Sym}^2 H^0(X, L)\rightarrow H^0(X, L^{\otimes 2})$ is the usual multiplication map of global sections.

\vskip 5pt

\noindent $(ii)$
Suppose $t=(X\cup_y E, \ell_X, \ell_E)\in \sigma^{-1}(\widetilde{\Delta}_1)$,
where $X$ is a curve of genus $g-1$, $E$ is an elliptic curve  and $\ell_X=|L_X|$ is the $X$-aspect of the corresponding limit linear series with $L_X\in
W^r_{d}(X)$ such that $h^0(X,L_X(-2y))\geq r$. If $L_X$ has no base point at $y$, then
\[
\E_t=H^0(X, L_X)\cong H^0\bigl(X,L_X(-2y)\bigr)\oplus K \cdot u \ \mbox{ and
} \ \F_t=H^0\bigl(X, L_X^{\otimes 2}(-2y)\bigr)\oplus  K \cdot u^2,
\]
where $u\in H^0(X, L_X)$ is any section such that $\mathrm{ord}_y(u)=0$.

\vskip 5pt

If $L_X$ has a base point at $y$, then $\E_t=H^0(X, L_X)\cong H^0(X, L_X(-y))$ and the image of $\F_t \rightarrow H^0(X, L_X^{\otimes 2})$ is the subspace $H^0\bigl(X, L_X^{\otimes 2}(-2y)\bigr)\subseteq H^0(X,L_X^{\otimes 2})$.

\vskip 4pt

\noindent $(iii)$  Let $t=[X_{yq}, L] \in \sigma^{-1}(\widetilde{\Delta}_0)$ be a point with $q,y \in X$ and let $L\in W^r_{d}(X_{yq})$ be a locally free sheaf of rank $1$, such that $h^0(X, \nu^*L(-y-q))\geq r$, where $\nu\colon X\rightarrow X_{yq}$ is the
normalization. Then one has the following description of the fibers:
\[
\E_t=H^0(X, \nu^*L)\ \mbox{ and }\ \F_t=H^0\bigl(X, \nu^*L^{\otimes
2}(-y-q)\bigr)\oplus  K \cdot u^2,
\]
where $u\in H^0(X, \nu^*L)$ is any section not vanishing at both points $y$ and $q$.

\vskip 5pt

\noindent $(iv)$ Let $t = [X_{yq}, \nu_*(A)]$, where $A\in W^r_{d-1}(X)$ and set again $X\cup_{\{y,q\}} R$ to be the fibre $\wp^{-1}(t)$.
Then $\E_t=H^0(X\cup R, \mathcal{L}_{X\cup R})\cong H^0(X,A)$ and $\F_t=H^0(X\cup R, \mathcal{L}^{\otimes 2}_{X\cup R})$. Furthermore,
$\phi(t)$ is the multiplication map on $X\cup R$.
\end{corollary}

\begin{proof}
The proof is essentially identical to the proof of \cite[Corollary 3.8]{FJP}; we omit the details.
\end{proof}

\subsection{Test curves in $\pm_{13}$.}

As in \cite{FJP}, the calculation of $[\widetilde{\mathfrak{D}}_{13}]^{\mathrm{virt}}$ is carried out by understanding the restriction of the morphism $\phi$ along the pull backs of the three standard test curves $F_0$, $F_{\mathrm{ell}}$, and $F_1$ inside $\pm_{13}$.  Let $[X, q]$ be a general pointed curve of genus $g-1$ and fix an elliptic curve $[E, y]$.  We then define
\[
F_0:=\Bigl\{X_{yq}:=X/y\sim q: y\in X\Bigr\}\subseteq \Delta_0^{\circ}\subseteq \ttem_{g} \   \mbox{ and  } \  F_1:=\Bigl\{X\cup_y E:y \in X\Bigr\}\subseteq \Delta_1^{\circ} \subseteq \ttem_{g}.
\]
Furthermore, we define the curve
\begin{equation}\label{fell}
F_{\mathrm{ell}}:=\Bigl\{[X\cup_q E_t]:t\in \PP^1\Bigr\}\subseteq \Delta_1 \subseteq \mm_{g},
\end{equation}
where $\{[E_t, q]\}_{t\in \PP^1}$ denotes a pencil of plane cubics and $q$ is a fixed point of the pencil.
We record the intersection of these test curves with the generators of $CH^1(\mm_{g})$:
\begin{align*}
F_0\cdot \lambda &= 0, &  F_0\cdot \delta_0 &= 2-2g, &  F_0 \cdot \delta_1 &= 1 \
\mbox{ and } & F_0 \cdot \delta_j &= 0 \ \mbox{ for } \ j=2, \ldots, \Big \lfloor \frac{g}{2}\Big \rfloor, \\
F_{\mathrm{ell}}\cdot \lambda &= 1, &  F_{\mathrm{ell}} \cdot \delta_0 &=12, & F_{\mathrm{ell}}\cdot \delta_1 &=-1 \ \mbox{ and } & F_{\mathrm{ell}} \cdot \delta_j &= 0 \ \mbox{ for } j=2, \ldots, \Big \lfloor \frac{g}{2} \Big \rfloor.
\end{align*}
Note also that $F_1\cdot \lambda=0, \ F_1\cdot \delta_i=4-2g, \ \mbox{ and } F_1\cdot \delta_j=0\ \mbox{ for } j\neq 1.$

\vskip 4pt

We now describe the pull back $\sigma^*(F_0)\subseteq \widetilde{\mathfrak G}^5_{16}$. Having fixed a general pointed curve $[X,q]\in \mm_{12,1}$, we introduce the variety
\begin{equation}\label{defx}
Y:=\Bigl\{(y, L)\in X\times W^5_{16}(X): h^0(X, L(-y-q))\geq 5\Bigr\},
\end{equation}
together with the projection $\pi_1\colon Y \rightarrow X$. Arguing in a way similar to \cite[Proposition 3.10]{FJP}, we conclude that $Y$ has pure dimension $2$, that is, its actual dimension equals its expected dimension as a degeneracy locus. We consider two curves inside $Y$, namely
\begin{align*}
\Gamma_1 &:= \Bigl\{(y, A(y)): y\in X, \ A\in W^5_{15}(X)\Bigr\} \ \mbox{  and } \\
\Gamma_2 &:= \Bigl\{(y, A(q)): y\in X, \ A\in W^5_{15}(X)\Bigr\},
\end{align*}
intesecting transversely along finitely many points. We then introduce the variety $\widetilde{Y}$ obtained from $Y$ by identifying for each $(y, A)\in X\times W^5_{15}(X)$, the points $(y, A(y))\in \Gamma_1$ and
$(y, A(q))\in \Gamma_2$. Let $\vartheta\colon Y\rightarrow \widetilde{Y}$ the projection map.

\begin{proposition}\label{limitlin0}
With notation as above, there is a birational morphism
\[
f\colon  \sigma^*(F_0)\rightarrow \widetilde{Y},
\]
which is an isomorphism outside $\vartheta (\pi_1^{-1}(q))$. The restriction of $f$ to $f^{-1}\bigl(\vartheta(\pi_1^{-1}(q))\bigr)$  forgets the aspect of each limit linear series on the elliptic curve $E_{\infty}$.  Furthermore, both $\E_{| \sigma^*(F_0)}$ and $\F_{| \sigma^*(F_0)}$ are pull backs under $f$ of vector bundles on $\widetilde{Y}$.
\end{proposition}

\begin{proof}
 The proof  is identical to that of \cite[Proposition 3.11]{FJP}.
\end{proof}


We now describe the pull back $\sigma^*(F_1)\subseteq \widetilde{\mathfrak{G}}^5_{16}$ and we  define the determinantal variety
\begin{equation}\label{defz}
Z := \Bigl\{(y, L)\in X\times W^5_{16}(X): h^0(X, L(-2y))\geq 5\Bigr\} .
\end{equation}
Because $X$ is general, arguing precisely like in \cite[Proposition 3.10]{FJP}, we find that  $Z$ is pure of dimension $2$. Next we observe that in order to estimate the intersection of $[\widetilde{\mathfrak{D}}_{13}]^{\mathrm{virt}}$ with the surface $\sigma^*(F_1)$ it suffices to restrict ourselves to $Z$:

\begin{proposition}\label{limitlin1}
The variety $Z$ is an irreducible component of $\sigma^*(F_1)$ and
\begin{align*}
c_2\bigl(\mathrm{Sym}^2(\E)^{\vee}-\F^{\vee}\bigr)_{| \sigma^*(F_1)} &= c_2\bigl(\mathrm{Sym}^2(\E)^{\vee} -\F^{\vee}\bigr)_{| Z}.
\end{align*}
\end{proposition}

\begin{proof}
Let $(\ell_X, \ell_E)\in \sigma^{-1}([X\cup_y E])$ be a limit linear series.  Observe that $\rho(13, 5, 16) = 1$, which is greater than or equal to the sum of the adjusted Brill-Noether numbers $\rho(\ell_X, y)+\rho(\ell_E, y)$ (see Definition~\ref{def:bnnumber}). Since $\rho(\ell_E, y)\geq 0$, it follows that $\rho(\ell_X, y)\in \{0,1\}$.  If $\rho(\ell_E, y)=0$, then $\ell_E=10y+|\OO_E(6y)|$ and the aspect $\ell_X\in G^5_{16}(X)$ is a complete linear series with a cusp at the point $y\in X$.
Therefore $(y,\ell_X)\in Z$, and in particular $Z\times \{\ell_E\}\cong Z$ is a union of irreducible components of $\sigma^*(F_1)$.

\vskip 3pt

The remaining components of $\sigma^*(F_1)$ are indexed by Schubert
indices
\[
\alpha:=\bigl(0\leq \alpha_0\leq \ldots \leq \alpha_5\leq 11=16-5\bigr),
\]
 such that $\alpha\geq (0, 1, 1, 1, 1, 1)$ holds lexicographically, and $\alpha_0+\cdots+\alpha_5\in \{6,7\}$, for $\rho(\ell_X,y)\geq -1$, for any point $y\in X$, see also \cite[Theorem 0.1]{Fa}.   For a Schubert index $\alpha$ satisfying these conditions, we let $\alpha^c:=(11-\alpha_5,\ldots, 11-\alpha_0)$  be the complementary Schubert index, and define
\[
Z_{\alpha}:=\bigl\{(y, \ell_X)\in X\times G^5_{16}(X): \alpha^{\ell_X}(y)\geq \alpha\bigr\}\ \mbox{ and }W_{\alpha}:=\bigl\{\ell_E\in G^5_{16}(E): \alpha^{\ell_E}(y)\geq \alpha^c\bigr\}.
\]
Then the following relation holds for certain natural coefficients $m_{\alpha}$:
\[
\sigma^*(F_1)=Z+\sum_{\alpha\geq (0,1,1,1,1,1)} m_{\alpha}\  \bigl(Z_{\alpha}\times W_{\alpha}\bigr).
\]
We now finish the proof by invoking the pointed Brill-Noether Theorem \cite[Theorem~1.1]{EH2}, which gives  $\mbox{dim } Z_{\alpha}=1+\rho(12, 5, 16)-(\alpha_0+\cdots+\alpha_5)\leq 1$. In the definition of the test curve $F_1$, the point of attachment $y\in E$ is fixed, therefore the restrictions of both $\E$ and $\F$ are pulled-back from $Z_{\alpha}$ and one obtains  $c_2\bigl(\mbox{Sym}^2(\E)^\vee-\F^{\vee}\bigr)_{|Z_{\alpha}\times W_{\alpha}}=0$ for dimension reasons.
\end{proof}

\subsection{Top Chern numbers on Jacobians.}
We use various facts about intersection theory on Jacobians, for which we refer to \cite[Chapters VII--VIII]{ACGH}.  We start with a general curve $X$ of genus $g$, fix a Poincar\'e line bundle $\P$ on $X\times \mbox{Pic}^d(X)$ and denote by
\[
\pi_1\colon X\times \mbox{Pic}^d(X)\rightarrow X \mbox{ and } \pi_2\colon X\times \mbox{Pic}^d(X)\rightarrow \mbox{Pic}^d(X)
\]
the two projections.  Let $\eta=\pi_1^*([x_0])\in H^2(X \times \mbox{Pic}^d(X), \mathbb Z)$, where $x_0\in X$ is a fixed point. We choose a symplectic basis $\delta_1,\ldots, \delta_{2g}\in H^1(X, \mathbb Z)\cong H^1(\mbox{Pic}^d(X), \mathbb Z)$, and then consider the class
\[
\gamma:=-\sum_{\alpha=1}^g \Bigl( \pi_1^*(\delta_{\alpha}) \pi_2^*(\delta_{g+\alpha}) - \pi_1^*(\delta_{g+\alpha})\pi_2^*(\delta_{\alpha}) \Bigr) \in H^2(X \times \mathrm{Pic}^d(X), \mathbb{Z} ).
\]
One has $c_1(\P)=d\cdot\eta+\gamma$, and the relations $\gamma^3=0$, $\gamma \eta=0$,  $\eta^2=0$, and $\gamma^2 = -2\eta \pi_2^*(\theta)$, for which we refer to \cite[page 335]{ACGH}.  Assuming $W^{r+1}_d(X)=\emptyset$  (which is what happens in the case  $g=12, r=5, d=16$ relevant to us), the smooth variety $W^r_d(X)$ admits a rank $r+1$ vector bundle
\[
\mathcal{M}:=(\pi_2)_{*}\Bigl(\mathcal{P}_{| X\times W^r_d(X)}\Bigr)
\]
with fibers $\cM_L\cong H^0(X,L)$, for $L\in W^r_d(X)$.  The Chern numbers of $\mathcal{M}$ are computed via the Harris-Tu
formula \cite{HT}. We write formally
\[
\sum_{i=0}^r c_i(\mathcal{M}^{\vee})=(1+x_1)\cdots (1+x_{r+1}).
\]
For a class $\zeta \in H^*\bigl(\mbox{Pic}^d(X), \mathbb{Z} \bigr)$, the Chern number $c_{j_1}(\mathcal{M}) \cdots c_{j_s}(\mathcal{M})\ \cdot \zeta \in H^{\mathrm{top}}(W^r_d(X), \mathbb{Z} )$ can be computed by using repeatedly the following formal identities\footnote{See \cite[\S~4.1]{FJP} for a detailed discussion of how to read and apply the Harris-Tu formula in this context.}:

\begin{equation}\label{harristu2}
x_1^{i_1}\cdots x_{r+1}^{i_{r+1}}\cdot
\theta^{\rho(g,r,d)-i_1-\cdots-i_{r+1}} =g!\ \frac{\prod_{j>k} (i_{k}-i_j+j-k)}{\prod_{k=1}^{r+1} (g-d+2r+i_k-k)!}.
\end{equation}

\vskip 4pt

We now specialize to the case when $X$ is a general curve of genus $12$, thus $W^5_{16}(X)$ is a smooth $6$-fold. By Grauert's Theorem,  $\mathcal{N}:=(R^1\pi_2)_{*}\Bigl(\mathcal{P}_{| X\times W^5_{16}(X)}\Bigr)$ is locally free of rank one. Set $y_1:=c_1(\mathcal{N})$.  We now explain how $y_1$ determine the Chern numbers of $\cM$.

\begin{proposition}
\label{chernosztalyok}
For a general curve $X$ of genus $12$ set  $c_i:=c_i(\cM^{\vee})$, for $i=1, \ldots, 6$, and $y_1:=c_1(\cN)$. Then the following relations hold in $H^*(W^5_{16}(X), \mathbb Z)$:
\begin{align*}
c_i=\frac{\theta^i}{i!}-\frac{\theta^{i-1}}{(i-1)!}y_1, \  \mbox{ for } i=1, \ldots, 6.
\end{align*}
\end{proposition}

\begin{proof}
For an effective divisor $D$ of sufficiently large degree on $X$,  there is an exact sequence
\[
0 \rightarrow \cM \rightarrow (\pi_{2})_*\Bigl( \P \otimes \OO(\pi^*D) \Bigr)\rightarrow (\pi_2)_*\Bigl( \P \otimes \OO(\pi_1^*D)_{| \pi_1^*D} \Bigr) \rightarrow R^1\pi_{2*} \Bigl( \P_{| X\times W^5_{16}(X)} \Bigr) \rightarrow 0.
\]
Recall that $\mathcal{N}$ is the vector bundle on the right in the exact sequence above.  By \cite[Chapter VII]{ACGH}, we have  $c_{\mathrm{tot}}\Bigl((\pi_2)_*(\P \otimes \OO(\pi_1^*D))\Bigr)=e^{-\theta}$, and the total Chern class of $(\pi_2)_* \Bigl( \P \otimes \OO(\pi_1^*D)_{| \pi_1^*D} \Bigr)$ is trivial.  We therefore obtain the formula
\[
(1+y_1) \cdot e^{-\theta} =1-c_1+c_2-\cdots +c_6,
\]
as claimed.
\end{proof}

Using Proposition \ref{chernosztalyok}, any Chern number on $W^5_{16}(X)$ can be expressed in terms of monomials in  $y_1$ and $\theta$. The following identity on $H^{12}(W^5_{16}(X),\mathbb Z)$ follows from (\ref{harristu2}) using the canonical isomorphism $H^1(X, L)\cong H^0\bigl(X, \omega_X\otimes L^{\vee}\bigr)^{\vee}$.
\begin{equation}\label{eq:topproducts}
\Bigl(\theta^i \cdot  y_1^{6-i}\Bigr)_{W^5_{16}(X)}=\frac{\theta^{12}}{(12-i)!}=i!{12\choose i}.
\end{equation}

\vskip 4pt

With this preparation in place, we now compute the classes of the loci $Y$ and $Z$.

\begin{proposition}
\label{xy}
Let $[X,q]$ be a general pointed curve of genus $12$, let $\cM$ denote the tautological rank $6$ vector bundle over $W^5_{16}(X)$, and set $c_i = c_i(\cM^{\vee}) \in H^{2i}(W^5_{16}(X), \mathbb{Z})$ as before.  The following formulas hold:
\begin{enumerate}
\item  $[Z] = \pi_2^*(c_5) - 6\eta \theta \pi_2^*(c_3) + (54\eta + 2\gamma) \pi_2^*(c_4) \in H^{10} \bigl(X \times W^5_{16}(X), \mathbb{Z} \bigr)$.

\item  $[Y] = \pi_2^*(c_5) - 2\eta \theta \pi_2^*(c_3) + (15\eta + \gamma) \pi_2^*(c_4) \in H^{10}\bigl(X \times W^5_{16}(X), \mathbb{Z} \bigr)$.
\end{enumerate}
\end{proposition}

\begin{proof}
The locus $Z$ has been defined by (\ref{defz})  as the degeneracy locus of a vector bundle morphism over the $7$-dimensional smooth variety $X \times W^5_{16}(X)$ (observe again that $W^6_{16}(X)=\emptyset)$. For each $(y,L) \in X \times W^5_{16}(X)$, there is a natural map
\[
H^0(X, L\otimes \OO_{2y})^{\vee} \rightarrow H^0(X, L)^{\vee}.
\]
These maps viewed together induce a  morphism $\zeta \colon J_1(\mathcal{P})^{\vee} \rightarrow \pi_2^*(\cM)^{\vee}$ of vector bundles.  Then $Z$ is the first degeneracy locus of $\zeta$ and applying the Porteous formula,
\[
[Z] = c_5 \Bigl( \pi^*_2 (\cM)^{\vee} - J_1 (\mathcal{P})^{\vee} \Bigr).
\]
The Chern classes of the jet bundle $J_1(\mathcal{P})$ are computed using the standard exact sequence
\[
0\longrightarrow \pi_1^*(\omega_X) \otimes \mathcal{P} \longrightarrow J_1(\mathcal{P}) \longrightarrow \mathcal{P} \longrightarrow 0 .
\]
We compute the total Chern class of the formal inverse of the jet bundle as follows:
\begin{align*}
c_{\mathrm{tot}} \bigl(J_1(\mathcal{P})^{\vee}\bigr)^{-1}&= \Bigl( \sum_{j\geq 0} (d(L) \eta + \gamma)^j \Bigr) \cdot \Bigl( \sum_{j\geq 0} \bigl( (2g(X)-2+d(L)) \eta + \gamma \bigr)^j \Bigr), \\
&=\bigl(1+16\eta + \gamma + \gamma^2 + \cdots \bigr) \cdot \bigl( 1+38\eta + \gamma + \gamma^2 + \cdots \bigr), \\ &= 1 + 54\eta + 2\gamma - 6\eta\theta.
\end{align*}
Multiplying this with the total class of $\pi_2^*(\cM)^{\vee}$, one finds the claimed formula for $[Z]$.

\vskip 4pt

To compute the class of $Y$ defined in (\ref{defx}), we consider the projections
\[
\mu, \nu\colon X \times X\times \mbox{Pic}^{16}(X) \rightarrow X \times \mbox{Pic}^{16}(X),
\]
and let $\Delta \subseteq X \times X \times \mbox{Pic}^{16}(X)$ be the diagonal. Set $\Gamma_q:=\{ q \} \times \mbox{Pic}^{16}(X)$ and consider the vector bundle $\cB:=\mu_*\bigl( \nu^*(\mathcal{P}) \otimes \OO_{\Delta + \nu^*(\Gamma_q)} \bigr)$.  There is a morphism $\chi\colon \cB^{\vee} \rightarrow (\pi_2)^*(\cM)^{\vee}$ of vector bundles over $X\times W^5_{16}(X)$ obtained as the dual of the evaluation map and the surface $Y$ is realized as its degeneracy locus.  Since we also have that
\[
c_{\mathrm{tot}} ( \cB^{\vee} )^{-1}=\Bigl( 1 + (d(L) \eta + \gamma) + (d(L) \eta + \gamma)^2 + \cdots \Bigr) \cdot \bigl( 1 - \eta \bigr) = 1 + 15\eta + \gamma - 2\eta\theta,
\]
we find the stated expression for $[Y]$  and finish the proof.
\end{proof}

We introduce two further vector bundles which appear in many of our calculation. Their Chern classes are computed via  Grothendieck-Riemann-Roch.

\begin{proposition}\label{a121}
Let $[X,q]$ be a general pointed curve of genus $12$ and consider the vector bundles $\cA_2$ and $\cB_2$ on $X \times {\rm{Pic}}^{16}(X)$ having fibers
\[
\cA_{2,(y,L)} = H^0 \bigl( X, L^{\otimes 2} (-2y) \bigr) \ \mbox{ and } \ \cB_{2,(y,L)} = H^0 \bigl( X, L^{\otimes 2} (-y-q) \bigr),
\]
respectively.  One then has the following formulas for their Chern classes:
\begin{align*}
c_1(\cA_2) &= -4 \theta - 4\gamma - 86 \eta, & c_1(\cB_2) &= -4 \theta - 2\gamma -31 \eta, \\
c_2(\cA_2) &= 8 \theta^2 + 320 \eta \theta + 16 \gamma \theta, & c_2(\cB_2) &= 8 \theta^2 + 116 \eta \theta + 8 \theta \gamma.
\end{align*}
\end{proposition}

\vspace{-5 pt}
\begin{proof}
We apply Grothendieck-Riemann-Roch to the projection map
\[
\nu\colon X\times X\times \mbox{Pic}^{16}(X)\rightarrow X\times \mbox{Pic}^{16}(X).
\]
  Via Grauert's Theorem, $\cA_2$ can be realized as a push forward under the map $\nu$, precisely
\[
\cA_2 = \nu_{!} \Bigl( \mu^*\bigl( \P^{\otimes 2} \otimes \OO_{X \times X \times \mathrm{Pic}^{16}(X)}(-2\Delta) \bigr) \Bigr) = \nu_* \Bigl( \mu^*\bigl( \P^{\otimes 2} \otimes \OO_{X \times X \times \mathrm{Pic}^{16}(X)}(-2\Delta) \bigr) \Bigr).
\]
Applying Grothendieck-Riemann-Roch to $\nu$, we find  $\mbox{ch}_2(\cA_2) = 8 \eta \theta$, and $\nu_* (c_1(\P)^2) = -2\theta$.  One then obtains  $c_1(\cA_2) = -4\theta - 4\gamma - (4d(L)+2g(C)-2) \eta$, which yields the formula for $c_2(\cA_2)$.  To determine the Chern classes of $\cB_2$, we observe $c_1(\cB_2) = -4\theta -2\gamma - (2d-1)\eta$ and $\mbox{ch}_2(\cB_2) = 4\eta \theta$.
\end{proof}

\section{The class of the virtual divisor $\widetilde{\mathfrak{D}}_{13}$}
In this section we determine the virtual class $[\widetilde{\mathfrak{D}}_{13}]^{\mathrm{virt}} := \sigma_* \Bigl(c_2(\mbox{Sym}^2(\E))^{\vee}-\F^{\vee}\Bigr)$ on $\widetilde{\cM}_{13}$.
We begin by recording the following formulas for a vector bundle $\mathcal{V}$ of rank $r+1$ on a stack $\mathcal{X}$:
\begin{equation*}
c_1\bigl(\mathrm{Sym}^2 (\mathcal{V})\bigr) = (r+2) c_1(\mathcal{V}), \ \  \ c_2\bigl(\mathrm{Sym}^2 (\mathcal{V})\bigr) = \frac{r(r+3)}{2} c_1^2(\mathcal{V}) + (r+3)c_2(\mathcal{V}).
\end{equation*}

\vskip 4pt

We apply these formulas for the first degeneracy locus of $\phi^{\vee}\colon \F^{\vee}\rightarrow \mbox{Sym}^2(\E)^{\vee}$.  By Definition~\ref{def:virtclass}, its class $[\fU]^{\vir}$ is given by
\begin{align}\label{eq:c2terms}
c_2\bigl(\mbox{Sym}^2(\E)^{\vee}-\F^{\vee}\bigr) & =c_2\bigl(\mbox{Sym}^2(\E)^{\vee}\bigr)-c_1\bigl(\mbox{Sym}^2(\E)^{\vee})\cdot c_1(\F^{\vee})+c_1^2(\F^{\vee})-c_2(\F^{\vee}),\\
 &= 20c_1^2(\E)+8c_2(\E)-7c_1(\E)\cdot c_1(\F)+c_1^2(\F)-c_2(\F). \nonumber
\end{align}
\vskip 3pt

In what follows we expand  the virtual class in $CH^1(\pm_{13})$ as
\begin{equation}\label{eq:expansion}
[\widetilde{\mathfrak{D}}_{13}]^{\mathrm{virt}} = a\lambda - b_0\delta_0 - b_1\delta_1.
\end{equation}
We compute the coefficients $a, b_0$ and $b_1$, by intersecting both sides of this expression with the test curves $F_0, F_1$ and $F_{\mathrm{ell}}$.  We start with the coefficient $b_1$.

\begin{theorem}
\label{d1}
Let $X$ be a general curve of genus $12$.  The coefficient $b_1$ in (\ref{eq:expansion}) is:
\[
b_1 = \frac{1}{2g(X)-2} \sigma^*(F_1)\cdot c_2 \bigr(\mathrm{Sym}^2(\E)^{\vee}-\F^{\vee} \bigr) = 11787.
\]
\end{theorem}

\begin{proof} We intersect the degeneracy locus of the map $\phi\colon \mbox{Sym}^2(\E)\rightarrow \F$ with $\sigma^*(F_1)$. By
 Proposition~\ref{limitlin1}, it suffices to estimate the contribution coming from $Z$. We write
\[
\sigma^*(F_1) \cdot c_2 \bigl(\mbox{Sym}^2(\E)^{\vee}-\F^{\vee}\bigr) = c_2 \bigl(\mbox{Sym}^2(\E)^{\vee}-\F^{\vee}\bigr)_{| Z}.
\]


In Proposition~\ref{xy}, we constructed a morphism  $\zeta\colon J_1(\P)^{\vee} \rightarrow \pi_2^*(\cM)^{\vee}$ of vector bundles on $Z$, whose fibers are the maps $H^0(\OO_{2y})^{\vee}\rightarrow H^0(X,L)^{\vee}$.  The kernel sheaf $\mbox{Ker}(\zeta)$ is locally free of rank $1$.  If $U$ is the line bundle on $Z$ with fibre
\[
U(y,L) = \frac{H^0(X, L)}{H^0(X, L(-2y))} \hookrightarrow H^0(X, L \otimes \OO_{2y})
\]
over a point $(y, L)\in Z$, then one has the following exact sequence over $Z$:
\[
0 \longrightarrow U \longrightarrow J_1(\P) \longrightarrow \bigl( \mbox{Ker}(\zeta) \bigr)^{\vee} \longrightarrow 0.
\]
In particular, by Proposition~\ref{xy}, we find that
\begin{equation}\label{Uosztaly}
c_1(U) = 2\gamma + 54\eta + c_1(\mathrm{Ker}(\zeta)).
\end{equation}

The product of the Chern class of $\mbox{Ker}(\zeta)$ with any class  $\xi \in H^2( X \times W^5_{16}(X), \mathbb{Z})$  is given by the Harris-Tu  formula \cite{HT}:
\begin{align}
\label{harristu1}
c_1 \bigl( \mathrm{Ker}(\zeta) \bigr) \cdot \xi_{| Z} &= -c_6 \Bigl( \pi_2^*(\cM)^{\vee} - J_1(\P)^{\vee} \Bigr) \cdot \xi_{| Z}, \\
& = -\Bigl( \pi_2^*(c_6) - 6\eta \theta \pi_2^*(c_4) + (54\eta +2 \gamma)\pi_2^*(c_5) \Bigr) \cdot \xi_{| Z}. \nonumber
\end{align}

Similarly, one has the formula \cite{HT} for the self-intersection on the surface $Z$:
\begin{equation}\label{harristu11}
c_1^2\bigl(\mathrm{Ker}(\zeta) \bigr)=\Bigl( \pi_2^*(c_7) - 6\eta \theta \pi_2^*(c_5) + (54\eta +2 \gamma)\pi_2^*(c_6) \Bigr)\in H^{14}\bigl(X\times W^5_{16}(X), \mathbb Z\bigr)\cong \mathbb Z.
\end{equation}
We also observe that  $c_7=0$, since the bundle $\cM$ has rank $6$.

\vskip 3pt

Let $\cA_3$ denote the vector bundle on $Z$ having fibers
\[
\cA_{3,(y,L)} = H^0(X, L^{\otimes 2})
\]
constructed as a push forward of a line bundle on $X \times X \times \mbox{Pic}^{16}(X)$. Then the line bundle $U^{\otimes 2}$ can be embedded in $\cA_3/\cA_2$.  We consider the quotient
\[
\G := \frac{\cA_3/\cA_2}{U^{\otimes 2}}.
\]
The morphism $U^{\otimes 2} \rightarrow \cA_3/\cA_2$ vanishes along the locus of pairs $(y,L)$ where $L$ has a base point.  It follows that the sheaf  $\G$ has torsion along the locus $\Gamma \subseteq Z$ consisting of pairs $(q,A(q))$, where $A \in W^5_{16}(X)$.  Furthermore, $\F_{|Z}$, as a subsheaf of $\cA_3$, can be identified with the kernel of the map $\cA_3 \rightarrow \G$. Summarizing, there is an exact sequence of vector bundles on $Z$
\begin{equation}
\label{exseqZ}
0 \longrightarrow \cA_{2 |Z}\longrightarrow \F_{| Z} \longrightarrow U^{\otimes 2} \longrightarrow 0.
\end{equation}
Over a general point $(y,L)\in Z$, this sequence reflects the decomposition
\[
\F(y,L) = H^0(X, L^{\otimes 2}(-2y)) \oplus K \cdot u^2,
\]
where $u \in H^0(X, L)$ is a section such that $\mbox{ord}_y (u) = 1$.


\vskip 4pt

Via the exact sequence (\ref{exseqZ}), one computes the Chern classes of $\F_{|Z}$:
\begin{align*}
c_1(\F_{| Z}) &= c_1(\cA_{2 |Z})+2c_1(U), &   c_2(\F_{| Z}) &= c_2(\cA_{2 |Z})+2c_1(\cA_{2 | Z}) c_1(U).
\end{align*}
Recalling that  $\E_{| Z} = \pi_2^*(\cM)_{|Z}$ and using (\ref{eq:c2terms}), we find that $\sigma^*(F_1) \cdot c_2 \bigl((\mbox{Sym}^2 \E)^{\vee}-\F^{\vee} \bigr)$ is equal to:
\begin{multline*}
20 c_1^2\bigl(\pi_2^*\cM^{\vee}_{|Z}\bigr) + 8c_2\bigl(\pi_2^*\cM^{\vee}_{|Z}\bigr)+ 7c_1\bigl(\pi_1^*\cM^{\vee}_{|Z}\bigr) \cdot c_1\bigl(\cA_{2 |Z}\bigr)+4c_1^2(U)-  \\
-c_2(\cA_{2 |Z})+14c_1\bigl(\pi_2^*\cM^{\vee}_{|Z}\bigr)\cdot c_1(U) + c_1^2(\cA_{2 | Z}) + 2 c_1^2(\cA_{2 |Z})\cdot c_1(U).
\end{multline*}

Here, $c_i(\pi_2^*\cM_{|Z}^{\vee}) = \pi_2^*(c_i)\in H^{2i}(Z, \mathbb{Z})$.  The Chern classes of $\cA_{2|Z}$ have been computed in Proposition \ref{a121}.  Formula (\ref{Uosztaly}) expresses $c_1(U)$ in terms of $c_1((\mbox{Ker}(\zeta))$ and the classes $\eta $ and $\gamma$. When expanding
$\sigma^*(F_1)\cdot c_2 \bigl(\mbox{Sym}^2(\E)^{\vee}-\F^{\vee}\bigr)$, one distinguishes between terms that \emph{do} and those that \emph{do not} contain
the first Chern class of $\mathrm{Ker}(\zeta)$. The coefficient of $c_1 \bigl( \mathrm{Ker}(\zeta) \bigr)$, as well as the contribution coming from $c_1^2\bigl( \mathrm{Ker}(\zeta) \bigr)$ in  the expression of $\sigma^*(F_1)\cdot c_2 \bigl( \mbox{Sym}^2( \E)^{\vee}-\F^{\vee} \bigr)$
is evaluated using the formulas (\ref{harristu1}) and (\ref{harristu11}) respectively.   To carry this out, we first consider the part of this product that \emph{does not} contain $c_1 \bigl( \mathrm{Ker}(\zeta) \bigr)$, and we obtain
\begin{multline*}
8 \pi_2^*(c_2)+ 20 \pi_2^*(c_1^2)  +c_1^2(\cA_{2 |Z}) +7 \pi_2^*(c_1)\cdot c_1(\cA_{2|Z})-c_2(\cA_{2|Z})
+4(2\gamma+54\eta)^2+\\ +2(2\gamma+54\eta)\cdot c_1(\cA_{2|Z})+14(2\gamma+54\eta)\cdot \pi_2^*(c_1)= \\
20\pi_2^*(c_1^2)+154\pi_2^*(c_1)\cdot \eta-28\pi_2^*(c_1)\cdot \theta-96\eta\theta+8\theta^2+8\pi_2^*(c_2)\in H^4 \bigl( X \times W^5_{16}(X), \mathbb{Z} \bigr).
\end{multline*}

This polynomial gets multiplied by the class $[Z]$, which is expressed in Proposition \ref{xy} as a degree $5$ polynomial in $\theta$, $\eta$ and $\pi_2^*(c_i)$. We obtain a homogeneous polynomial of degree $7$ viewed as an element of $H^{14}\bigl(X\times W^5_{16}(X), \mathbb Z\bigr)$.

\vskip 3pt

Next we turn our attention to the contribution $\sigma^*(F_1) \cdot c_2 \bigl(\mathrm{Sym}^2(\E)^{\vee} - \F^{\vee} \bigr)$  coming from terms that do contain $c_1 \bigl( \mbox{Ker}(\zeta) \bigr)$. This is given by the following formula:
\[
4c_1^2(\mbox{Ker}(\zeta)+c_1(\mbox{Ker}(\zeta)\cdot \Bigl(8(2\gamma+54\eta)+2c_1(\cA_{2|Z})+14\pi_2^*(c_1)\Bigr).
\]

Using (\ref{harristu1}) and (\ref{harristu11}), one ends up with the following homogeneous polynomial of degree $7$ in $\eta$, $\theta$, and $\pi_2^*(c_i)$ for $i = 1, \ldots,6$:
\[
84\pi_2^*(c_1c_4)\theta\eta-48\pi_2^*(c_4)\theta^2\eta-756\pi_2^*(c_1c_5)\eta+440\pi_2^*(c_5)\theta\eta-44\pi_2^*(c_6)\eta.
\]

Adding together the parts that do and those that do not contain $c_1\bigl(\mbox{Ker}(\zeta)\bigr)$, and using the fact that the only monomials that need to be retained are those containing $\eta$, after manipulations carried out using \emph{Maple}, one finds
\begin{multline*}
\sigma^*(F_1) \cdot c_2 \bigl(\mathrm{Sym}^2(\E)^{\vee} - \F^{\vee}\bigr) = \eta \pi_2^*\Bigl( -602 c_1 c_5 + 432 c_2 c_4 -120c_1^2c_3\theta +168c_1c_3\theta^2 - \\
- 48 c_3 \theta^3 +1080c_1^2c_4-1428c_1c_4\theta -48c_2c_3\theta+ 384 c_4\theta^2 +344c_5\theta-44c_6\Bigr).
\end{multline*}

We suppress $\eta$ and the remaining polynomial lives inside $H^{12}(W^5_{16}(X), \mathbb{Z})\cong \mathbb Z$.  Using (\ref{chernosztalyok}) this expression is equal to
\[
\sigma^*(F_1) \cdot c_2 \bigl(\mathrm{Sym}^2(\E)^{\vee} - \F^{\vee}\bigr)=\frac{193}{45}\theta^6-\frac{1271}{30}\theta^5y_1+\frac{1607}{12}\theta^4y_1^2-120\theta_3y_1^3=259314,
\]
where for the last step we used the formulas (\ref{eq:topproducts}). We conclude
\[
b_1 = \frac{1}{22} \sigma^*(F_1) \cdot c_2 \bigl(\mbox{Sym}^2(\E)^{\vee}-\F^{\vee} \bigr) = 11787,
\]
as required.
\end{proof}

\begin{theorem}
\label{d0}
Let $[X,q]$ be a general pointed curve of genus $12$ and let $F_0 \subseteq \widetilde{\Delta}_0\subseteq \widetilde{\mathcal{M}}_{13}$ be the associated test curve.  Then the coefficient of $\delta_0$ in the expression (\ref{eq:expansion}) of $[\widetilde{\mathfrak{D}}_{13}]^{\mathrm{virt}}$ is equal to
\[
b_0 = \frac{\sigma^*(F_0) \cdot c_2 \Bigl(\mathrm{Sym}^2(\E)^{\vee}-\F^{\vee}\Bigr) + b_1}{24}=2247.
\]
\end{theorem}

\begin{proof}
Using Proposition~\ref{limitlin0}, we observe that
\begin{equation*}
c_2 \bigl(\mathrm{Sym}^2(\E)^{\vee}-\F^{\vee} \bigr)_{| \sigma^*(F_0)} = c_2 \bigl(\mathrm{Sym}^2(\E)^{\vee}-\F^{\vee}\bigr)_{|Y}.
\end{equation*}
We shall evaluate the Chern classes of $\F_{|Y}$ via  the line bundle $V$ on $Y$ with  fibre
\[
V(y,L) = \frac{H^0(X, L)}{H^0(X, L(-y-q))} \hookrightarrow H^0(X, L \otimes \OO_{y+q})
\]
over a point $(y,L) \in Y$.  We write the following  exact sequence over $Y$
\[
0 \longrightarrow V \longrightarrow \cB \longrightarrow \bigl( \mbox{Ker}(\chi) \bigr)^{\vee} \longrightarrow 0,
\]
where the morphism $\chi\colon \cB^{\vee} \rightarrow \pi_2^*(\cM)^{\vee}$ was defined in the final part of the proof of Proposition~\ref{xy}. In particular, we have
\begin{equation*}
c_1(V)=15\eta+\gamma+c_1\bigl(\mbox{Ker}(\chi)\bigr).
\end{equation*}

The effect on multiplying  $c_1\bigl(\mbox{Ker}(\chi)\bigr)$ against a class $\xi\in H^2\bigl(X\times W^5_{16}(X), \mathbb Z\bigr)$ is described  by applying once more the Harris-Tu \cite{HT} formula:
\begin{equation}\label{harristu3}
c_1 \bigl( \mathrm{Ker}(\chi) \bigr) \cdot \xi_{| Y} =\Bigl( -\pi_2^*(c_6) - 2\eta \theta \pi_2^*(c_4) + (15\eta +\gamma)\pi_2^*(c_5) \Bigr) \cdot \xi_{| Y},
\end{equation}
where we recall that $\pi_2\colon X\times W^5_{16}(X)\rightarrow W^5_{16}(X)$ and $c_i\in H^{2i}(W^5_{16}(X), \mathbb Z)$.
Similarly, for the self-intersection on $Y$ the following formula holds:
\begin{equation}\label{harristu4}
c_1^2\bigl(\mathrm{Ker}(\chi) \bigr)=- 2\eta \theta \pi_2^*(c_5) + (15\eta + \gamma)\pi_2^*(c_6) \in  H^{14}\bigl(X\times W^5_{16}(X), \mathbb Z\bigr).
\end{equation}

\vskip 3pt

We have also introduced in Proposition~\ref{a121} the vector bundle $\cB_2$ on $X\times \mbox{Pic}^{16}(X)$ with fibers $\cB_{2, (y,L)}=H^0(X, L^{\otimes 2}(-y-q))$ over a point $(y,L)$.  A local calculation along the lines of the one in  the proof of Theorem~\ref{d1} shows that one also has an exact sequence on $Y$, which can then be used to determine the Chern numbers of $\F_{|Y}$.
\[
0 \longrightarrow \cB_{2|Y} \longrightarrow \F_{|Y} \longrightarrow V^{\otimes 2} \longrightarrow 0.
\]
This exact sequence reflects the fact for a general point $(y,L)\in Y$ one has a decomposition
$\F(y,L)=H^0(X,L^{\otimes 2}(-y-q))\oplus K\cdot u^2$, where $u\in H^0(X,L)$ is a section not vanishing at $y$ and $q$.
We thus obtain the formulas:

\begin{align*}
c_1(\F_{| Y}) &= c_1(\cB_{2 |Z})+2c_1(V), &   c_2(\F_{| Y}) &= c_2(\cB_{2 |Y})+2c_1(\cB_{2 | Y}) c_1(V).
\end{align*}

\vskip 3pt

To estimate $c_2\bigl(\mbox{Sym}^2(\E)^{\vee}-\F^{\vee}\bigr)_{|Y}$ we use (\ref{eq:c2terms}) and write:
\begin{align*}
\sigma^*(F_0) &\cdot c_2 \bigl((\mbox{Sym}^2 \E)^{\vee}-\F^{\vee} \bigr) = 20c_1^2\bigl(\pi_1^*\cM^{\vee}_{|Y}\bigr)  + 8c_2\bigl(\pi_2^*\cM^{\vee}_{|Y}\bigr)+ 7c_1\bigl(\pi_1^*\cM^{\vee}_{|Y}\bigr) \cdot c_1\bigl(\cB_{2 |Y}\bigr)  \\
&+4c_1^2(V)-c_2(\cB_{2 |Y})+14c_1\bigl(\pi_2^*\cM^{\vee}_{|Y}\bigr)\cdot c_1(V) + c_1^2(\cB_{2 | Y}) + 2 c_1(\cB_{2 |Y})\cdot c_1(V).
\end{align*}

We expand this expression, collect the terms that do not contain $c_1 \bigl( \mbox{Ker}(\chi) \bigr)$, and obtain:
\begin{align*}
20\pi_2^*(c_1^2)-7\eta\pi_2^*(c_1)-28\theta \cdot \pi_2^*(c_1)+4\theta \eta+8\theta^2+8\pi_2^*(c_2).
\end{align*}
This quadratic polynomial gets multiplied with the class $[Y]$ computed in Proposition \ref{xy}. Next, we  collect the terms in
$\sigma^*(F_0) \cdot c_2 \bigl(\mbox{Sym}^2 \E^{\vee}-\F^{\vee} \bigr)$ that do contain $c_1 \bigl( \mbox{Ker}(\chi) \bigr)$:

\begin{align*}
4c_1^2\bigl(\mbox{Ker}(\chi)\bigr)+c_1\bigl(\mbox{Ker}(\chi)\bigr)\Bigl(8(15\eta+\gamma)+14\pi_2^*(c_1)+2c_1(\cB_{2|Y}\Bigr).
\end{align*}
This part of the contribution is evaluated using formulas (\ref{harristu3})  and (\ref{harristu4}).

\vskip 4pt

Putting everything together, we obtain a polynomial in $H^{14}\bigl(X\times W^5_{16}(X), \mathbb Z\bigr)\cong \mathbb Z$, as in the proof of Theorem \ref{d1}:
\[
\sigma^*(F_0) \cdot c_2 \bigl(\mathrm{Sym}^2(\E)^{\vee} - \F^{\vee}\bigr) = \eta \pi_2^*\Bigl(-40c_1^2c_3\theta+56c_1c_3\theta^2-16c_3\theta^3+300c_1^2c_4-392c_1c_4\theta-
\]
\[-16c_2c_3\theta+104c_4\theta^2-217c_1c_5+120c_2c_4+124c_5\theta+2c_6\Bigr).
\]
Applying (\ref{chernosztalyok}) and then (\ref{eq:topproducts}), after eliminating $\eta$, we obtain
\[\sigma^*(F_0) \cdot c_2 \bigl(\mathrm{Sym}^2(\E)^{\vee} - \F^{\vee}\bigr)=\frac{161}{180}\theta^6-\frac{28}{3}\theta^5y_1+\frac{755}{24}\theta^4y_1^2-30\theta^3y_1^3=42141.
\]
\vskip -19pt \end{proof}

\vskip 6pt

We can now complete the calculation of  $[\widetilde{\mathfrak{D}}_{13}]^{\mathrm{virt}}$.

\begin{proof}[Proof of Theorem~\ref{rho1virtual}.]
We consider the curve  $F_{\mathrm{ell}} \subseteq \pm_{g}$ defined in (\ref{fell}) obtained by attaching at the fixed point of a general curve $X$ of genus $12$ a pencil of plane cubics at one of the base points of the pencil.  Then one has the relation
\[
a - 12b_0 + b_1 = F_{\mathrm{ell}} \cdot \sigma_* c_2 \bigl(\mathrm{Sym}^2(\E)^{\vee}-\F^{\vee} \bigr) = 0.
\]
Using Theorems \ref{d1} and \ref{d0}, we thus find $a=15177$, for the $\lambda$-coefficient in the expansion (\ref{eq:expansion}). This completes the calculation of the virtual class $[\widetilde{\mathfrak{D}}_{13}]^{\vir}$.
\end{proof}

\vskip 3pt

\noindent We finally explain how Theorem \ref{rho1virtual} and Theorem \ref{thm:strongmrc}  (proved in \S \ref{Sec:Construct}) together imply Theorem \ref{thm:main13}.

\medskip

\begin{proof}[Proof of Theorem~\ref{thm:main13}.]
We write $[\overline{\mathfrak{D}}_{13}]=a\lambda-b_0\delta_0-\cdots-b_6\delta_6$, where
$a, b_0$ and $b_1$ are determined by Theorem \ref{rho1virtual}.  Applying \cite[Theorem 1.1]{FP} we have the inequalities  $b_i\geq (6i+8)b_0-(i+1)a\geq b_0$ for $i=2, \ldots, 6$, which shows that $s(\overline{\mathfrak{D}}_{13})=\frac{a}{b_0}=\frac{5059}{749}$.
\end{proof}

\section{The Strong Maximal Rank Conjecture in genus $13$} \label{Sec:Construct}

In this section and the next, we prove that $\widetilde{\mathfrak{D}}_{13}$ is not all of $\pm_{13}$ and that its condimension 1 part represents the virtual class $[\widetilde{\mathfrak{D}}_{13}]^{\vir}$.

To show that  $\widetilde{\mathfrak{D}}_{13}$ is not all of $\pm_{13}$, it suffices to prove the existence of one Brill-Noether general smooth curve $X$ of genus $13$ such that, for every  $L\in W^5_{16}(X)$, the multiplication map
\[
\phi_L\colon \mathrm{Sym}^2 H^0(X,L)\rightarrow H^0(X,L^{\otimes 2})
\]
is surjective. This is one case of the Strong Maximal Rank Conjecture \cite{AF}. The locus of such curves is Zariski open; to prove that it is nonempty over every algebraically closed field of characteristic zero, it suffices to show this over one such field.  Hence, we can and do assume that our ground field $K$ is spherically complete with respect to a surjective valuation $\nu \colon K^\times \to \RR$, and that $K$ has residue characteristic zero. This allows us to discuss the nonarchimedean analytifications of curves, the skeletons of those analytifications, and the tropicalizations of rational functions, viewed as sections of $L$ and $L^{\otimes 2}$. In this framework, we apply the method of tropical independence to give a lower bound for the rank of the multiplication map $\phi_L$, for every $L \in W^5_{16}(X)$.  The motivation and technical foundations for this approach are detailed in \S\S 1.4-1.5, \S\S 2.4-2.5, and \S 6 of \cite{FJP}, to which we refer the reader for details and further references.

After proving this case of the Strong Maximal Rank Conjecture, we will furthermore show that no component of the degeneracy locus $\fU$ in the parameter space $\widetilde{\mathfrak{G}}^5_{16}$ over $\pm_{13}$ maps with generically positive dimensional fibers onto a divisor in $\pm_{13}$. As in \cite{FJP}, this additional step is necessary to show that the push-forward of the virtual class is effective, and our proof involves analogous arguments on lower genus curves for linear series with ramification.  In particular, we will consider linear series with ramification on curves of genus $11$ and $12$ in \S\ref{sec:effectivity}, and so we set up our arguments here to work in this greater generality.

Let $X$ be a smooth projective curve of genus $11 \leq g \leq 13$ over $K$ whose Berkovich analytification $X^\an$ has a skeleton $\Gamma$ which is a chain of $g$ loops connected by bridges, as shown.  In order to simplify notation later, the vertices of $\Gamma$ are labeled $w_{13-g}, \ldots , w_{13}$, and $v_{14-g}, \ldots, v_{14}$, as shown in Figure~\ref{Fig:TheGraph}.

{\centering
\begin{figure}[H]
\begin{tikzpicture}
\draw (-1.85,-0.1) node {\footnotesize $w_{13-g}$};
\draw (-1.5,0)--(-0.5,0);
\draw (0,0) circle (0.5);
\draw (-.85,0.2) node {\footnotesize $v_{14-g}$};
\draw (.9,-0.2) node {\footnotesize $w_{14-g}$};
\draw (0.5,0)--(1.5,0);
\draw (2,0) circle (0.5);
\draw (2.5,0)--(3.5,0);
\draw (4,0) circle (0.5);
\draw (4.5,0)--(5.5,0);
\draw (6,0) circle (0.5);
\draw (6.5,0)--(7.5,0);
\draw (8,0) circle (0.5);
\draw (7.3,0.2) node {\footnotesize $v_{13}$};
\draw (8.8,-0.2) node {\footnotesize $w_{13}$};
\draw (8.5,0)--(9.5,0);
\draw (9.9,0) node {\footnotesize $v_{14}$};
\draw [<->] (2.6,0.15)--(3.3,0.15);
\draw [<->] (4.6,.25) arc[radius = 0.65, start angle=20, end angle=160];
\draw [<->] (4.61, -.15) arc[radius = 0.63, start angle=-9, end angle=-173];
\draw (3,0.4) node {\footnotesize$n_k$};
\draw (4,1) node {\footnotesize$\ell_k$};
\draw (4,-1) node {\footnotesize$m_k$};
\end{tikzpicture}
\caption{The chain of loops $\Gamma$.}
\label{Fig:TheGraph}
\end{figure}
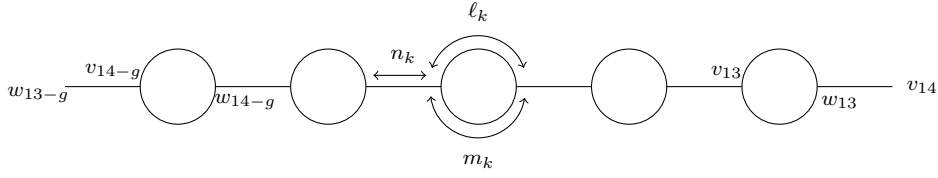
}

\noindent We write $\gamma_k$ for the loop formed by the two edges of length $\ell_k$ and $m_k$ between $v_k$ and $w_k$, for $14-g \leq k \leq 13$.  Similarly, we write $\beta_k$ for the bridge between $w_{k-1}$ and $v_{k}$, for $14-g \leq k \leq 14$, which has length $n_k$.  Except where stated otherwise, we assume that these edge lengths
satisfy
\begin{equation} \label{eq:admissible}
\ell_{k+1} \ll m_k \ll \ell_k \ll n_{k+1} \ll n_{k} \quad \mbox{ for all $k$}.
\end{equation}
These conditions on the edge lengths are precisely as in \cite[\S 6.4]{FJP}. Any curve $X$ whose analytification has such a skeleton is Brill-Noether general \cite{CDPR}.

Given a line bundle $L$ on $X$ we choose an identification $L = \mathcal{O}_X(D_X)$ so that any linear series $V \subseteq H^0(X,L)$ is identified with a finite dimensional vector space of rational functions $V \subseteq K(X)$.  The tropicalization of any nonzero rational function $f$ on $X$ is a piecewise linear function with integer slopes on $\Gamma$, and we write $\trop V$ for the set of all tropicalizations of nonzero functions in $V$.

Any sum of two functions in $\trop V$ is the tropicalization of a function in the image of the multiplication map $\phi_V\colon \Sym^2 V \to H^0(X, L^{\otimes 2})$. We say that a set of functions $\{ \psi_0, \ldots, \psi_n\}$ on $\Gamma$ is {\em tropically independent} if there are real numbers $b_0, \ldots, b_n$ such that
\[
\min \{ \psi_0 + b_0, \ldots, \psi_n + b_n\} \neq \min \{\psi_0 + b_0, \ldots, \widehat{ \psi_j + b_j }, \ldots, \psi_n + b_n \} \quad \mbox{ for $0 \leq j \leq n$}.
\]
In other words, $\{ \psi_0, \ldots, \psi_n \}$ is tropically independent if there are real numbers $b_0, \ldots, b_n$ such that each $\psi_j + b_j$ achieves the minimum uniquely in $\min_i \{ \psi_i + b_i \}$ at some point $v \in \Gamma$. The function $\theta = \min_i \{ \psi_i + b_i\}$ is then called an {\em independence}, since it verifies that $\{\psi_0, \ldots, \psi_n\}$ is independent.

We recall that tropical independence is a sufficient condition for linear independence; if $f_0, \ldots, f_n$ are nonzero rational functions on $X$ such that $\{ \trop(f_0), \ldots, \trop(f_n) \}$ is tropically independent on $\Gamma$, then $\{f_0, \ldots, f_n\}$ is linearly independent in $K(X)$.  Therefore, the relevant case of the Strong Maximal Rank Conjecture, and hence the fact that $\widetilde{\mathfrak{D}}_{13}$ is a divisor, follows immediately from the following.

\begin{theorem}
\label{thm:independence13}
Let $X$ be a curve of genus $13$ with skeleton $\Gamma$. Let $V$ be a linear series of degree 16 and dimension 5 on $X$, and let $\Sigma = \trop V$.  Then there is an independence $\theta$ among 20 pairwise sums of functions in $\Sigma$.
\end{theorem}

We will use the following generalization of Theorem~\ref{thm:independence13} in our proof that $\widetilde{\mathfrak{D}}_{13}$ represents the virtual class; the generalization involves analogous statements for linear series satisfying certain ramification conditions in genus 11 and 12.  The situation is closely parallel to that in \cite[\S 8.2]{FJP}. Recall that $a_0^V(p) < \ldots < a_r^V(p)$ denotes the vanishing sequence of a linear series $V$ of rank $r$ at a point $p$.

\begin{theorem}
\label{thm:independence}
Let $X$ be a curve of genus $g \in \{ 11, 12, 13 \}$ whose skeleton is $\Gamma$, and let $p \in X$ be a point specializing to $w_{13-g}$. Let $V$ be a linear series of degree 16 and dimension 5 on $X$, and let $\Sigma = \trop V$.  Assume that
\begin{enumerate}
\item  if $g=12$, then $a_1^V (p) \geq 2$, and
\item  if $g=11$, then either $a_1^V (p) \geq 3$, or $a_0^V (p) \geq 1$ and $a_2^V (p) \geq 4$.
\end{enumerate}
Then there is an independence $\theta$ among 20 pairwise sums of functions in $\Sigma$.
\end{theorem}

The remainder of this section is devoted to the proof of Theorem~\ref{thm:independence}. Our approach to constructing the independence is similar to that of \cite{FJP}, with a few important differences that we highlight when they arise.  Throughout, we let $D_X$ be a divisor class on $X$ with $V \subseteq H^0 (X, \cO (D_X))$.  We write $D = \Trop (D_X)$, and we assume that $D$ is a break divisor, meaning that it is the unique effective representative of its equivalence class with multiplicity $\deg D - g$ at $w_0$ and precisely one point of multiplicity 1 on each loop $\gamma_k$.  (See, for instance, \cite{ABKS}.) Let $R(D)$ denoted the complete tropical linear series of $D$, as in \cite{HMY}.  In other words, $R(D) = \{ \psi \in \PL(\Gamma) : D + \ddiv (\psi) \geq 0 \}$. Note, in particular, that $\Trop(V)$ is a tropical submodule of $R(D)$.

\begin{remark} \label{rem:differences}
The differences between the constructions of independences here and those in \cite{FJP} are subtle but crucial. Even  when $g = 13$, $[D]$ is vertex avoiding, and $\Sigma$ is unramified (the cases treated in \S\ref{Sec:VA}), if we apply the algorithm of \cite[\S 7]{FJP} naively, we obtain an independence among only 19 functions in $\Sigma$. To overcome this difficulty, we divide the graph into blocks in such a way that the lingering loop is the last loop in its block and has exactly two permissible functions.  This allows us to alter the algorithm slightly and assign a function to the lingering loop, raising the total number of functions in the independence to 20. See Remark~\ref{rem:newidea}.
\end{remark}

\subsection{The unramified vertex avoiding case}
\label{Sec:VA}

We first consider the case where $g=13$, $D$ is vertex avoiding, and $\Sigma = \trop V$ is unramified.  Unramified means that the ramification weights of $\trop V$ at $w_0$ and $v_{14}$, in the sense of \cite[Definition~6.17]{FJP}, are zero.  Vertex avoiding means that, for $0 \leq i \leq 5$, there is a unique divisor $D_i \sim D$ such that $D_i - iw_0 - (5-i)v_{14}$ is effective.  A vertex avoiding divisor is unramified if and only if the support of $D_i - iw_0 - (5-i)v_{14}$ contains neither $w_0$ nor $v_{14}$, for all $i$.

For $\psi \in \Sigma$, we write $s_k(\psi)$ and $s'_k(\psi)$ for the rightward slopes along the incoming and outgoing bridges of the $k$th loop $\gamma_k$, at $v_k$ and $w_k$, respectively. Since $\dim V = 6$, the functions in $\Sigma$ have exactly $6$ distinct slopes along each tangent vector in $\Gamma$.

\begin{figure}[h!]
\begin{tikzpicture}
\begin{scope}[grow=right, baseline]
\draw (-1,0) circle (1);
\draw (-3.5,0)--(-2,0);
\draw (0,0)--(1.5,0);
\draw [ball color=black] (-2,0.) circle (0.55mm);
\draw [ball color=black] (0,0) circle (0.55mm);
\draw (-1,-1.25) node {{$\gamma_k$}};
\draw (-2.25,-.25) node {{$v_k$}};
\draw (.25,-.25) node {{$w_k$}};
\draw (-2.5,.45) node {{$s_k$}};
\draw [->] (-3,.2)--(-2.15,0.2);
\draw (.6,.45) node {{$s'_k$}};
\draw [->] (.15,.2)--(1,0.2);
\end{scope}
\end{tikzpicture}
\caption{The slopes $s_k$ and $s'_k$.}
\end{figure}
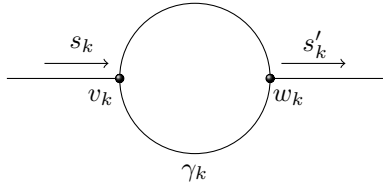

\begin{definition}
Let $s_k[0] < \cdots <s_k[5]$ and $s'_k[0]< \cdots < s'_k[5]$ denote the 6 distinct rightward slopes that occur as $s_k(\psi)$ and $s'_k(\psi)$, for $\psi \in \Sigma$.
\end{definition}

Since $D$ is vertex avoiding, there is a function $\varphi_i \in \Sigma$ such that
\[
s_k (\varphi_i ) = s_k [i] \mbox{ and } s'_k (\varphi_i ) = s'_k [i] \mbox{ for all } k,
\]
and it is unique up to additive constants.  Since $\Sigma$ is also unramified, there is a unique {\em lingering loop} $\gamma_{\ell}$, i.e., a unique loop $\gamma_\ell$ such that $s'_{\ell} [i] = s_{\ell} [i]$ for all $i$. Moreover, there is no function $\varphi \in \Sigma$ with the property that $s_{\ell} (\varphi) \leq s_{\ell} [i]$ and $s'_{\ell} (\varphi) \geq s'_{\ell} [i+1]$.  This last condition means that $\gamma_{\ell}$ is not a \emph{switching loop}, in the sense of \cite[Definition~6.19]{FJP}.

Our assumption that $\Sigma$ is unramified implies that the break divisor $D$ satisfies $\deg_{w_0} D = 3$, and the rightward slopes of the functions $\psi_i$ at $w_0$ are
\[
(s'_0 [0] , \ldots , s'_0 [5]) = (-2,-1,0,1,2, 3).
\]
Let us consider how the slope vector $(s'_k [0] , \ldots , s'_k [5])$ changes as we go from left to right across the graph. When crossing a loop other than the lingering loop $\gamma_{\ell}$, one of these slopes increases by 1, and the other 5 remain the same.  So, after the first non-lingering loop, the slopes are $(-2,-1,0,1,2,4)$, and after the second non-lingering loop, the slopes are either $(-2,-1,0,1,2,5)$ or $(-2,-1,0,1,3,4)$.  The data of these slopes is recorded by a standard Young tableau on a rectangle with 2 rows and 6 columns, filled with the symbols 1 through 13, excluding $\ell$.  If the symbol $k$ appears in column $i$, then it is the $(5-i)$th slope that increases on the loop $\gamma_k$, i.e., $s'_k[5-i] = s_k[5-i] + 1$.  Note, in particular, that each slope increases exactly twice, so $s'_{13} = (0,1,2,3,4,5)$ and no slope is ever greater than $5$.

\medskip

Let $\varphi_{ij} := \varphi_i + \varphi_j$.  To prove Theorem~\ref{thm:independence}, we construct an independence $\theta$ among 20 of the 21 functions in
\[
\cB = \{ \varphi_{ij}  \colon 0 \leq i \leq j \leq 5 \} .
\]
In order to describe this construction, we divide the graph into three connected regions consisting of some number of loops and the bridges between them that we call \emph{blocks}.  The construction ensures that, within each block, the slope of $\theta$ is nearly constant on each bridge, equal to 4 on bridges in the first block, 3 on bridges in the second block, and 2 on bridges in the third block.  The slope decreases by 1 at the midpoint of the bridges between blocks.

The blocks are specified as follows. Recall that $\gamma_\ell$ is the lingering loop. Let
\[
z_1 = \min \{ 6, \ell \} \quad \quad \mbox{ \ and \ } \quad \quad
z_2 = \max \{ 7, \ell \} .
\]
Then $\gamma_{z_1}$ and $\gamma_{z_2}$ are the last loops of the first and second blocks, respectively.  We construct our independence $\theta$ to satisfy
\begin{equation}\label{eq:slopes}
s_k (\theta) = \left\{ \begin{array}{ll}
4 & \textrm{if $k \leq z_1$,} \\
3 & \textrm{if $z_1 < k \leq z_2$,}\\
2 & \textrm{if $z_2 < k \leq 13$.}
\end{array} \right.
\end{equation}
Note that either $z_1$ or $z_2$ is equal to $\ell$, so the lingering loop $\gamma_{\ell}$ is always the last loop in its block.

When we construct $\theta$ as a tropical linear combination of the functions in $\cB$, we keep track of which functions achieve the minimum on which loops and bridges of $\Gamma$. The specified slopes of $\theta$ along the bridges within each block place natural constraints on which functions can achieve the minimum on a given loop, which we encode in the following definition of \emph{permissibility}.  In the vertex avoiding case, we apply this condition only to functions $\varphi_{ij} \in \cB$. However, we state the definition of permissibility more generally, for arbitrary functions $\psi$ in the complete tropical linear series $R(D)$, for later use in \S\S\ref{Sec:Slopes}-\ref{sec:switching}.

\begin{definition} \label{def:permissible}
Let $\psi \in R(D)$.  We say that $\psi$ is \emph{permissible} on $\gamma_k$ if
\[
s_k (\psi) \leq s_k (\theta) \quad \quad \mbox{ and } \quad \quad s'_k (\psi) \geq s_k (\theta).
\]
We say that $\psi$ is \emph{permissible} on a block if it is permissible on some loop in that block.  \end{definition}

To understand this definition, suppose that $\theta$ has nearly constant slope along the bridges within each block and on each half of the bridges between blocks, and that it is written as the minimum of finitely many functions in $R(D)$, including $\psi$. If $s_k (\psi) \geq s_k (\theta) + 1$, then the value of $\psi$ at $v_k$ exceeds the value of $\theta$ at $v_k$ by at least the length of the bridge $\beta_k$ (or half this length, if $\beta_k$ is the bridge between two blocks).   Since this bridge is much longer than the loop $\gamma_k$, it follows that $\psi$ cannot achieve the minimum at any point of $\gamma_k$.  A similar argument shows that if $s'_k (\psi) \leq s_k (\theta) -1$, then $\psi$ cannot achieve the minimum at any point of $\gamma_k$.

We construct $\theta$ algorithmically, moving from left to right across the graph. At each step, we keep track of which functions in $\cB$ are permissible on the given loop.  The set of loops on which a given function $\psi$ is permissible are indexed by the integers in an interval \cite[Remark~7.9]{FJP}, so we pay special attention to the first and last loops in these intervals.

Suppose $\gamma_k$ is the first loop on which $\varphi_{ij} \in \cB$ is permissible and it is not the first loop in its block.  Then $\gamma_k$ is the unique loop on which $\varphi_{ij}$ is permissible such that the first inequality in Definition~\ref{def:permissible} is strict.  Similarly, suppose $\gamma_k$ is the last loop on which $\varphi_{ij}$ is permissible and it is not the last loop in its block.  Then $\gamma_k$ is the unique loop on which $\varphi_{ij}$ is permissible such that the second inequality in Definition~\ref{def:permissible} is strict.  This motivates the following definition:

\begin{definition}
A permissible function $\psi$ is \emph{new} if $s_k (\psi) \leq s_k (\theta) - 1$ and \emph{departing} if $s'_k (\psi) \geq s_k (\theta) + 1$.
\end{definition}

Our choice of $z_1$ and $z_2$ determines which loops have new permissible functions in $\cB$.

\begin{proposition}
\label{Prop:VANew}
There are no new permissible functions of the form $\varphi_{ij}$ on $\gamma_k$ if and only if $k = \ell$ or
\vspace{-8 pt}
\begin{multicols}{2}
\begin{enumerate}
\item  $\ell > 6$ and $k = 6$;
\item  $\ell > 7$ and $k = 7$;
\item  $\ell < 9$ and $k = 9$; or
\item  $\ell \leq 7$, $s'_7 [5] = 4$, and $k=8$.
\end{enumerate}
\end{multicols}
\end{proposition}

\begin{proof}
There is no new permissible function on the lingering loop $\gamma_{\ell}$.  Suppose $k \neq \ell$.  Let $j$ be the unique integer satisfying $s'_k [j] = s_k [j] + 1$.   There is a new permissible function in $\cB$ on $\gamma_k$ if and only if either the function $\varphi_{jj}$ is both new and departing, or there is an integer $i$ such that $s'_k (\varphi_{ij}) = s_k (\theta)$.  We now examine when such an $i$ exists.

The values $s'_k [i]$ are 6 distinct integers between $-2$ and 5.  Let $a$ and $b$ be the two integers in this range that are not equal to $s'_k [i]$ for any $i$.  On the $h$th non-lingering loop, one has
\[
h = \sum_{i=0}^5 (s'_k [i] +2 - i) = 9 - (a+b) .
\]
Since $s'_k [j] = s_k [j] + 1$, we must have that $s'_k [j]$ is equal to either $a+1$ or $b+1$.  Without loss of generality, assume that it is equal to $a+1$.  There does not exist $i$ such that $s'_k [i] + s'_k [j] = s'_k (\theta)$ if and only if $s'_k (\theta)-(a+1)$ is greater than 5, smaller than $-2$, or equal to either $a$ or $b$.  If it is equal to $a$, then the function $\varphi_{jj}$ is both new and departing.  Since $s'_k (\theta) \leq 4$ and $a+1 \geq -1$, we see that $s'_k (\theta)-(a+1)$ cannot be greater than 5, and $s'_k (\theta)-(a+1)$ is smaller than $-2$ if and only if $s'_k (\theta) = 2$ and $a=4$.  By the above calculation, $b = s'_k (\theta)-(a+1)$ if and only if $h =10 - s'_k (\theta)$.

The 6th non-lingering loop is contained in the first block if and only if $\ell > 6$.  The 7th non-lingering loop is contained in the second block if and only if $\ell > 7$.  The 8th non-lingering loop is contained in the third block if and only if $\ell < 9$.  Finally, if $a=4$, then $\gamma_k$ is one of the first 7 non-lingering loops.  If $\gamma_k$ is in the third block, then since $z_2 \geq 7$, we have $\ell \leq 7$, and $\gamma_k$ is the first loop in the third block.
\end{proof}

Having determined which loops have new permissible functions in $\cB$, we can now strategically choose the subset $\cB' \subset \cB$ from which we will construct the independence $\theta$, so that the number of permissible functions in $\cB'$ on each block is precisely one more than the number of loops in the block.  Note that $|\cB| = 21$, so $\cB'$ is chosen by omitting a single function $\psi$ from $\cB$.

\begin{definition}
If $\ell \leq 7$, let $\psi \in \cB$ be a function that is permissible on the second block.  Otherwise, let $\psi \in \cB$ be a function that is permissible on the third block. Let $\cB' = \cB \smallsetminus \{ \psi \}$.
\end{definition}

\begin{remark}
Note that there may be several functions that are permissible on the specified block; it does not matter which of these we omit from $\cB'$.
\end{remark}

\begin{lemma} \label{lem:countpermissible}
On each block, the number of permissible functions in $\cB'$ is one more than the number of loops.
\end{lemma}

\begin{proof}
This follows directly from Proposition~\ref{Prop:VANew}.  Specifically, since $z_1 = \min \{ 6, \ell \}$, there is a new permissible function in $\cB$ on each loop of the first block, except for the last one.  Since there are precisely two pairs $(i,j)$ such that $s_1 (\varphi_{ij}) = 4$, we see that the number of permissible functions on the first block is 1 more than the number of loops.  By symmetry, if $z_2 \leq 7$, then the number of permissible functions in $\cB$ on the third block is 1 more than the number of loops, and if $z_2 > 7$, it is 2 more.  But when $z_2 > 7$, one of these functions is not in $\cB'$.

Finally, we consider the middle block.  We count the number of pairs $(i,j)$ such that $s'_{z_1} (\varphi_{ij}) = 3$.  Since 3 is odd, if $(i,j)$ is such a pair, then $i \neq j$.  It follows that there are 3 such pairs if and only if $s'_{z_1} [i] + s'_{z_1} [5-i] = 3$ for all $i$, which implies that there are precisely 6 non-lingering loops in the first block.  It follows that, if $\ell < 7$, then there are precisely two such pairs, and if $\ell \geq 7$, there are three such pairs.  By Proposition~\ref{Prop:VANew}, if $\ell < 7$, there is a new permissible function on every loop of the middle block.  If $\ell = 7$, then the middle block contains only one loop, and since this loop is lingering, there are no new permissible functions on it.  In both of these cases, the number of permissible functions in $\cB$ on the middle block is therefore 2 more than the number of loops, but one of these functions is not in $\cB'$.  If $\ell > 7$, then there are no new permissible functions on $\gamma_7$ or $\gamma_{\ell}$, so the number of permissible functions is 1 more than the number of loops.
\end{proof}

We now describe the algorithm for constructing our independence $$\theta = \min_{\varphi_{ij} \in \cB'} \{ \varphi_{ij} + c_{ij} \},$$ with slopes $s_k(\theta)$ as specified in \eqref{eq:slopes}, when $g = 13$, $D$ is vertex avoiding, and $\Sigma$ is unramified. The algorithm is quite similar to that presented in \cite[\S~7]{FJP}. We include the details.  See Example~\ref{ex:randomtableau} for an illustration of the output in one particular case.

In this algorithm, we move from left to right across each of the three blocks where $s_k(\theta)$ is constant, adjusting the coefficients of unassigned permissible functions and assigning one function $\varphi_{ij} \in \cB'$ to each loop so that each function achieves the minimum uniquely on some part of the loop to which it is assigned. At the end of each block, there is one remaining unassigned permissible function that achieves the minimum uniquely on the bridge immediately after the block, which we assign to that bridge.  Since there are 13 loops and three blocks, this gives us an independent configuration of 16  functions.  The remaining 4 functions, with slopes too high or too low to be permissible on any block, achieve the minimum uniquely on the bridges to the left of the first loop or to the right of the last loop, respectively.  Example~\ref{ex:randomtableau} illustrates the procedure for one randomly chosen tableau.  We now list a few of the key properties of the algorithm:
\renewcommand{\theenumi}{\roman{enumi}}
\begin{enumerate}
\item Once a function has been assigned to a bridge or loop, it always achieves the minimum uniquely at some point on that bridge or loop;
\item A function never achieves the minimum on any loop to the right of the bridge or loop to which it is assigned;
\item The coefficient of each function is initialized to $\infty$ and then assigned a finite value when the function is assigned to a bridge or becomes permissible on a loop, whichever comes first.
\item After the initial assignment of a finite coefficient, subsequent adjustments to this coefficient are smaller and smaller perturbations.  This is related to the fact that the edges get shorter and shorter as we move from left to right across the graph.
\item Only the coefficients of unassigned functions are adjusted, and all adjustments are upward.  This ensures that once a function is assigned and achieves the minimum uniquely on a loop, it always achieves the minimum uniquely on that loop.
\item Exactly one function is assigned to each of the 13 loops, and the remaining seven functions are assigned to either the leftmost bridge, the rightmost bridge, or one of the three bridges after the blocks.
\end{enumerate}

\noindent The algorithm terminates when we reach the rightmost bridge, at which point each of the 20 functions $\{ \varphi_{ij} + c_{ij} : \varphi_{ij} \in \cB'\}$ achieves the minimum uniquely at some point on the graph.

\begin{remark} \label{rem:newidea}
The one crucial difference, in comparison with the construction in \cite[\S~7]{FJP}, is that we do \emph{not} skip the lingering loop $\gamma_{\ell}$.  Instead, since $\gamma_{\ell}$ is the last loop in its block, there are precisely two unassigned permissible functions on $\gamma_{\ell}$.  These two functions do not have identical restrictions to $\gamma_{\ell}$.  Thus, if we adjust their coefficients upward so that they agree at $w_{\ell}$, one of them will obtain the minimum uniquely at some point of the loop $\gamma_{\ell}$.  We assign this function to $\gamma_{\ell}$ and adjust its coefficient upward by an amount small enough so that it still obtains the minimum uniquely at some point of $\gamma_{\ell}$.  The other achieves the minimum uniquely at $w_{\ell}$, and we assign it to the bridge $\beta_{\ell+1}$.
\end{remark}

The algorithm depends on the following basic properties of the permissible functions $\varphi_{ij}$.

\begin{lemma}
\label{lem:onedeparts}
There is at most one departing permissible function $\varphi_{ij}$  on each loop $\gamma_k$.  Furthermore, if $\gamma_k$ is lingering then there are none.
\end{lemma}

\begin{proof}
The proof is identical to \cite[Lemma~7.12]{FJP}.
\end{proof}

\begin{lemma}
\label{Lem:VAAtMostThree}
For any loop $\gamma_k$, there are at most 3 non-departing permissible functions in $\cB$ on $\gamma_k$.
\end{lemma}

\begin{proof}
If $\varphi_{ij}$ is a non-departing permissible function on $\gamma_k$, then $s_{k+1} (\varphi_{ij}) = s_{k} (\theta)$.  For each $i$, this equality holds for at most one $j$, and the lemma follows.
\end{proof}

\begin{proposition}  \label{prop:threeshape}
Consider a set of at most three non-departing permissible functions from $\cB$ on a loop $\gamma_k$ and assume that all of the functions take the same value at $w_k$.  Then there is a point of $\gamma_k$ at which one of these functions is strictly less than the others.
\end{proposition}

\begin{proof}
The proof is identical to \cite[Proposition~7.20]{FJP}.
\end{proof}

The algorithm is as follows:

\bigskip

\noindent \textbf{Start at the first bridge.}
Start at $\beta_1$ and initialize $c_{55} = 0$.  Initialize $c_{45}$ so that $\varphi_{45} + c_{45}$ equals $\varphi_{55}$ at a point one third of the way from $w_0$ to $v_1$.   Initialize $c_{44}$ and $c_{35}$ so that $\varphi_{44} + c_{44}$ and $\varphi_{35} + c_{35}$ agree with $\varphi_{45} + c_{45}$ at a point two thirds of the way from $w_0$ to $v_1$.  Initialize all other coefficients $c_{ij}$ to $\infty$.  Note that $\varphi_{55}$ and $\varphi_{45}$ achieve the minimum uniquely on the first and second third of $\beta_1$, respectively.  Assign both of these functions to $\beta_1$, and proceed to the first loop.

\medskip

\noindent \textbf{Loop subroutine.}
Each time we arrive at a loop $\gamma_k$, apply the following steps.

\medskip

\noindent \textbf{Loop subroutine, step 1:  Re-initialize unassigned coefficients.}
By Lemma~\ref{lem:nonew} below, there are at least two unassigned permissible functions.  Find the unassigned permissible function $\varphi_{ij}$ that maximizes $\varphi_{ij}(w_k) + c_{ij}$.   Initialize the coefficients of the new permissible functions (if any) and adjust the coefficients of the other unassigned permissible functions upward so that they all agree with $\varphi_{ij}$ at $w_k$.    (The unassigned permissible functions are strictly less than all other functions on $\gamma_k$, even after this upward adjustment; see Lemma~\ref{lem:1more}.)

\medskip

\noindent \textbf{Loop subroutine, step 2:  Assign departing functions.}
If there is a departing function, assign it to the loop.  (There is at most one, by Lemma~\ref{lem:onedeparts}.) Adjust the coefficients of the other permissible functions upward so that all of the functions agree at a point on the following bridge a short distance to the right of $w_k$, but far enough so that the departing function achieves the minimum uniquely on the whole loop.  This is possible because the bridge is much longer than the edges in the loop.  Proceed to the next loop.

\medskip

\noindent \textbf{Loop subroutine, step 4: Otherwise, use Proposition~\ref{prop:threeshape}.}
By Lemma~\ref{Lem:VAAtMostThree}, there are at most 3 non-departing functions.  By Proposition~\ref{prop:threeshape}, there is one $\varphi_{ij}$ that achieves the minimum uniquely at some point of $\gamma_k$.  We adjust the coefficient of $\varphi_{ij}$ upward by $\frac{1}{3}m_k$.  This ensures that it will never achieve the minimum on any loops to the right, yet still achieves the minimum uniquely on this loop; see Lemma~\ref{lem:1more}, below.  Assign $\varphi_{ij}$ to $\gamma_k$, and proceed to the next loop.

\medskip

\noindent \textbf{Proceeding to the next loop.}
If the next loop is contained in the same block, then move right to the next loop and apply the loop subroutine.  Otherwise, the current loop is the last loop in its block.  In this case, proceed to the next block.

\medskip

\noindent \textbf{Proceeding to the next block.}
After applying the loop subroutine to the last loop in a block, there is exactly one unassigned permissible function in $\cB'$, by Lemma~\ref{lem:countpermissible}.  The unassigned permissible function $\varphi_{ij}$ achieves the minimum uniquely on the beginning of the outgoing bridge, without any further adjustment of coefficients. Assign $\varphi_{ij}$ to this bridge.

If we are at the last loop $\gamma_{g}$, then proceed to the last bridge.  Otherwise, there are several  permissible functions on the first loop of the next block, as detailed in Lemma~\ref{Lem:VAAtMostThree}, above.  Initialize the coefficient of each permissible function on the first loop of the next block so that it is equal to $\theta$ at the midpoint of the bridge between the blocks, and then apply the loop subroutine.

\medskip

\noindent \textbf{The last bridge.}
Initialize the coefficient $c_{01}$ so that $\varphi_{01} + c_{01}$ equals $\theta$ at the midpoint of the last bridge $\beta_{14}$.  Initialize $c_{00}$ so that $\varphi_{00} + c_{00}$ equals $\theta$ halfway between the midpoint and the rightmost endpoint.  Note that both of these functions now achieve the minimum uniquely at some point on the second half of $\beta_{14}$.  Assign both of these functions to $\beta_{14}$, and output $\theta = \min \{\varphi_{ij} +  c_{ij} : \varphi_{ij} \in \cB' \}$.

\bigskip

To verify that this algorithm produces a tropical independence, we first show that there are at least two unassigned permissible functions on each loop.

\begin{lemma}
\label{lem:nonew}
There are at least two unassigned permissible functions on each loop $\gamma_k$.
\end{lemma}

\begin{proof}
By Lemma~\ref{lem:countpermissible}, the number of permissible functions in $\cB'$ on the block containing $\gamma_k$ is one more than the number of loops.  Since there is at most one new function per loop, the number of functions in $\cB'$ that are permissible on some loop between the first loop of the block and $\gamma_k$, inclusive, is at least one more than the number of loops.  Finally, note that exactly one function is assigned to each loop, and moreover, if a function is departing, it is assigned.  It follows by induction on $k'$ that the number of functions in $\cB'$ that are unassigned and permissible on some loop between $\gamma_{k'}$ and $\gamma_k$ is at least $k-k'+2$.  Hence, the number of unassigned permissible functions on $\gamma_k$ is at least two.
\end{proof}

We now verify that this algorithm produces a tropical independence.

\begin{lemma}
\label{lem:1more}
Suppose that $\varphi_{ij}$ is assigned to the loop $\gamma_k$ or the bridge $\beta_k$.  Then $\varphi_{ij}$ does not achieve the minimum at any point to the right of $v_{k+1}$.
\end{lemma}

\begin{proof}
If $\gamma_k$ is a non-lingering loop, then the proof is the same as \cite[Lemma~7.21]{FJP}.  On the other hand, if $\gamma_k$ is the lingering loop, then it is the last loop in its block.  Since $v_{k+1}$ is the start of the next block, $\varphi_{ij}$ cannot achieve the minimum at any point to the right of $v_{k+1}$.
\end{proof}

This completes the proof of Theorem~\ref{thm:independence} in the vertex avoiding case.

\begin{remark}
\label{Rmk:BridgeLength}
For future reference, we note that the proof of Lemma~\ref{lem:1more} does not depend on the relative lengths of the bridges.  It only uses that the bridges are much longer than the loops. The assumption that each bridge is much longer than the next is only used later, when there are decreasing bridges, decreasing loops, or switching loops.
\end{remark}

\begin{remark}
\label{Rmk:11and12}
If $\Gamma'$ is the subgraph of $\Gamma$ to the right of $w_1$, then $\Gamma'$ is a chain of 12 loops whose edge lengths satisfy the required conditions, and if the first loop is non-lingering, then the restriction of $\Sigma$ to $\Gamma'$ satisfies the ramification condition of Theorem~\ref{thm:independence}, with equality.  Similarly, the subgraph to the right of $w_2$ is a chain of 11 loops whose edge lengths satisfy the required conditions, and the restriction of $\Sigma$ to this subgraph satisfies the ramification condition of Theorem~\ref{thm:independence}, with equality.
To produce an independence in these cases, assign each function in $\cB'$ with slope greater than 4 to the first bridge, and then proceed as above.  There are precisely $15-g$ such functions, and they have distinct slopes along the first bridge, as in \cite[Lemma~8.25]{FJP}.  Because of this, we can choose coefficients so that each one obtains the minimum uniquely at some point of the first bridge.  Thus the unramified vertex avoiding cases of Theorem~\ref{thm:independence} for $g = 11$ and $12$ (i.e., when $\Sigma$ is unramified at $v_{14}$ and there is no extra ramification at $w_{13-g}$ beyond what is required by the inequalities on vanishing orders in the statement of the theorem) follow from essentially the same argument as for $g = 13$.  Our choice to index the vertices starting at $w_{13-g}$ reflects the idea that these linear series with ramification on a chain of $g = 11$ or $12$ loops behave like linear series on a chain of 13 loops restricted to the subgraph to the right of $w_{13-g}$.
\end{remark}

\begin{example}
\label{ex:randomtableau}
We illustrate the construction with an example.  Let $[D]$ be a vertex avoiding class of degree $16$ and rank $5$ associated to the tableau in Figure~\ref{Fig:RandomTableau}.

\begin{figure}[h!]
\begin{ytableau}
1 & 3 & 4 & 8 & 9 & 10 \\
2 & 5 & 7 & 11 & 12 & 13
\end{ytableau}
\caption{The tableau corresponding to the divisor $D$.}
\label{Fig:RandomTableau}
\end{figure}

The independence $\theta = \min_{ij} \{ \varphi_{ij} + c_{ij} \}$ that we construct is depicted schematically in Figure~\ref{Fig:Config}.  The graph should be read from left to right and top to bottom, so the first 6 loops appear in the first row, with $\gamma_1$ on the left and $\gamma_6$ on the right, and $\gamma_{13}$ is the last loop in the third row.  The rows correspond to the three blocks.   The 31 dots indicate the support of the divisor $D' = 2D + \ddiv (\theta)$. Note that $\deg(D') = 32$; the point on the bridge $\beta_4$ appears with multiplicity 2, as marked.  Because $\ell = 6$, there is a function that is permissible on the second block in $\cB$ but not $\cB'$.  The functions in $\cB$ that are permissible on the second block are precisely $\varphi_{05}$, $\varphi_{14}$, and $\varphi_{23}$; we have chosen (arbitrarily) to omit $\varphi_{23}$ from $\cB'$. Each of the 20 functions $\varphi_{ij}$ in $\cB'$ achieves the minimum uniquely on the connected component of the complement of $\mathrm{Supp}(D')$ labeled $ij$.

\begin{figure}[H]

\begin{center}
\scalebox{.8}{
\begin{tikzpicture}
\draw (-1.5,12)--(0,12);
\draw (-1.25,12.2) node {\footnotesize $55$};
\draw [ball color=black] (-1,12) circle (0.55mm);
\draw (-0.75,12.2) node {\footnotesize $45$};
\draw [ball color=black] (-0.5,12) circle (0.55mm);

\draw (0.5,12) circle (0.5);
\draw (0.5,12.7) node {\footnotesize $35$};
\draw [ball color=black] (0.75,12.43) circle (0.55mm);

\draw (1,12)--(2,12);
\draw [ball color=black] (1.15,12) circle (0.55mm);
\draw (2.5,12) circle (0.5);
\draw (2.5,12.7) node {\footnotesize $25$};
\draw [ball color=black] (2.75,12.43) circle (0.55mm);

\draw (3,12)--(4,12);
\draw [ball color=black] (3.15,12) circle (0.55mm);
\draw (4.5,12) circle (0.5);
\draw (4.5,12.7) node {\footnotesize $44$};

\draw (5,12)--(6,12);
\draw [ball color=black] (5.15,12) circle (0.55mm);
\draw (5.15,11.8) node {\tiny $2$};
\draw (6.5,12) circle (0.5);
\draw (6.5,12.7) node {\footnotesize $34$};
\draw [ball color=black] (6.25,12.43) circle (0.55mm);

\draw (7,12)--(8,12);
\draw [ball color=black] (7.15,12) circle (0.55mm);
\draw (8.5,12) circle (0.5);
\draw (8.5,12.7) node {\footnotesize $33$};
\draw [ball color=black] (8.75,12.43) circle (0.55mm);
\draw [ball color=black] (8.98,11.9) circle (0.55mm);
\draw (9,12)--(10,12);

\draw (10.5,12) circle (0.5);
\draw (10.5,12.7) node {\footnotesize $15$};
\draw (11.25,12.2) node {\footnotesize $24$};
\draw [ball color=black] (10.75,12.43) circle (0.55mm);
\draw [ball color=black] (10.25,12.43) circle (0.55mm);
\draw (11,12)--(12,12);
\draw [ball color=black] (11.5,12) circle (0.55mm);
\draw (12.5,12) node {$\cdots$};
\end{tikzpicture}
}

\bigskip

\scalebox{.78}{
\begin{tikzpicture}
\draw (6.5,6) node {$\cdots$};
\draw (7,6)--(8,6);
\draw (8.5,6) circle (0.5);
\draw (8.5,6.7) node {\footnotesize $14$};
\draw [ball color=black] (8.25,6.43) circle (0.55mm);
\draw [ball color=black] (8.75,6.43) circle (0.55mm);
\draw (9,6)--(10,6);
\draw [ball color=black] (9.5,6) circle (0.55mm);
\draw (9.25,6.2) node {\footnotesize $05$};
\draw (10.5,6) node {$\cdots$};
\end{tikzpicture}
}

\bigskip

\scalebox{.8}{
\begin{tikzpicture}
\draw (-4.5,3) node {$\cdots$};
\draw (-4,3)--(-3,3);
\draw (-2.5,3) circle (0.5);
\draw (-2.5,3.7) node {\footnotesize $22$};
\draw [ball color=black] (-2.75,3.43) circle (0.55mm);
\draw [ball color=black] (-2.25,3.43) circle (0.55mm);

\draw (-2,3)--(-1,3);
\draw (-.5,3) circle (0.5);
\draw (-.5,3.7) node {\footnotesize $13$};
\draw [ball color=black] (-.25,3.43) circle (0.55mm);
\draw [ball color=black] (0.15,3) circle (0.55mm);

\draw (0,3)--(1,3);
\draw (1.5,3) circle (0.5);
\draw (1.5,3.7) node {\footnotesize $04$};
\draw [ball color=black] (1.25,3.43) circle (0.55mm);
\draw [ball color=black] (2.15,3) circle (0.55mm);

\draw (2,3)--(3,3);
\draw (3.5,3) circle (0.5);
\draw (3.5,3.7) node {\footnotesize $03$};
\draw [ball color=black] (3.25,3.43) circle (0.55mm);
\draw [ball color=black] (3.75,3.43) circle (0.55mm);

\draw (4,3)--(5,3);
\draw [ball color=black] (4.15,3) circle (0.55mm);
\draw (5.5,3) circle (0.5);
\draw (5.5,3.7) node {\footnotesize $12$};
\draw [ball color=black] (5.25,3.43) circle (0.55mm);

\draw (6,3)--(7,3);
\draw [ball color=black] (6.15,3) circle (0.55mm);
\draw (7.5,3) circle (0.5);
\draw (7.5,3.7) node {\footnotesize $11$};
\draw [ball color=black] (7.25,3.43) circle (0.55mm);
\draw [ball color=black] (7.75,3.43) circle (0.55mm);

\draw (8,3)--(9.75,3);
\draw [ball color=black] (8.65,3) circle (0.55mm);
\draw [ball color=black] (9.2,3) circle (0.55mm);
\draw (8.4,3.2) node {\footnotesize $02$};
\draw (8.95,3.2) node {\footnotesize $01$};
\draw (9.5,3.2) node {\footnotesize $00$};
\end{tikzpicture}
}
\end{center}
\caption{The divisor $D' = 2D + \ddiv (\theta)$.  The function $\varphi_{ij}$ achieves the minimum uniquely on the region labeled $ij$ in $\Gamma \smallsetminus \mathrm{Supp}(D')$. }
\label{Fig:Config}
\end{figure}
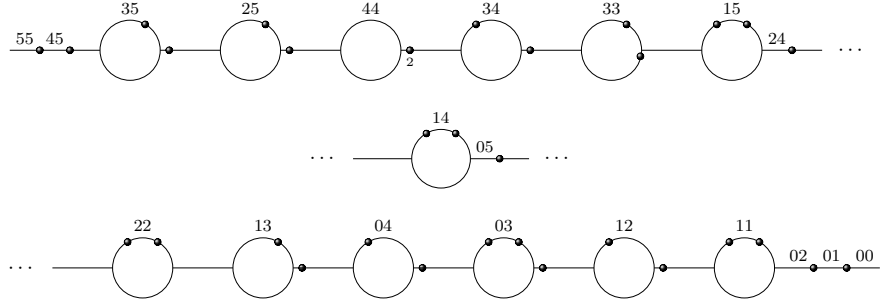

\end{example}

\subsection{No switching loops}
\label{Sec:Slopes}

Recall that a loop $\gamma_\ell$ is a {\em switching loop} for $\Sigma$ if there is some $\varphi \in \Sigma$ and some $h$ such that $s_\ell(\varphi) \leq s_\ell[h]$ and $s'_\ell(\varphi) \geq s'_\ell[h+1]$.  It is a \emph{lingering loop} if it is not a switching loop and $s_\ell[i] = s'_\ell[i]$ for all $i$. Recall also that $\gamma_\ell$ is a {\em decreasing loop} if $s_\ell[h] > s'_\ell[h]$.  Similarly $\beta_{\ell}$ is a {\em decreasing bridge} if $s'_{\ell-1}[h] > s_{\ell}[h]$.

Because we are only considering cases where the adjusted Brill-Noether number is at most 1, by \cite[Proposition~6.18]{FJP}, we know that there is at most one lingering loop, one positive ramification weight, one  decreasing loop, one decreasing bridge, or one switching loop, and these possibilities are mutually exclusive.  Moreover, for decreasing loops and bridges, the index $h$ is unique and the decrease in slope is exactly 1.   In this subsection, we consider all cases where there is no switching loop. The cases with a switching loop are discussed in \S \ref{sec:switching}.

Assume $\Sigma$ has no switching loops. Then, for all $i$ there is a function $\varphi_i \in \Sigma$ such that
\[
s_k (\varphi_i ) = s_k [i] \quad \mbox{ and } \quad s'_k (\varphi_i ) = s'_k [i] \mbox{ for all } k.
\]
We keep the notation $\varphi_{ij} = \varphi_i + \varphi_j$ and
$
\cB = \{ \varphi_{ij} : \colon 0 \leq i \leq j \leq 5 \} .
$
As in the unramified vertex avoiding case, we choose a subset $\cB' \subseteq \cB$ of 20 functions, and we choose integers $z_1$ and $z_2$ in order to divide the graph $\Gamma$ into 3 blocks.  We make our choices to satisfy the following conditions:
\begin{enumerate}
\item  no two functions in $\cB'$ that are permissible on $\gamma_k$ differ by a constant on $\gamma_k$,
\item  the number of functions in $\cB'$ that are permissible on each block is at most one more than the number of loops in that block,
\item  no function in $\cB'$ is permissible on more than one block,
\item  if $\gamma_k$ is a lingering loop, then it is the last loop in its block,
\item  if $\gamma_k$ is a decreasing loop and $j$ is the unique value such that $s'_k [j] < s_k [j]$, then no function of the form $\varphi_{ij} \in \cB'$ is permissible on $\gamma_k$, and
\item  if $\beta_k$ is a decreasing bridge and $j$ is the unique value such that $s_k [j] < s'_{k-1} [j]$, then either $\beta_k$ is a bridge between blocks, or no function of the form $\varphi_{ij} \in \cB'$ is permissible on $\gamma_{k-1}$.
\end{enumerate}

\begin{proposition}
\label{Prop:ConditionsImpyIndependent}
If $\cB'$ satisfies conditions (i)-(vi), then the functions in $\cB'$ are independent.
\end{proposition}

\begin{proof}
The algorithm for constructing the tropical independence is identical to the algorithm of Section~\ref{Sec:VA}, with the following exceptions.   First, as in Remark~\ref{Rmk:11and12}, we assign every function with slope greater than 4 to the first bridge.
Second, the procedure for Proceeding to the Next Block must be altered slightly when the bridge between the blocks is a decreasing bridge.

When the bridge between the blocks is a decreasing bridge, there is a unique point $v$ on the bridge where one of the functions $\varphi_i$ is locally nonlinear.  We initialize the coefficients of the new permissible functions on the next block so that they are equal to $\theta$ at a point to the right of $v$.  If one of the blocks contains zero loops, we set the coefficient of the unique function with slope equal to that of $\theta$ so that it is equal to $\theta$ at a point to the right of $v$, and initialize the coefficients of the new permissible functions on the next block so that they are equal to $\theta$ at a point to the right of this.

We note that there are at most 3 non-departing permissible functions in $\cB'$ on each loop.  This is because a non-departing permissible function $\varphi_{ij}$ on $\gamma_{k}$ satisfies $s_{k+1}(\varphi_{ij}) = s_k(\theta)$, and for each $i$ this equality can hold for at most one $j$.

To see that this algorithm produces an independence, suppose that $\varphi_{ij}$ is assigned to the loop $\gamma_k$ or the bridge $\beta_k$.  We show that $\varphi_{ij}$ does not achieve the minimum at any point to the right of $v_{k+1}$.  If $\gamma_k$ and $\beta_k$ both have multiplicity zero, then the argument is the same as in \cite[Lemma~7.21]{FJP}.  On the other hand, if $\gamma_k$ has positive multiplicity, then either $\gamma_k$ is a decreasing loop, or by (iv) it is the last loop in its block.  If $\gamma_k$ is a decreasing loop, then by (v) there is no function in $\cB'$ that is permissible on $\gamma_k$ and contains the decreasing function as a summand, so the result holds again by \cite[Lemma~7.21]{FJP}.  We may therefore assume that $\gamma_k$ is the last loop in its block, in which case the argument is identical to the vertex avoiding case above.

Similarly, if $\beta_k$ has positive multiplicity, then by (vi) there are two possibilities.  If $\varphi_{ij}$ does not contain the decreasing function as a summand, then there is nothing to show.  Otherwise, $\beta_k$ is a bridge between blocks. By (iii) the function $\varphi_{ij}$ is only permissible on one block.  Since $v_{k+1}$ is the start of the next block, $\varphi_{ij}$ cannot achieve the minimum at any point to the right of $v_{k+1}$.
\end{proof}

For the rest of this section, we explain how to choose $z_1$, $z_2$, and the set $\cB'$ in order to satisfy conditions (i)-(vi).  This is done by a careful case analysis, depending on combinatorial properties of the tropical linear series $\Sigma$.

\medskip

\textbf{Case 1:  There are no loops or bridges of positive multiplicity.} This guarantees that either the linear series is ramified at $v_{14}$, or has ``extra ramification'' at $w_{13-g}$, meaning that $g=13$ and the linear series is ramified at $w_0$, or $g = 11$ or 12 and the linear series has more ramification than what is imposed by the inequalities on vanishing numbers in Theorem~\ref{thm:independence}s.  In these cases, which are mutually exclusive, we choose $z_1$ and $z_2$ so that $\gamma_{z_1}$ is the first loop in the first block with no new function, and $\gamma_{z_2 + 1}$ is the last loop in the last block with no departing function.  These loops exist by a counting argument, but we make the choice explicit.

\emph{If $\Sigma$ is ramified at $v_{14}$}, let $k$ be the smallest positive integer such that $s'_k [5] = 6$, and define
\begin{equation} \label{eq:blocksrightramified}
z_1 = \left\{ \begin{array}{ll}
6 & \textrm{if $k \geq 7$;} \\
7 & \textrm{if $k \leq 6$;}
\end{array} \right. \quad \quad \mbox{ and } \quad \quad
z_2 = \max \{ k-1,7 \} .
\end{equation}
\emph{If $\Sigma$ has extra ramification at $w_{13-g}$}, let $k$ be the largest positive integer such that $s_k [0] = -3$, and define
\begin{equation} \label{eq:blocksleftramified}
z_1 = \min \{ k, 6 \}, \quad \quad \mbox{ and } \quad \quad
z_2 = \left\{ \begin{array}{ll}
6 & \textrm{if $k \geq 8$;} \\
7 & \textrm{if $k \leq 7$.}
\end{array} \right. .
\end{equation}
Let $\psi \in \cB$ be a function that is permissible on the second block, and let $\cB' = \cB \smallsetminus \{ \psi \}$.  (In the case where $z_1 = z_2$, let $\psi \in \cB$ be a function with $s_{z_1+1} (\psi) = 3$.)

If there is a loop or bridge of positive multiplicity, then since $\rho = 1$, there is only one such loop or bridge, and it has multiplicity 1.

\medskip

\textbf{Case 2:  There is a bridge $\beta_{\ell}$ of multiplicity 1.}  If $\ell \geq 8$ and $s'_{\ell-1} [5] = 6$, then define $z_1$ and $z_2$ as in \eqref{eq:blocksrightramified}. If $\ell \leq 7$ and $s_{\ell} [0] = -3$, then define $z_1$ and $z_2$ as in \eqref{eq:blocksleftramified}. Otherwise, define
\[
z_1 = \min \{ \ell - 1, 6 \} ,  \quad \quad \mbox{ and } \quad \quad  z_2 = \ell - 1.
\]
If $\ell \geq 8$ and $s_{\ell-1} [5] = 6$ or $\ell \leq 7$ and $s_{\ell} [0] = -3$, then as above, we let $\psi \in \cB$ be a function that is permissible on the second block, and let $\cB' = \cB \smallsetminus \{ \psi \}$.  Otherwise, let $h$ be the unique integer such that $s_{\ell} [h] < s'_{\ell -1} [h]$.  If $\ell \neq 5,6$, then we will see in Lemma~\ref{lem:depart} that either there is a unique $i$ such that $s'_{\ell -1} [h] + s'_{\ell -1} [i] = s_{\ell -1} (\theta)$, or $2s'_{\ell -1} [h] = s_{\ell -1} (\theta) + 1$, but not both.  In the first case, we let $\cB' = \cB \smallsetminus \{ \varphi_{hi} \}$, and in the second case, we let $\cB' = \cB \smallsetminus \{ \varphi_{hh} \}$.  (The function in $\cB \smallsetminus \cB'$ is permissible on both blocks to either side of the bridge $\beta_{\ell}$.)  If $\ell = 5$ or 6, then we will see in Lemma~\ref{lem:depart} that there is a unique $i$ such that $s'_{\ell -1} [h] + s'_{\ell -1} [i] = s_{\ell -1} (\theta) - 1$, and we again let $\cB' = \cB \smallsetminus \{ \varphi_{hi} \}$.

\bigskip

It remains to consider the cases where there is a loop of multiplicity one. The case of a switching loop is left to the next subsection.  In the case of a lingering loop, we construct an independence exactly as in \S\ref{Sec:VA}.  (See Remark~\ref{Rmk:11and12} for an explanation of how the algorithm for $g = 13$ is adapted to the cases where $g = 11$ or $g = 12$.) We now discuss the remaining case, where there is a decreasing loop.

\medskip

\textbf{Case 3: There is a decreasing loop $\gamma_{\ell}$.}  If $\ell \geq 8$ and $s_{\ell} [5] = 6$, then define $z_1$ and $z_2$ as in \eqref{eq:blocksrightramified}. If $\ell \leq 7$ and $s'_{\ell} [0] = -3$, then define $z_1$ and $z_2$ as in \eqref{eq:blocksleftramified}. Otherwise, define

\[
z_1 = \left\{ \begin{array}{ll}
\ell & \textrm{if $\ell < 6$,} \\
5 & \textrm{if $\ell = 6$,} \\
6 & \textrm{if $\ell > 6$,}
\end{array} \right.
\quad \quad \mbox{ and } \quad \quad
z_2 = \left\{ \begin{array}{ll}
\ell - 1 & \textrm{if $\ell > 8$,} \\
8 & \textrm{if $\ell = 8$,} \\
7 & \textrm{if $\ell < 8$.}
\end{array} \right.
\]
If $\ell \geq 8$ and $s_{\ell} [5] = 6$ or $\ell \leq 7$ and $s_{\ell} [0] = -3$, then as above, we let $\psi \in \cB$ be a function that is permissible on the second block, and let $\cB' = \cB \smallsetminus \{ \psi \}$.  Otherwise, let $h$ be the unique integer such that $s'_{\ell} [h] < s_{\ell} [h]$.  If $\ell < 6$ or $\ell = 7,8$ then $\gamma_{\ell}$ is the last loop in its block, and we will see in Lemma~\ref{lem:depart} that either there is a unique $i$ such that $s_{\ell} [h] + s_{\ell} [i] = s_{\ell} (\theta)$, or $2s_{\ell} [h] = s_{\ell}(\theta) + 1$, but not both.  In the first case, we let $\cB' = \cB \smallsetminus \{ \varphi_{hi} \}$, and in the second case, we let $\cB' = \cB \smallsetminus \{ \varphi_{hh} \}$.  If $\ell > 8$ or $\ell = 6$, then we will see that either there is a unique $i$ such that $s_{\ell} [h] + s_{\ell} [i] = s_{\ell - 1} (\theta)$, or $2s_{\ell}[h] = s_{\ell -1} (\theta) + 1$.  Again, in the first case, we let $\cB' = \cB \smallsetminus \{ \varphi_{hi} \}$, and in the second case, we let $\cB' = \cB \smallsetminus \{ \varphi_{hh} \}$.

In the cases above, we asserted several times that certain functions exist with specified slopes.  To prove this, we need to generalize Proposition~\ref{Prop:VANew}.  We first define the following function.
\[
\tau (k) = \sum_{i=0}^5 (s'_k [i] +2 - i) .
\]
Note that, if there is a loop of positive multiplicity and $\gamma_{\ell}$ is the $k$th loop of multiplicity zero, then $k = \tau (\ell)$.  The following observation serves as the basis for our counting arguments.

\begin{lemma}
\label{lem:depart}
For a fixed $k$, suppose that $-2 \leq s'_k [i] \leq 5$ for all $i$.  Let $j$ be an integer such that $s'_k [j] - 1$ is not equal to $-3$ or $s'_k [i]$ for any $i$.  For $s$ in the range $2 \leq s \leq 5$, there does not exist $i$ such that $s'_k [i] + s'_k [j] = s$ if and only if one of the following holds:
\vspace{-8 pt}
\begin{multicols}{2}
\begin{enumerate}
\item  $\tau (k) = 10-s$;
\item  $s=5$, $j=0$, and $s'_k [0] = -1$;
\item  $s=2$, $j=5$, and $s'_k [5] = 5$;
\item  $2 s'_k [j] = s+1$.
\end{enumerate}
\end{multicols}
\end{lemma}

\begin{proof}
The argument is identical to that of Proposition~\ref{Prop:VANew}.
\end{proof}

There are additional relevant cases, when $s'_k [5] = 6$ or $s'_k [0] = -3$.

\begin{lemma}
\label{lem:RamifyDepart}
If $s'_k [5] = 6$, then there does not exist $i$ such that $s'_k [i] + 6 \leq 3$.  Similarly, if $s_k [0] = -3$, then there does not exist $i$ such that $s'_k [i] - 2 \geq 4$.
\end{lemma}

\begin{proof}
Since $\rho = 1$, if $s'_k [5] = 6$, then $s'_k [0] \geq -2$.  It follows that $s'_k [i] + 6 \geq 4$ for all $i$.  Similarly, if $s_k [0] = -3$, then $s'_k [i] \leq 5$ for all $i$.  It follows that $s'_k [i] - 2 \leq 3$ for all $i$.
\end{proof}

\medskip
\begin{lemma}
The set $\cB'$ satisfies conditions (i)-(vi).
\end{lemma}

\begin{proof}
\emph{Condition (i):} If $s'_k [i] \geq s_k [i]$ for all $i$, then the result is immediate, so we may assume that $\gamma_k$ is a decreasing loop.  Let $h$ be the unique integer such that $s'_k [h] = s_k [h] + 1$, and let $h'$ be the unique integer such that $s'_k [h'] = s_k [h']-1$.  If $\varphi_{hh'}$ is not permissible on $\gamma_k$, then again there is nothing to show.  If $\varphi_{hh'}$ is permissible, then by Lemma~\ref{lem:depart}, we must have $s_k (\theta) = 10-k$.  By construction, this occurs if and only if $k=7$, in which case $\varphi_{hh'} \notin \cB'$.

\medskip

\emph{Condition (ii):} Consider the first block first.  There are two functions $\psi \in \cB$ with the property that $s'_{13-g} (\psi) = 4$.  The result will therefore hold for the first block if and only if the first block contains a loop with no new permissible functions.  Let $\gamma_k$ be a loop of multiplicity zero that is contained in the first block.  By Lemmas~\ref{lem:depart} and~\ref{lem:RamifyDepart}, there is no new permissible function on $\gamma_k$ if and only if $\tau (k) = 6$ or $s'_k [0] = s_k [0] + 1 = -2$.  Thus, the number of permissible functions in $\cB$ on the first block is at most 2 more than the number of loops in Cases 2 or 3 when $\ell \leq 6$ and $s_{\ell} [0] \geq -2$, and 1 more than the number of loops in the remaining cases.
In Cases 2 and 3 when $\ell \leq 6$ and $s_{\ell} [0] \geq -2$, the function in $\cB \smallsetminus \cB'$ is permissible on the first block.  Since this function is not in $\cB'$, the number of functions in $\cB'$ that are permissible on the first block is one less than the number in $\cB$.  The third block follows from a completely symmetric argument.

For the second block, note that if $\tau (z_1) = 6$, then there are 3 functions $\psi \in \cB$ with the property that $s'_{z_1} (\psi) = 3$, and otherwise there are only two such functions.  In every case, either $\tau (z_1) < 6$ or by Lemma~\ref{lem:depart}, the second block contains a loop with no new permissible functions.  Since the function in  $\cB \smallsetminus \cB'$ is permissible on the second block, we see that the number of permissible functions on the second block is one more than the number of loops. (Note that this holds even in the case where the second block contains zero loops, in which case there is exactly one permissible function on the second block.)

\medskip

\emph{Condition (iii):} Suppose that $\varphi_{ij} \in \cB$ is permissible on more than one block.  First, consider the case where $\beta_{\ell}$ is a bridge of multiplicity one, and let $h$ be the unique integer such that $s_{\ell} [h] = s_{\ell -1} [h] -1$.  If $\varphi_{ij}$ is permissible on more than one block, then $j = h$ and either $s'_{\ell-1} [h] + s'_{\ell-1} [i] = s_{\ell -1} (\theta)$, or $i=h$ and $2s'_{\ell-1} [h] = s_{\ell -1} (\theta) + 1$.  If $-2 \leq s_{\ell} [h] \leq 5$, then by Lemma~\ref{lem:depart}, such an $i$ exists if and only if $\ell \neq 5,6$, and by construction, we have $\varphi_{hi} \notin \cB'$.  Similarly, if $s_{\ell}[h] = -3$, then by Lemma~\ref{lem:RamifyDepart}, such an $i$ exists if and only if $\ell \geq 8$, and if $s_{\ell}[h] = 5$, then such an $i$ exists if and only if $\ell \leq 7$.  In both cases, we have $\varphi_{hi} \notin \cB'$.

Next, consider the case where $\gamma_{\ell}$ is a decreasing loop.  By construction, $\gamma_{\ell}$ is either the first or last loop in its block.  Let $h$ be the unique integer such that $s'_{\ell} [h] = s_{\ell} [h] - 1$.  If $\gamma_{\ell}$ is the last loop in its block and $\varphi_{ij}$ is permissible on both the block containing $\gamma_{\ell}$ and the next block, then $j = h$ and either $s_{\ell} [h] + s_{\ell} [i] = s_{\ell} (\theta)$, or $i=h$ and $2s_{\ell} [h] = s_{\ell} (\theta) + 1$.  But then $\varphi_{ij} \notin \cB'$.  Similarly, if $\gamma_{\ell}$ is the first loop in its block, and $\varphi_{ij}$ is permissible on both the block containing $\gamma_{\ell}$ and the preceding block, then $j=h$ and either $s_{\ell} [h] + s_{\ell} [i] = s_{\ell - 1} (\theta)$, or $2s_{\ell}[h] = s_{\ell -1} (\theta) + 1$.  If $\ell \neq 7$, then again $\varphi_{ij} \notin \cB'$.  Finally, note that if $\gamma_{\ell}$ is \emph{both} the first and last loop in its block, then $\ell = 7$, and the only functions $\varphi_{ij}$ that are permissible on $\gamma_7$ satisfy $s'_{\ell} [i] + s'_{\ell} [j] = 3$.  The result follows.

\medskip

\emph{Condition (iv):} If $\gamma_{\ell}$ is a lingering loop, then we follow the construction of the vertex avoiding case of the previous subsection, in which $\gamma_{\ell}$ is the last loop in its block.

\medskip

\emph{Condition (v):}  Let $\gamma_k$ be a decreasing loop, let $h$ be the unique integer such that $s'_k [h] = s_k [h] + 1$, and let $h'$ be the unique integer such that $s'_k [h'] = s_k [h']-1$.  If $\varphi_{hh'}$ is permissible, then  $\varphi_{hh'} \notin \cB'$, as shown in the proof of condition (i).

\medskip

\emph{Condition (vi):} Let $\beta_k$ be a decreasing bridge and let $j$ is the unique value such that $s_k [j] < s'_{k-1} [j]$.  If $\beta_k$ is not a bridge between blocks, then by construction either $j=0$, $k \leq 7$, and $s_k[0] = -3$, or $j=5$, $k \geq 8$, and $s_k[5]=5$.  In both cases, by Lemma~\ref{lem:RamifyDepart}, we see that there is no $i$ such that $\varphi_{ij} \in \cB'$ is permissible on $\gamma_{k-1}$.
\end{proof}

This completes the proof of Theorem~\ref{thm:independence} in all cases where there is no switching loop for $\Sigma$.

\subsection{Switching loops} \label{sec:switching}

We now consider the case where there is a switching loop $\gamma_{\ell}$ that switches slope $h$.  This means that $s_{\ell} [i] = s'_{\ell} [i]$ for all $i$, and there exists a function $\varphi \in R(D)$ satisfying
\[
s_{\ell} (\varphi) = s_{\ell} [h], \mbox{ } s'_{\ell} (\varphi) = s'_{\ell} [h] + 1 = s'_{\ell} [h+1] .
\]
In this case, we define $z_1$ and $z_2$ as follows:
\begin{displaymath}
z_1 = \left\{ \begin{array}{ll}
\ell & \textrm{if $\ell < 6$,} \\
5 & \textrm{if $\ell = 6$,} \\
6 & \textrm{if $\ell > 6$;}
\end{array} \right.
\quad \quad \mbox{ and } \quad \quad
z_2 = \left\{ \begin{array}{ll}
7 & \textrm{if $\ell < 6$,} \\
\ell & \textrm{if $\ell \geq 6$.}
\end{array} \right.
\end{displaymath}
As in Section~\ref{Sec:VA}, we will construct our independence $\theta$ to satisfy
\begin{equation*}
s_k (\theta) = \left\{ \begin{array}{ll}
4 & \textrm{if $k \leq z_1$,} \\
3 & \textrm{if $z_1 < k \leq z_2$,}\\
2 & \textrm{if $z_2 < k \leq 13$.}
\end{array} \right.
\end{equation*}
In the preceding cases, we identified functions $\varphi_i \in \Sigma$ with designated slope $s_k (\varphi_i) = s_k [i]$ along each bridge $\beta_k$.  When there is a switching loop, this is possible for $i \neq h, h+1$, but such a function does not necessarily exist for $i = h, h+1$.  Instead, we identify a collection of functions in $\Sigma$ with designated slope along some of the bridges, and with slopes along the remaining bridges in a restricted range.

\begin{proposition}
There is a pencil $W \subseteq V$ with $\varphi_A , \varphi_B,$ and $\varphi_C$ in $\trop(W)$ such that:
\begin{enumerate}
\item  $s'_k (\varphi_A) = s'_k [h]$ for all $k < \ell$;
\item  $s_k (\varphi_B) = s_k [h+1]$ for all $k > \ell$;
\item  $s_k (\varphi_C) = s_k [h+1]$ for all $k\leq\ell$, and $s'_k (\varphi_C) = s_k [h]$ for all $k \geq \ell$;
\item   $s_k (\varphi_\bullet) \in \{ s_k [h], s_k [h+1]\}$ and $s'_k (\varphi_\bullet) \in \{ s'_k [h], s'_k [h+1]\}$ for all $k$.
\end{enumerate}
\end{proposition}

\begin{proof}
The proof is essential the same as that of \cite[Proposition~9.18]{FJP}. We include the details for completeness.  First, there exists $\varphi_A \in \Sigma$ such that $s'_{13-g} (\varphi_A) \leq s'_{13-g} [h]$, and $s_{14} (\varphi_A) \geq s_{14} [h]$.  Since $\gamma_{\ell}$ is the only switching loop, we have $s'_k (\varphi_A) \leq s'_k [h]$ for $k < \ell$, and $s'_k (\varphi_A) \geq s'_k [h]$ for $k \geq \ell$.  In particular, $s'_{\ell} (\varphi_A) \geq s'_{\ell} [h]$, so $s'_{\ell-1} (\varphi_A) \geq s'_{\ell-1} [h]$, and it follows that $s'_{\ell-1} (\varphi_A) = s'_{\ell-1} [h]$.  This proves (i), because there are no switching loops to the left of $\gamma_{\ell}$.  The proof of (ii) is similar.

We now prove (iii).  Given $\varphi_A$ and $\varphi_B$ in $\Sigma$ satisfying (i) and (ii), choose $f_A$ and $f_B \in V$ tropicalizing to $\varphi_A $ and $\varphi_B$, respectively.  Let $W$ be the pencil spanned by $f_A$ and $f_B$.  Arguments similar to the proof of (i) above show that $s_k(\trop(W)) = (s_k[h], s_k[h+1])$, for all $k$.  Choose a function $f \in W$ such that $\varphi = \trop(f)$ satisfies $s_{\ell} (\varphi) = s_{\ell} [h+1]$.  Then $s_k (\varphi) = s_k [h+1]$ for $k < \ell$.  Similarly, choose $\varphi' \in \trop(W)$ such that $s'_{\ell} (\varphi') = s'_{\ell} [h]$, which implies that $s_k (\varphi') = s_k [h]$ for $k > \ell$.  Finally, by adding a scalar to $\varphi'$, we may assume that $\varphi$ and $\varphi'$ agree on the loop $\gamma_{\ell}$, and set $\varphi_C = \min \{ \varphi, \varphi' \}$.
\end{proof}

In three steps, we now construct a tropical independence among 20 pairwise sums of functions in
\[
\mathcal{S} := \{ \varphi_i : i \neq h, h+1 \} \cup \{ \varphi_A, \varphi_B, \varphi_C \}.
\]

\subsubsection{Step 1} First, we identify a collection of simpler functions in $R(D)$ that are not necessarily in $\Sigma$.  Unlike $\varphi_A$ and $\varphi_B$, these functions are completely explicit; they have fixed slopes at every point of the graph, rather than slopes in a restricted range.  Moreover, these functions generate a tropical submodule containing $\varphi_A$, $\varphi_B$, and $\varphi_C$.

\begin{proposition}
There are functions $\varphi^0_h, \varphi^0_{h+1},$ and $\varphi^{\infty}_h$ in  $R(D)$ such that:
\begin{enumerate}[(1)]
\item  $s_k (\varphi^0_h) = s_k [h]$ and $s'_k (\varphi^0_h) = s'_k[h]$ for all $k$;
\item  $s_k (\varphi^0_{h+1}) = s_k [h+1]$ and $s'_k (\varphi^0_{h+1}) = s'_k [h+1]$ for all $k$;
\item  $s_k (\varphi^{\infty}_h) = s_k [h]$, $s'_{k-1} (\varphi^{\infty}_h) = s'_{k-1} [h]$ for all $k\leq\ell$, \\ and $s_k (\varphi^{\infty}_h) = s_k [h+1]$, $s'_{k-1} (\varphi^{\infty}_h) = s'_{k-1} [h+1]$ for all $k > \ell$.
\item  The function $\varphi_A$ is a tropical linear combination of the functions $\varphi^0_h$ and $\varphi^{\infty}_h$, where the two functions simultaneously achieve the minimum at a point to the right of $\gamma_{\ell}$.
\item  The function $\varphi_B$ is a tropical linear combination of the functions $\varphi^0_{h+1}$ and $\varphi^{\infty}_h$, where the two functions simultaneously achieve the minimum at a point to the left of $\gamma_{\ell}$.
\item  The function $\varphi_C$ is a tropical linear combination of the functions $\varphi^0_h$ and $\varphi^0_{h+1}$, where the two functions simultaneously achieve the minimum on the loop $\gamma_{\ell}$ where they agree.
\end{enumerate}
\end{proposition}

\begin{proof}
The construction of the functions is essentially the same as in \cite[Lemmas 9.8 and 9.19]{FJP}, but w describe the essential argument here, for the reader's convenience.  To construct $\varphi^{\infty}_h$, consider a function that agrees with $\varphi_A$ to the left of $\gamma_{\ell}$ and with $\varphi_B$ to the right.  Because the two functions agree on $\gamma_{\ell}$, they ``glue'' together to give a function in $R(D)$.  The construction of the other two functions is similar.  The verification that $\varphi_A$, $\varphi_B$, and $\varphi_C$ are tropical linear combinations as claimed is the same as in \cite[Lemmas 9.10 and 9.19]{FJP}.
\end{proof}

\subsubsection{Step 2} Next, we choose a set $\cB''$ of 20 pairwise sums of functions in
\[
\cA := \{ \varphi_i \colon i \neq h, h+1 \} \cup \{ \varphi^0_h , \varphi^0_{h+1}, \varphi^{\infty}_h \}.
\]
that satisfies conditions (i)-(vi) of \S\ref{Sec:Slopes}. We will choose this set so that, moreover, the independence $\theta$ produced by the algorithm from \S\ref{Sec:Slopes} satisfies a technical condition involving the best approximations of $\theta$ by certain functions in $R(D)$ that are not in the set (Lemma~\ref{Lem:B''}).

 Start with the set $\cB$ of pairwise sums of elements in $\cA \smallsetminus \{ \varphi_h^\infty \}$.
Note that $|\cB | = 21$.  As a first step toward specifying $\cB''$, we choose one function $\psi \in \cB$, of the form $\varphi_i + \varphi_j$ for $i, j \neq h, h+1$,  to exclude.   If $\ell \leq 7$, $\ell \neq 6$, let $\psi$ to such a function that is permissible on the second block.  If $\ell = 6$, let $\psi \in \cB$ to be a function that is permissible on the first block.  Otherwise, if $\ell > 7$, let $\psi \in \cB$ be a function that is permissible on the third block. This choice of $\psi$ guarantees that the number of functions in $\cB' := \cB \smallsetminus \{ \psi \}$ that are permissible on each block is one more than the number of loops in that block. In order to ensure a certain technical condition in the next step (Lemma~\ref{Lem:B''}), in the cases  where there is some $j$ such that $s'_{\ell} [h+1] + s'_{\ell} [j] = s_{\ell} (\theta) + 1$, we adjust $\cB'$ by removing one more function and replacing it with $\varphi_h^\infty + \varphi$, for some $\varphi \in \cA$.

Suppose there is some $\varphi \in \cA \smallsetminus \{ \varphi^{\infty}_h \}$ such that $s'_{\ell} [h+1] + s'_{\ell} (\varphi) = s_{\ell} (\theta) + 1$. Then we define $\cB'' := \cB \cup \{ \varphi^{\infty}_h + \varphi \} \smallsetminus \{ \varphi^0_h + \varphi \}$.  Otherwise, if there is no such $\varphi$, let $\cB'' := \cB'$.

\begin{lemma}
The set $\cB''$ satisfies conditions (i)-(vi) of \S \ref{Sec:Slopes}, and therefore the algorithm in \S \ref{Sec:Slopes} produces an independence $\theta$ among the functions in $\cB''$ with slopes $s_\ell(\theta)$ as specified above.
\end{lemma}

\begin{proof}
We first prove (i).  First, note that, for any function $\varphi \in \cA$, the functions $\varphi + \varphi^0_h , \varphi + \varphi^0_{h+1}$ have identical restrictions to the switching loop $\gamma_{\ell}$.  Because these two functions have different slopes along $\beta_{\ell}$ and $\beta_{\ell + 1}$, however, we see that they cannot both be permissible on $\gamma_{\ell}$.  In the case where $\varphi^{\infty}_h + \varphi \in \cB''$, we see that the restriction of this function to a loop $\gamma_k$ with $k \geq \ell$ agrees with that of the function $\varphi^0_{h+1} + \varphi$.  We note, however, that since $s'_{\ell} [h+1] + s'_{\ell} (\varphi) = s_{\ell} (\theta) + 1$, if $k \geq \ell$, the function $\varphi^0_{h+1} + \varphi$ is not permissible on the loop $\gamma_k$.

If $\cB'' = \cB'$, then condition (ii) holds by the same argument as Lemma~\ref{lem:countpermissible}.  Otherwise, note that the function in $\cB'' \smallsetminus \cB'$ is permissible on the same block as the function in $\cB' \smallsetminus \cB''$, so condition (ii) still holds.  Condition (iii) holds because the slopes functions in $\cA$ do not decrease from one bridge to the next.  Conditions (iv)-(vi) hold vacuously.  By Proposition~\ref{Prop:ConditionsImpyIndependent}, therefore, there is an independence $\vartheta$ among the functions in $\cB''$.
\end{proof}

\subsubsection{Step 3} Finally, we choose a set $\cT$ of 20 pairwise sums of functions in $\mathcal{S}$ and show that the best approximation of the $\theta$ by $\cT$, defined as follows, is an independence.

\begin{definition} \label{def:best-approx}
Let $\cT$ be a finite subset of $\PL(\Gamma)$.  The \emph{best approximation} of $\theta \in \PL(\Gamma)$ by $\cT$ is
\begin{equation} \label{eq:bestapprox}
\vartheta_\cT := \min \{ \varphi - c(\varphi, \theta) : \varphi \in \cT \},
\end{equation}
where $c(\varphi, \theta) = \min \{ \varphi(v) - \theta(v) : v \in \Gamma \}$.
\end{definition}

\noindent Note that $\vartheta_{\cT} \geq \theta$, and every function $\varphi \in \cT$ achieves the minimum at some point.

\begin{lemma}
\label{Lem:Replace}
Let $\theta = \min_{\psi \in \cB''} \{ \psi + a_\psi \}$. Suppose $\varphi = \min_{\psi' \in \mathcal{C}} \{ \psi' + b_{\psi'} \}$, where $\mathcal{C} \subset \cB''$.  Then the best approximation of $\theta$ by $\varphi$ achieves equality on the entire region where some $\psi' \in \mathcal{C}$ achieves the minimum in $\theta$.
\end{lemma}

\begin{proof}
Let $c = \min_{\psi' \in \mathcal{C}} \{ b_{\psi'} - a_{\psi'}\}$. Choose $\psi' \in \mathcal{C}$ such that $c = b_{\psi'} - a_{\psi'}$.  Then $\varphi - c \geq \theta$, with equality at points where $\psi'$ achieves the minimum in $\theta$.
\end{proof}

We now study the best approximation of $\theta$ by various pairwise sums of function in $\mathcal{S}$.

\begin{lemma}
\label{Lem:Approx}
Let $\varphi \in \cA \smallsetminus \{ \varphi^{\infty}_h \}$.  The best approximation of $\theta$ by $\varphi_C + \varphi$ achieves equality on the region where either $\varphi^0_h + \varphi$ or $\varphi^0_{h+1} + \varphi$ achieves the minimum.  \end{lemma}

\begin{proof}
If $\cB''$ contains both $\varphi^0_h + \varphi$ and $\varphi^0_{h+1} + \varphi$, then since $\varphi_C + \varphi$ is a tropical linear combination of these two functions, the result follows from Lemma~\ref{Lem:Replace}.  If not, then by construction $\cB''$ does not contain $\varphi^0_h + \varphi$, and $s'_{\ell} [h+1] + s'_{\ell} (\varphi) = s_{\ell} (\theta) + 1$.  In this case, $\varphi_C + \varphi$ has slope greater than $s_{\ell} (\theta)$ on $\beta_{\ell}$, so it must achieve equality to the left of $\gamma_{\ell}$, where it agrees with $\varphi^0_{h+1} + \varphi$.
\end{proof}

\begin{lemma} \label{Lem:B''}
Let $\varphi \in \cA \smallsetminus \{ \varphi^{\infty}_h \}$.  If $\varphi^{\infty}_h + \varphi \notin \cB''$, then the best approximation of $\theta$ by $\varphi^{\infty}_h + \varphi$ achieves equality on the region where either $\varphi^0_h + \varphi$ or $\varphi^0_{h+1} + \varphi$ achieves the minimum.
\end{lemma}

\begin{proof}
If $\varphi^0_h + \varphi$ is assigned to a loop $\gamma_k$ with $k < \ell$, then since $\varphi^{\infty}_h \geq \varphi^0_h$ with equality to the left of $\gamma_{\ell}$, we see that the best approximation of $\theta$ by $\varphi^{\infty}_h + \varphi$ achieves equality on the region where $\varphi^0_h + \varphi$ achieves the minimum.  Similarly, if $\varphi^0_{h+1} + \varphi$ is assigned to a loop $\gamma_k$ with $k \geq \ell$, then the best approximation of $\theta$ by $\varphi^{\infty}_h + \varphi$ achieves equality on the region where $\varphi^0_{h+1} + \varphi$ achieves the minimum.  It therefore suffices to consider the case where $\varphi^0_h + \varphi$ is not assigned to a loop $\gamma_k$ with $k < \ell$, but $\varphi^0_{h+1} + \varphi$ is.  By Lemma~\ref{lem:depart}, on every loop $\gamma_k$ in the same block as $\gamma_{\ell}$ with $k < \ell$, there is a departing function.   It follows that
\[
s_{\ell} [h+1] + s_{\ell} (\varphi) \geq s_{\ell} (\theta) + 1.
\]
Since $\varphi^0 _h + \varphi$ is not assigned to a loop $\gamma_k$ with $k < \ell$, we must have equality in the expression above.  By construction, in this case $\varphi^{\infty}_h + \varphi \in \cB''$.
\end{proof}

\begin{remark}
\label{Rmk:NotBoth}
Note that it is possible that the best approximation of $\theta$ by $\varphi_C + \varphi$ achieves equality on \emph{both} the region where $\varphi^0_h + \varphi$ achieves the minimum and the region where $\varphi^0_{h+1} + \varphi$ achieves the minimum.  However, the set of independences is open in the set of all tropical linear combinations.  In other words, if the coefficients are varied in a sufficiently small neighborhood, the result is still and independence.  One can therefore choose the independence $\theta$ to rule out this possibility.
\end{remark}

We now describe our choice of the set $\cT$.  We will define sets $\cT_j , \cT'$ below, and define
\[
\cT = \{ \varphi_{ij} \in \cB'' : i,j \neq h,h+1 \} \cup \Big( \bigcup_{j \neq h, h+1} \cT_j \Big) \cup \cT' .
\]
For $j \neq h, h+1$, if the best approximation of $\theta$ by $\varphi_C + \varphi_j$ achieves equality where $\varphi^0_h + \varphi_j$ achieves the minimum, let $\cT_j = \{ \varphi_B + \varphi_j , \varphi_C + \varphi_j \}$.  Otherwise, if the best approximation of $\theta$ by $\varphi_C + \varphi_j$ achieves equality where $\varphi^0_{h+1} + \varphi_j$ achieves the minimum, then let $\cT_j = \{ \varphi_A + \varphi_j , \varphi_C + \varphi_j \}$.

Similarly, we define $\cT'$ to be a set of three pairwise sums of elements of $\{ \varphi_A , \varphi_B , \varphi_C \}$, with our choice depending on where certain functions achieve equality in the best approximation.  In all cases, $\varphi_C + \varphi_C \in \cT'$.  The other functions in $\cT'$ are determined by the following rules:
\begin{itemize}
\item If the best approximation of $\theta$ by $\varphi_C + \varphi_C$ achieves equality at a point to the left of $\gamma_{\ell}$, then $\varphi_A + \varphi_C \in \cT'$.  Otherwise, $\varphi_B + \varphi_C \in \cT'$.
\item Suppose $\varphi_A + \varphi_C \in \cT'$. If the best approximation of $\theta$ by $\varphi_A + \varphi_C$ achieves equality at a point to the left of $\gamma_{\ell}$, then $\varphi_A + \varphi_A \in \cT'$. Otherwise $\varphi_A + \varphi_B \in \cT'$.
\item Suppose $\varphi_B + \varphi_C \in \cT'$. If the best approximation of $\theta$ by $\varphi_B + \varphi_C$ achieves equality at a point to the left of $\gamma_{\ell}$, then $\varphi_A + \varphi_B \in \cT'$.  Otherwise $\varphi_B + \varphi_B \in \cT'$.
\end{itemize}

\begin{theorem}
The best approximation $\vartheta_{\cT}$ is an independence, and $\vartheta_{\cT} = \theta$ as functions.
\end{theorem}

\begin{proof}
We show that there is a bijection $F : \cT \to \cB''$ with the property that each $\psi \in \cT$ achieves the minimum in $\vartheta_{\cT}$ on exactly the same region where $F(\psi)$ achieves the minimum in $\theta$.  From this it follows that $\vartheta_{\cT}$ is an independence, and that $\vartheta_{\cT} = \theta$.

For $i, j \neq h, h+1$, we set $F (\varphi_{ij}) = \varphi_{ij}$.  Next, consider a value $j \neq h, h+1$.  We describe the restriction of $F$ to the subset $\cT_j$.  The restriction of $F$ to $\cT'$ admits a similar description.  By Lemma~\ref{Lem:Approx}, the best approximation of $\theta$ by $\varphi_C + \varphi_j$ achieves equality on the region where either $\varphi^0_h + \varphi_j$ or $\varphi^0_{h+1} + \varphi_j$ achieves the minimum (but not both, see Remark~\ref{Rmk:NotBoth}).  If it achieves equality on the region where $\varphi^0_h + \varphi_j$ achieves the minimum, set $F ( \varphi_C + \varphi_j ) = \varphi^0_h + \varphi_j$.  Otherwise, set $F ( \varphi_C + \varphi_j ) = \varphi^0_{h+1} + \varphi_j$.

Suppose that $F(\varphi_C + \varphi_j ) = \varphi^0_{h+1} + \varphi_j$.  The case where $F( \varphi_C + \varphi_j ) = \varphi^0_h + \varphi_j$ follows from a similar (in fact, simpler) argument.  Since $\varphi_C$ agrees with $\varphi^0_{h+1}$ at points on or to the left of $\gamma_{\ell}$, we have $\varphi_A + \varphi_j \in \cT$.  If $\varphi^{\infty}_h + \varphi_j \in \cB''$, then we set $F( \varphi_A + \varphi _j ) = \varphi^{\infty}_h + \varphi_j$.  In this case, we have $s'_{\ell} [h] + s'_{\ell} [j] = s_{\ell} (\theta)$.  Since $\gamma_{\ell}$ is the last loop in its block, we see that the slope of $\varphi_A + \varphi_j$ is greater than that of $\theta$ on the right half of $\beta_{\ell + 1}$.  Thus, the best approximation of $\theta$ by $\varphi_A + \varphi_j$ must achieve equality to the left of $\beta_{\ell + 1}$, where $\varphi_A + \varphi_j$ agrees with $\varphi^{\infty}_h + \varphi_j$.

If $\varphi^{\infty}_h + \varphi_j \notin \cB''$, then set $F( \varphi_A + \varphi_j ) = \varphi^0_h + \varphi_j$, and consider the best approximation $\theta'$ of $\theta$ by $\cB'' \cup \{ \varphi^{\infty}_h + \varphi_j \}$.  Note that the coefficient of $\varphi_A + \varphi_j$ is the same in the best approximation of $\theta'$ by $\varphi_A + \varphi_j$ and the best approximation of $\theta$ by $\varphi_A + \varphi_j$.  By Lemma~\ref{Lem:B''}, $\varphi^{\infty}_h + \varphi_j$ achieves equality in $\theta'$ on the region where either $\varphi^0_h + \varphi_j$ or $\varphi^0_{h+1} + \varphi_j$ achieves the minimum in $\theta$.  Then, since $\varphi_A$ is a linear combination of $\varphi^0_h$ and $\varphi^{\infty}_h$, by Lemma~\ref{Lem:Replace}, it follows that the best approximation of $\theta$ by $\varphi_A + \varphi_j$ achieves equality on the region where either $\varphi^0_h + \varphi_j$ or $\varphi^0_{h+1} + \varphi_j$ achieves the minimum.  But $\varphi_A$ and $\varphi_C$ do not agree at any point to the left of $\gamma_{\ell}$, so the best approximation of $\theta$ by $\varphi_A + \varphi_j$ must achieve equality on the region where either $\varphi^0_h + \varphi_j$ achieves the minimum.
\end{proof}

\section{Effectivity of the virtual class} \label{sec:effectivity}

Recall that $\widetilde{\mathfrak{M}}_{13}$ is an open substack of the moduli stack of stable curves, and $\Grd$ is a stack of generalized limit linear series of rank $r$ and degree $d$ over $\widetilde{\mathfrak{M}}_{13}$.  There is a morphism of vector bundles $\phi:\mbox{Sym}^2(\E)\rightarrow \F$ over $\Grd$, whose degeneracy locus is denoted by $\fU$.

The case of Theorem~\ref{thm:independence} where $g = 13$ shows that the push forward $\sigma_*[\fU]^\vir$ under the proper forgetful map $\sigma:\Grd\rightarrow \widetilde{\mathcal{M}}_g$ is a divisor, not just a divisor class.  In our proof that $\sigma_*[\fU]^\vir$ is effective, we will use the additional cases where $g = 11$ or $12$.  Theorem~\ref{thm:independence} implies the following:

\begin{theorem}
Let $X$ be a general curve of genus $g \in \{ 11, 12, 13 \}$, and let $p \in X$ be a general point. Let $V \subseteq H^0(X,L)$ be a linear series of degree 16 and rank 5. Assume that
\begin{enumerate}
\item  if $g=12$, then $a_1^V (p) \geq 2$, and
\item  if $g=11$, then either $a_1^V (p) \geq 3$, or $a_0^V (p) + a_2^V (p) \geq 5$.
\end{enumerate}
Then the multiplication map
$
\phi_{V} \colon \Sym^2 V \to H^0 (X,L^{\otimes 2})
$
is surjective.
\end{theorem}

We now prove that $\fU$ is generically finite over each component of $\sigma_*[\fU]^\vir$, which implies that $\sigma_*[\fU]^\vir$ is effective.  Our argument follows closely that of \cite[\S~10]{FJP}.  Indeed, several of the lemmas and propositions along the way are identical, and we omit those proofs.  As in \cite[\S~10]{FJP}, we suppose that $Z \subseteq \overline{\mathcal{M}}_{13}$ is an irreducible divisor and that $\sigma \vert_{\fU}$ has positive dimensional fibers over the generic point of $Z$.  Let $\jmath_2 \colon \mm_{2,1} \rightarrow \mm_{13}$ be the map obtained by attaching an arbitrary pointed curve of genus $2$ to a fixed general pointed curve $(X,p)$ of genus 11.  Since $g = 13$ is odd, by \cite[Proposition~2.2]{FJP}, it suffices to show the following:
\begin{enumerate}[(1)]
\item  $Z$ is the closure of a divisor in $\mathcal{M}_{13}$,
\item  $j_2^* (Z) = 0$, and
\item  $Z$ does not contain any codimension 2 stratum $\Delta_{2,j}$.
\end{enumerate}

The only irreducible boundary divisors in $\widetilde{\mathcal{M}}_{13}$ are $\Delta_0^{\circ}$ and $\Delta_1^{\circ}$.  Therefore, item (1), that $Z$ is the closure of a divisor in $\mathcal{M}_{13}$, is a consequence of the following.

\begin{proposition}
The image of the degeneracy locus $\fU$ does not contain $\Delta_0^{\circ}$ or $\Delta_1^{\circ}$.
\end{proposition}

\begin{proof}
The proof is identical to \cite[Proposition~10.3]{FJP}.
\end{proof}

The proofs of (2) and (3) use the following lemma.

\begin{lemma}
\label{lem:ramify}
If $[X] \in Z$ and $p \in X$, then there is a linear series  $V \in G^5_{16} (X)$ that is ramified at $p$ such that $\phi_V$ is not surjective.
\end{lemma}

\begin{proof}
The proof is identical to \cite[Lemma~10.4]{FJP}.
\end{proof}

\subsection{Pulling back to $\mm_{2,1}$}
In order to verify (2), we consider the preimage of $Z$ under the map $\jmath_2$.

\begin{lemma}
The preimage $\jmath_2^{-1}(Z)$ is contained in the Weierstrass divisor $\overline{\mathcal{W}}_2$ in $\mm_{2,1}$.
\end{lemma}

\begin{proof}
The proof is identical to \cite[Lemma~10.5]{FJP}.
\end{proof}

To prove that $\jmath_2^*(Z) = 0$, we consider the following construction.  Let $\Gamma$ be a chain of 13 loops with the following restrictions on edge lengths:
\begin{enumerate}
\item  $m_2 = \ell_2$ (that is, the second loop has torsion index 2),
\item  $n_3 \gg n_2$, and
\item  $\ell_{k+1} \ll m_k \ll \ell_k \ll n_{k+1} \ll n_k$ for all $k \neq 2$.
\end{enumerate}
The last condition says that, subject to the constraints of conditions (i) and (ii), the edge lengths otherwise satisfy \eqref{eq:admissible}.  Let $X$ be a smooth curve of genus 13 over $K$ whose skeleton is $\Gamma$.  We first note the following.

\begin{lemma}
\label{Lem:j2}
If $[X] \notin Z$, then $\jmath_2^*(Z) = 0$.
\end{lemma}

\begin{proof}
This proof is identical to the first part of the proof of \cite[Proposition~10.6]{FJP}. \end{proof}

\begin{proposition}
We have $\jmath_2^*(Z) = 0$.
\end{proposition}

\begin{proof}
By Lemma~\ref{Lem:j2}, it suffices to show that $[X] \notin Z$.  We divide $\Gamma$ into two subgraphs $\widetilde \Gamma'$ and $\widetilde \Gamma$, to the left and right, respectively, of the midpoint of the long bridge $\beta_3$.  Let $q \in X$ be a point specializing to $v_{14}$.  If $[X] \in Z$, by Lemma~\ref{lem:ramify} there is a linear series in the degeneracy locus over $X$ that is ramified at $q$.  We now show that this impossible.

Let $\ell=\bigl( L, V\bigr)\in G^5_{16}(X)$ be a linear series ramified at $q$.  We may assume that $L = \mathcal{O}(D_X)$, where $D = \Trop (D_{X})$ is a break divisor, and consider $\Sigma = \trop (V)$.  We will show that there are 20 tropically independent pairwise sums of functions in $\Sigma$ using a variant of the arguments in \S~\ref{Sec:Construct}.  It follows that the multiplication map $\phi_\ell$ is surjective, and hence $[X]$ cannot be in $Z$.

To produce 20 tropically independent pairwise sums of functions in $\Sigma$, following the methods of \S~\ref{Sec:Construct}, we first consider the slope sequence along the long bridge $\beta_3$.  First, suppose that either $s_3 [4] \leq 2$ or $s_3 [3] + s_3 [5] \leq 5$.  In this case, even though the restriction of $\Sigma$ to $\widetilde \Gamma$ is not the tropicalization of a linear series on a pointed curve of genus 11 with prescribed ramification, it satisfies all of the combinatorial properties of the tropicalization of such a linear series.  The proof of Theorem~\ref{thm:independence} then goes through verbatim, yielding a tropical linear combination of 20 functions in $\Sigma$ such that each function achieves the minimum uniquely at some point of $\widetilde \Gamma \subseteq \Gamma$.

For the remainder of the proof, we therefore assume that $s_3[4] \geq 3$ and $s_3 [3] + s_3 [5] \geq 6$.  Since $\deg D_{\vert {\widetilde \Gamma'}} = 5$, we see that $s_3[5] \leq 5$.  Moreover, since the divisor $D_{\vert {\widetilde \Gamma'}} - s_3[4]w_2$ has positive rank on $\widetilde \Gamma'$, and no divisor of degree 1 on $\widetilde \Gamma'$ has positive rank, $s_3[4]$ must be exactly 3.  Since the canonical class is the only divisor class of degree 2 and rank 1 on $\widetilde \Gamma'$, we see that $D_{\vert{\widetilde \Gamma'}} \sim K_{\widetilde \Gamma'} + 3w_2$.  This yields an upper bound on each of the slopes $s_3[i]$, and these bounds determine the slopes for $i \geq 2$:
\[
s_3[5] = 5, \ s_3[4] = 3, \ s_3 [3] = 1, \ s_3 [2] = 0.
\]
Moreover, we must have $s'_2 [i] = s_3 [i]$ for $2 \leq i \leq 5$.   Since $\ell$ is ramified at $q$, we also have $s_{14}[5] \geq 6$. These conditions together imply that the sum of the multiplicities of all loops and bridges on $\widetilde \Gamma$ is at most 1.

To construct an independence on $\Gamma$, we first construct an independence among 5 functions on $\widetilde \Gamma'$.  This is done exactly as in \cite[Figure~26]{FJP}, and we omit the details.

Next, we construct an independence among 15 pairwise sums of functions in $\Sigma$ restricted to $\widetilde \Gamma$, with the property that any function $\psi$ that obtains the minimum on $\widetilde \Gamma$ satisfies $s'_2 (\psi) \leq 4$.  Note that each of the functions $\psi$ that obtains the minimum on $\widetilde \Gamma'$ satisfies $s_3 (\psi) \geq 5$.  Since the bridge $\beta_3$ is very long, it follows that no function that obtains the minimum on one of the two subgraphs can obtain the minimum on the other.  Thus, we have constructed a tropical linear combination of 20 pairwise sums of functions in $\Sigma$ in which 5 achieve the minimum uniquely at some point of $\widetilde \Gamma'$ and 15 achieve the minimum uniquely at some point of $\widetilde \Gamma$.  In particular, this is an independence, as required.

It remains to construct an independence among 15 pairwise sums of functions in $\Sigma$ restricted to $\widetilde \Gamma$.  To do this, we run the algorithm from \cite{FJP}, with one change.  (Indeed, one could imagine that $\Gamma$ is simply the first 13 loops in a chain of 23 loops.  We construct the independence from \cite[\S~10.3]{FJP}, and restrict it to $\Gamma$.)  At the start, we skip the step named ``Start at the First Bridge''.  Instead, we do not assign any function $\psi$ with $s_3 (\psi) \geq 5$, and we start with the Loop Subroutine applied to $\gamma_3$.  Following this construction, there will only be two blocks, and there will be two functions with slope 2 along the last bridge $\beta_{14}$.  We eliminate one of these functions from $\cB$, and assign the other to $\beta_{14}$.  The rest of the argument is exactly the same as that of \cite{FJP}.
\end{proof}

\subsection{Higher codimension boundary strata}

It remains to verify (3), that $Z$ does not contain any of the codimension 2 boundary strata $\Delta_{2,j} \subseteq \mm_{13}$.

\begin{proposition}
The component $Z$ does not contain any codimension 2 stratum $\Delta_{2,j}$.
\end{proposition}

\begin{proof}
The proof is again a variation on the independence constructions from the proof of Theorem~\ref{thm:independence}.  We fix $\ell = 11-j$.  Let $Y_1$ be a smooth curve of genus 2 over $K$ whose skeleton $\Gamma_1$ is a chain of 2 loops with bridges, and let $p \in Y_1$ be a point specializing to the right endpoint of $\Gamma_1$.  Similarly, let $Y_2$ and $Y_3$ be smooth curves of genus $\ell$ and $j$, respectively, whose skeletons $\Gamma_2$ and $\Gamma_3$, are chains of $\ell$ loops and $j$ loops with edge lengths satisfying \eqref{eq:admissible}.  Suppose further that the edges in the final loop of $\Gamma_2$ are much longer than those in the first loop of $\Gamma_3$.  Let $p, q \in Y_2$ be points specializing to the left and right endpoints of $\Gamma_2$, respectively, and let $q \in Y_3$ be a point specializing to the left endpoint of $\Gamma_3$.  We show that  $[Y'] = [Y_1 \cup_{p} Y_2 \cup_{q }Y_3] \in \Delta_{2,j}$ is not contained in $Z$.

As in the proof of \cite[Proposition~10.6]{FJP}, if $[Y'] \in Z$, then $Z$ contains points $[X]$ corresponding to smooth curves whose skeletons are arbitrarily close to the skeleton of $Y'$ in the natural topology on $\M_{13}^{\trop}$.  In particular, there is an $X \in Z$ with skeleton a chain of loops $\Gamma_X$ whose edge lengths satisfy all the conditions of \eqref{eq:admissible}, except that the bridges $\beta_3$ and $\beta_{\ell}$ are exceedingly long in comparison to the other edges.  Let $\Gamma$ be the subgraph of $\Gamma_X$ to the right of the midpoint of the bridge $\beta_3$. Note that $\Gamma$ is a chain of 11 loops, labeled $\gamma_3, \ldots, \gamma_{13}$, with bridges labeled $\beta_3, \ldots, \beta_{14}$.

By Lemma~\ref{lem:ramify}, there is a linear series $V$ of degree $16$ and rank $5$ on $X$ that is ramified at a point $x$ specializing to the righthand endpoint $v_{14}$, and such that $\phi_{V}$ is not surjective.  We will show that this is not possible, using the tropical independence construction from \S~\ref{Sec:Construct}.  Let $\Sigma = \trop(V)$.  We have that either $s'_2 [4] \leq 2$ or $s'_2 [3] + s'_2 [4] \leq 5$.  Also, since $V$ is ramified at $x$, we have $s_{14}[5] \geq 6$.  These conditions imply that the multiplicity of every loop and bridge is zero.  In particular, for each $i$ there is a function $\varphi_i$ satisfying
\[
s_k (\varphi_i ) = s'_{k-1} (\varphi_i ) = s_k [i] = s'_{k-1}[i] \mbox{ for all } k.
\]
These functions have constant slope along bridges, and the slopes $s_k (\varphi_i)$ are nondecreasing in $k$.  These properties guarantee that,
even though the bridge $\beta_{\ell}$ is very long, a function $\varphi_{ij}$ can only obtain the minimum on a loop or bridge where it is permissible.

Even though the restriction of $\Sigma$ to $\Gamma$ is not the tropicalization of a linear series on a curve of genus 11 with prescribed ramification at two specified points specializing to the left and right endpoints of $\Gamma$, it satisfies all of the combinatorial properties of the tropicalization of such a linear series, and we may apply the algorithm from \S~\ref{Sec:Construct}.  Because we are in a situation where the relative lengths of the bridges do not matter (Remark~\ref{Rmk:BridgeLength}) the construction yields an independence among 20 pairwise sums of functions in $\Sigma$, and the proposition follows.
\end{proof}

\section{The Bertram-Feinberg-Mukai Conjecture in genus $13$}

The aim of this section is to prove the existence part of the Bertram-Feinberg-Mukai conjecture on $\mm_{13}$.
For a smooth curve $X$ of genus $g$, we denote by $SU_X(2,\omega)$ the moduli space of $S$-equivalence classes of semistable rank $2$ vector bundles $E$ on $X$ with $\mbox{det}(E)\cong \omega_X$.  For an integer $k\geq 0$, the Brill-Noether locus
\[
SU_X(2,\omega, k):=\bigl\{E\in SU_X(2, \omega): h^0(X,E)\geq k\bigr\}
\]
has the structure of a Lagrangian degeneracy locus and each component of $SU_X(2,\omega, k)$ has dimension at least $\beta(2,g,k)=3g-3-{k+1\choose 2}$, see \cite{Mu}. Furthermore, $SU_X(2,\omega,k)$ is smooth of dimension $\beta(2, g, k)$ at a point $[E]$ corresponding to a stable vector bundle if and only if the Mukai-Petri map (\ref{eq:mukaipetri}) is injective. Of particular interest to us is the case
\[g=13 \ \mbox{ and } \  k=8,\]
in which case $\beta(2,13,8)=0$. First, using linkage methods, we show that a general curve of genus $13$ carries a stable vector bundle $E\in SU_X(2,\omega, 8)$. Then using a Hecke correspondence, we compute the fundamental class of $SU_X(2,\omega,8)$.

\begin{theorem}\label{thm:existence}
A general curve $X$ of genus $13$ caries a stable vector bundle $E$ of rank $2$ with $\mathrm{det}(E)\cong \omega_X$ and $h^0(X,E)=8.$
\end{theorem}

As a first step towards proving Theorem \ref{thm:existence}, we determine the extension type of the vector bundles in question.

\begin{proposition}\label{prop:extension}
For a general curve  $X$ of genus $13$, every vector bundle $E\in SU_X(2,\omega, 8)$ can be represented  as an extension
\begin{equation}\label{eq:ext}
0\longrightarrow \OO_X(D)\longrightarrow E\longrightarrow \omega_X(-D)\longrightarrow 0,
\end{equation}
where $D$ is an effective divisor of degree $6$ on $X$, such that $L:=\omega_X(-D)\in W^6_{18}(X)$ is very ample and the map $\phi_L\colon \mathrm{Sym}^2 H^0(X,L)\rightarrow H^0(X,L^{\otimes 2})$ is not surjective. Conversely, a very ample $L\in W^6_{18}(X)$ with $\phi_L$ not surjective induces a stable vector bundle $E\in SU_X(2,\omega, 8)$.
\end{proposition}

\begin{proof}
Using a result of Segre, see \cite{Na} or  \cite[Proposition 3.1]{LN} for modern proofs, every semistable vector bundle $E$ on $X$ of rank $2$ and canonical determinant carries a line subbundle $\OO_X(D)\hookrightarrow E$ with $\mbox{deg}(D)\geq \frac{g-2}{2}$. Therefore, in our case  $\mbox{deg}(D)\geq 6$.

If $h^0(X,\OO_X(D))\geq 2$, since $h^0(X,\OO_X(D))+h^0(X, \omega_X(-D))\geq h^0(X,E)=8$ it follows from the Brill-Noether Theorem and  Riemann-Roch that $\mbox{deg}(D)=8$, hence $\omega_X(-D)\in W^5_{16}(X)$. It follows that the extension (\ref{eq:ext}) lies in the kernel of the map $$\mbox{Ext}^1(\omega_X(-D), D)\rightarrow H^0(\omega_X(-D))^{\vee}\otimes H^1(D).$$ This implies that the multiplication map $\phi_{\omega_X(-D)}\colon \mathrm{Sym}^2 H^0(X, \omega_X(-D))\rightarrow H^0(X, \omega_X^{\otimes 2}(-2D))$ is not surjective, which contradicts Theorem \ref{thm:strongmrc}. Therefore $h^0(X, \OO_X(D))=1$, in which case necessarily $\mbox{deg}(D)=6$ and  $h^0(X,E)=h^0(X, \OO_X(D))+h^0(X,\omega_X(-D))$. Setting $L:=\omega_X(-D)\in W^6_{18}(X)$, an extension $E$ satisfies $h^0 (X,E) = 8$ if and only if the extension class of $E$ in $\mathrm{Ext}^1 (L,D)$ lies in the kernel of the linear map
\[
\mathrm{Ext}^1 (L,D) \to H^0 (L)^{\vee} \otimes H^1 (D) .
\]
Thus, an extension (\ref{eq:ext}) exists if and only if the multiplication map
\begin{equation*}
\phi_L\colon \mathrm{Sym}^2 H^0(L)\rightarrow H^0(X,L^{\otimes 2}) \cong \mathrm{Ext}^1 (L,D)^{\vee}
\end{equation*}
is not surjective.
We claim that $L$ is very ample. Otherwise, there exist points $x,y\in X$ such that $L':=L(-x-y)\in W^5_{16}(X)$. Since $X$ is general,   by Theorem \ref{thm:strongmrc} the multiplication map $\phi_{L'}\colon \mbox{Sym}^2 H^0(X,L')\rightarrow H^0\bigl(X,(L')^{\otimes 2}\bigr)$ is surjective, implying the inclusion $H^0\bigl(X, (L')^{\otimes 2}(x+y)\bigr)\subseteq \mbox{Im}(\phi_L)$. We deduce that $[E]$ lies in the kernel of the map
\[
\mbox{Ext}^1(L, D)\rightarrow \mbox{Ext}^1(L(-x-y),D).
\]
That is, the vector bundle $E$ can also be represented as an extension
\[
0\longrightarrow L(-x-y)\longrightarrow E\longrightarrow \OO_X(D+x+y)\longrightarrow 0,
\]
thus contradicting the semistability of $E$. We conclude that $L$ has to be  very ample.

\vskip 3pt

Conversely, each \emph{very ample} linear system $L\in W^6_{18}(X)$, for which the map $\phi_L$ is not surjective induces a stable vector bundle $E$; see also \cite[7.2]{CGZ}. Indeed, let us assume $E$ is not semistable. In view of the extension (\ref{eq:ext}), a maximally destabilizing line subbundle of $E$ is of the form $L(-M)$, where $M$ is an effective divisor on $X$ with $\mbox{deg}(M)\leq 6$. Therefore, apart from (\ref{eq:ext}),  $E$ can also be realized as an extension
\[
0\longrightarrow L(-M) \longrightarrow E \longrightarrow \OO_X(D+M)\longrightarrow 0.
\]
By applying Riemann-Roch, one can then write
\[
h^0\bigl(X, L(-M)\bigr)+h^1\bigl(X, L(-M)\bigr)=h^0(X,L)+h^1(X,L)-2\ \mathrm{dim} \ \frac{H^0(X,L)}{H^0(X,L(-M))}+\mathrm{deg}(M).
\]
Since
\[
h^0(X,L)+h^1(X,L)=h^0(X,E)\leq h^0\bigl(X, L(-M)\bigr)+h^1\bigl(X, L(-M)\bigr),
\]
it follows that
\[
\mbox{deg}(M)\geq 2\ \mathrm{dim} \frac{H^0(L)}{H^0(L(-M))}.
\]
Since $L$ is very ample, we find $\mbox{deg}(M)\in \{4,5,6\}$. In each case,  the Brill-Noether number
of  $L(-M)$ is negative, contradicting the generality of $X$. Therefore $E$ is stable.
\end{proof}

\vskip 4pt

\begin{proof}[Proof of Theorem \ref{thm:existence}.]
By Proposition \ref{prop:extension}, it suffices to show that for a general curve $X$ of genus $13$, there exists a very ample linear system $L\in W^6_{18}(X)$ such that $\phi_L$ is not surjective.  We use a method inspired by Verra's proof of the unirationality of $\mm_{14}$ \cite{Ve}.  To illustrate the idea behind the proof, first suppose that there exists an embedding $\varphi_L\colon X\hookrightarrow \PP^6$ given by $L\in W^6_{18}(X)$, such that the map $\phi_L$ is not surjective. In particular, $X\subseteq \mathbb P^6$ lies on at least $5={8\choose 2}-h^0(X, L^{\otimes 2})-1$ quadrics. We expect the base locus of this system of quadrics to be a reducible curve (of degree $32$), containing $X$ as a component and write accordingly
\begin{equation*}
X+C =\mathrm{Bs}\ \bigl|\mathcal{I}_{X/\PP^6}(2)\bigr|.
\end{equation*}
Assuming that $X$ and $C$ intersect transversally, we obtain that $X+C$ is a complete intersection curve in $\PP^6$. Therefore  $C$ is a  curve of degree $14=2^5-\mbox{deg}(X)$ and applying the adjunction formula  $2g(X)-2g(C)=(10-7)\bigl(\mbox{deg}(X)-\mbox{deg}(C)\bigr)=12$ (see for instance
\cite[p.~1429]{Ve}), we obtain $g(C)=7$.

\vskip 4pt

We now reverse this procedure and start with a general curve $C\subseteq \PP^6$ of genus $7$ embedded by a $7$-dimensional linear system $V\subseteq H^0(C, L_C)$, where $L_C\in \mbox{Pic}^{14}(C)$ is a general line bundle, therefore $h^0(C,L_C)=8$. Consider the multiplication map
\[
\phi_{V}\colon \mbox{Sym}^2(V)\rightarrow H^0(C, L_C^{\otimes 2})
\]
and observe that $\mbox{Ker}(\phi_{V})$ has dimension  at least $6=\mbox{dim } \mbox{Sym}^2(V)-h^0(L_C^{\otimes 2})$. Choose a general $5$-dimensional system of quadrics $W\in G\bigl(5, H^0(\PP^6, \I_{C/\PP^6}(2))\bigr)$. We then expect
\begin{equation}\label{eq:X}
\mathrm{Bs}\ |W|=C+X\subseteq \PP^6
\end{equation}
to be a nodal curve, and the curve $X$ linked to $C$ to be a smooth curve of degree $18$ and genus $13$. Setting $L:=\OO_X(1)\in W^6_{18}(X)$, by construction $L$ is very ample and the embedded curve $X\subseteq \PP^6$ lies on at least $5$ quadrics, therefore  $\phi_L$ is not surjective.

\vskip 4pt

To carry this out, one needs to check some transversality statements. Let $\mathcal{P}ic_7^{14}$ be the universal Picard variety parametrizing pairs $[C, L_C]$, where $C$ is a smooth curve of genus $7$ and $L_C\in \mbox{Pic}^{14}(C)$. As pointed out in \cite[Theorem 1.2]{Ve}, it follows from Mukai's work \cite{Mu1} that $\mathcal{P}ic_7^{14}$ is unirational. We introduce the variety:
\begin{align*}
\mathcal{Y}:=\Bigl\{[C,L_C, V, W]: [C, L_C]\in \mathcal{P}ic_7^{14},\  V\in G\bigl(6, H^0(C,L_C)\bigr),  W\in G\bigl(5, \mbox{Ker}(\phi_V)\bigr)\Bigr\}
\end{align*}
The forgetful map $\mathcal{Y}\rightarrow \mathcal{P}ic_7^{14}$ has the structure of an iterated locally trivial projective bundle over $\mathcal{P}ic_7^{14}$, therefore  $\mathcal{Y}$ is unirational as well. Moreover,
\[
\mbox{dim}(\mathcal{Y})=\mbox{dim}(\mathcal{P}ic_7^{14})+\mbox{dim } G(7,8)+\mbox{dim } G(5,6)=4\cdot 7-3+7+5=37.
\]
One has a rational \emph{linkage map}
\[
\chi\colon \mathcal{Y}\dashrightarrow \mathcal{SU}_{13}(2,\omega,8), \ \ \ [C, L_C, V, W]\mapsto [X, L, E],
\]
where $X$ is defined by (\ref{eq:X}), $L:=\OO_X(1)\in W^6_{18}(X)$ and $E\in SU_X(2,\omega, 8)$ is the rank $2$ vector bundle defined uniquely by the extension $0\rightarrow \omega_X\otimes L^{\vee}\rightarrow E\rightarrow L\rightarrow 0$.

\vskip 4pt

To show that $\chi$ is well-defined it suffices to produce one example of a point in $\mathcal{Y}$ for which all these assumptions are realized. To that end, we  consider $11$ general points  $p_1, \ldots, p_5$ and $q_1,\ldots, q_6$ respectively in $\PP^2$ and the
linear system
\[
H\equiv 6h-2(E_{p_1}+\cdots + E_{p_5})-(E_{q_1}+\cdots + E_{q_6})
\]
 on the blow-up $S=\mbox{Bl}_{11}(\PP^2)$ at these points. Here $h$ denotes the
pullback of the line class from $\PP^2$. Via \emph{Macaulay2} one checks that $S\stackrel{|H|}\hookrightarrow
\PP^6$ is an embedding and the graded Betti diagram of $S$ is the following:
\[
\begin{matrix}
 1 & -  & -  & - & -  \\
 - & 5 & - & - &- \\
 - & -  & 15  & 16& 15
\end{matrix}
\]
 Next we consider a general curve $C\subseteq S$ in
the linear system
\[
C\equiv 10h-4(E_{p_1}+E_{p_2}+E_{p_3}+E_{p_4})-3 E_{p_5}-2(E_{q_1}+E_{q_2})-(E_{q_3}+E_{q_4}+E_{q_5}+E_{q_6}).
\]
Via \emph{Macaulay2}, we verify that $C$ is
smooth, $g(C)=7$ and $\mbox{deg}(C)=14$.
Furthermore, using that $H^1\bigl(\PP^6, \I_{S/\PP^6}(2)\bigr)=0$, we have an exact sequence
\[
0\longrightarrow H^0\bigl(\PP^6, \I_{S/\PP^6}(2)\bigr)\longrightarrow H^0\bigl(\PP^6, \I_{C/\PP^6}(2)\bigr)\longrightarrow
H^0\bigl(S, \OO_S(2H-C)\bigr)\longrightarrow 0.
\]
Since $\OO_S(2H-C)=\OO_S(2h-E_{p_5}-E_{q_3}-E_{q_4}-E_{q_5}-E_{q_6})$, clearly $h^0\bigl(S, \OO_S(2H-C)\bigr)=1$, therefore $h^0\bigl(\PP^6, \I_{C/\PP^6}(2)\bigr)=6$.  That is, $C\subseteq \PP^6$ is a $2$-normal curve.

One also verifies with \emph{Macaulay2} that $C\subseteq \PP^6$ is scheme-theoretically cut out by quadrics.  Using \cite[Proposition 2.2]{Ve}, $C$ lies on a smooth surface $Y\subseteq \PP^6$ which is a complete intersection of $4$ quadrics containing $C$. Furthermore, the linear system $\bigl|\OO_Y(2H-C)\bigr|$ is base point free, so a general element $X\in \bigl|\OO_Y(2H-C)\bigr|$ is a smooth curve of genus $13$ meeting $C$ transversally. Finally, a standard argument using the exact sequence $0\rightarrow \OO_Y(H-X)\rightarrow \OO_Y(H)\rightarrow \OO_X(H)\rightarrow 0$ shows that since $C$ is $2$-normal, the residual curve $X$ is $1$-normal.  That is, $h^1\bigl(X, \OO_X(1)\bigr)=1$. This implies that the map $\chi \colon \mathcal{Y}\dashrightarrow \mathcal{SU}_{13}(2, \omega, 8)$ is well-defined and dominant.
\end{proof}

\begin{corollary}\label{cor:unirational}
The parameter space $\mathcal{SU}_{13}(2, \omega,8)$ is unirational.
\end{corollary}

\begin{proof}
This follows from the proof of Theorem \ref{thm:existence} and from the unirationality of $\mathcal{Y}$.
\end{proof}

\subsection{The fundamental class of $SU_X(2,\omega,8)$ for a general curve.}

It is essential for our calculations to determine the degree of the map
\[
\vartheta \colon \mathcal{SU}_{13}(2, \omega,8)\rightarrow \cM_{13}, \ \ \  \vartheta\bigl([X,E]\bigr)=[X].
\]
We fix a general curve $X$ of genus $g$ and a point $p\in X$. Since the moduli space $SU_X(2,\omega)$ is singular, in order to determine the fundamental class of the non-abelian Brill-Noether locus $SU_X(2,\omega, k)$, following \cite{Ne}, \cite{LN}, \cite{Mu} one uses instead the Hecke correspondence relating $SU_X(2,\omega)$ to the smooth moduli space $SU_X(2,\omega(p))$ of stable rank $2$ vector bundles $F$ on $X$ with $\mbox{det}(F)\cong \omega_X(p)$.

Recall that $SU_X(2, \omega(p))$ is a fine moduli space. Hence there is a universal rank $2$ vector bundle $\F$ on $X\times SU_X(2, \omega(p))$ and we consider the \emph{Hecke correspondence}
\[
\PPP:=\PP\Bigl(\F_{|\{p\}\times SU_X(2,\omega(p))}\Bigr),
\]
endowed with the projection $\pi_1\colon \PPP\rightarrow SU_X(2,\omega(p))$. The points of $\PPP$ are exact sequences
\begin{equation}\label{eq:hecke}
0\longrightarrow E\longrightarrow F\longrightarrow  K(p)\longrightarrow 0,
\end{equation}
where $F\in SU_X(2,\omega(p))$, and therefore $\mbox{det}(E)\cong \omega_X$. One has a diagram
\[
\xymatrix{
  & \PPP \ar[dl]_{\pi_1} \ar[dr]^{\rho} & \\
   SU_X(2,\omega(p)) & & SU_X(2,\omega)       \\
                 }
\]
where $\rho$ assigns to a sequence (\ref{eq:hecke}) the semistable vector bundle $E$ .
Set
\[
h:=c_1\bigl(\OO_{\PPP}(1)\bigr)=\rho^*c_1(\L_{\mathrm{ev}}),
\]
where $\L_{\mathrm{ev}}$ is the determinant line bundle on $SU_X(2,\omega)$, associated to the effective divisor
\[
\Theta:=\bigl\{E\in SU_X(2,\omega): H^0(X,E)\neq 0\bigr\}.
\]
Set $\alpha:=c_1\bigl(\L_{\mathrm{odd}}\bigr)\in H^2\bigl(SU_X(2,\omega(p)),\mathbb Z\bigr)$, where
$\L_{\mathrm{odd}}$ is the ample generator of $\mbox{Pic}\bigl(SU_X(2,\omega(p))\bigr)$. Note that $\mbox{Pic}(\PPP)$
is generated by $h$ and by $\pi_1^*(\alpha)$.

\vskip 3pt

For each $k\in\mathbb N$, the non-abelian Brill-Noether locus
\[
B_{\PPP}(k):=\Bigl\{\bigl[0\rightarrow E\rightarrow F\rightarrow  K(p)\rightarrow 0\bigr]\in \PPP: h^0(X,E)\geq k\Bigr\}
\]
has the structure of a Lagrangian degeneracy locus of expected codimension $\beta(2,g,k)+1=3g-2-{k+1\choose 2}$, see \cite[Section 5]{Mu} and \cite[Section 2]{LN}. As such, its virtual class $[B_{\PPP}(k)]^{\vir}\in H^*(\PPP, \mathbb Q)$ can be computed in terms of certain tautological classes, whose definition we recall now.

\vskip 3pt

Following \cite{Ne}, we consider the K\"unneth decomposition of the Chern classes of $\F$, using that $\mbox{det}(\F)\cong \omega_X(p)\boxtimes \L_{\mathrm{odd}}$ and  write
\begin{equation*}
c_1(\F)=\alpha+(2g-1)\varphi \ \mbox{ and } \ c_2(\F)=\chi+\psi+g \alpha\otimes \varphi,
\end{equation*}
where $\varphi\in H^2(X,\mathbb Q)$ is the fundamental class of the curve, $\chi \in  H^4\bigl(SU_X(2,\omega(p)), \mathbb Q\bigr)$
and $\psi$ is in $H^3\bigl(SU_X(2,\omega(p)), \mathbb Q\bigr)\otimes H^1(X,\mathbb Q)$. Finally, we define the class
\[
\gamma \in H^6\bigl(SU_X(2,\omega(p)),\mathbb Q\bigr)
\]
by the formula
$\psi^2=\gamma\otimes \varphi $. One has the relation
\[
h^2=\alpha h-\frac{\alpha^2-\beta}{4}\in H^4(\PPP,\mathbb Q),
\]
from which we can recursively determine all powers of $h$. We summarize as follows:

\begin{proposition}\label{prop:powers_of_h}
For each $n\geq 2$, the following relation holds in $H^*(\PPP,\mathbb Q)$:
\begin{equation*}
h^n=\frac{h(-2\alpha+2h)\sqrt{\beta}+\alpha^2-2\alpha h+\beta}{\sqrt{\beta}(\alpha^2-\beta)} \Bigl(\frac{\alpha+\sqrt{\beta}}{2}\Bigr)^n
 +\frac{h(2\alpha-2h)\sqrt{\beta}+\alpha^2-2\alpha h+\beta}{\sqrt{\beta}(\alpha^2-\beta)} \Bigl(\frac{\alpha-\sqrt{\beta}}{2}\Bigr)^n.
\end{equation*}
\end{proposition}

\vskip 3pt

In this formula $\sqrt{\beta}$ is a formal root of the class $\beta$.
Applying \cite[Section 3]{LN} or \cite{Mu} one can endow $B_{\PPP}(k)$ with the structure of a Lagrangian degeneracy locus as follows. Let $\E$ be the vector bundle on $X\times \PPP$ defined by the following exact sequence:
\[
0\longrightarrow \E\longrightarrow \bigl(\mathrm{id}\times \pi_1\bigr)^*(\F)\rightarrow (p_2)_*\bigl(\OO_{\PPP}(1)\bigr) \longrightarrow 0,
\]
where $p_2\colon X\times \PPP\rightarrow \PPP$ is the projection. Choose an effective divisor $D$ of large degree on $X$ and also denote by $D$ its pull-back under $X\times \PPP \rightarrow X$. Then $(p_2)_*\bigl(\E/\E(-D)\bigr)$ and $(p_2)_*(\E(D))$ are Lagrangian subbundles
of $(p_2)_*\bigl(\E(D)/\E(-D)\bigr)$. For each point $t:=[0\rightarrow E\rightarrow F\rightarrow   K(p)\rightarrow 0]\in \PPP$, one has
\[
(p_2)_*\Bigl(\E(D)\Bigr)(t) \cap (p_2)_*\Bigl(\E/\E(-D)\Bigr)(t)\cong H^0(X,E).
\]

\vskip 3pt

Assume from now on $g=13$ and $k=8$, therefore we expect $B_{\PPP}(8)$ to be $1$-dimensional.
Applying the formalism for Lagrangian degeneracy loci \cite[Proposition 1.11]{Mu}, we find the following determinantal formula for its virtual fundamental class
\begin{equation}\label{eq:bp8}
[B_{\PPP}(8)]^{\vir}=\begin{vmatrix}
c_8 & c_9 & c_{10} & c_{11} & c_{12} & c_{13} & c_{14} & c_{15}\\
c_6 & c_7 & c_{8} &  c_{9} & c_{10} & c_{11} & c_{12} & c_{13}\\
c_4 & c_5 & c_{6} & c_{7} & c_{8} & c_{9} & c_{10} & c_{11}\\
c_2 & c_3 & c_{4} & c_{5} & c_{6} & c_{7} & c_{8} & c_{9}\\
c_0 & c_1 & c_{2} & c_{3} & c_{4} & c_{5} & c_{6} & c_{7}\\
0 & 0 & c_{0} & c_{1} & c_{2} & c_{3} & c_{4} & c_{5}\\
0 & 0 & 0 & 0 & c_{0} & c_{1} & c_{2} & c_{3}\\
0 & 0 & 0 & 0 & 0 & 0 & c_{0} & c_{1}\\
\end{vmatrix}
\end{equation}
where $c_i\in H^{2i}(\PPP, \mathbb Q)$ are defined recursively by the following formulas, see \cite[Corollary 4.2]{LN}:
\begin{equation}\label{eq:recursion_initial}
c_0=2, \ c_1=h, \ c_2=\frac{h^2}{2}, \ c_3=\frac{1}{3}\Bigl(\frac{h^3}{2}+\frac{\beta h}{4}-\frac{\gamma}{2}\Bigr),\
c_4=\frac{1}{4}\Bigl(\frac{h^4}{6}+\frac{\beta h^2}{3}-\frac{2\gamma h}{3}\Bigr),
\end{equation}
and for each $n\geq 1$,
\begin{equation}\label{eq:recursion}
(n+4)c_{n+4}-\frac{n+2}{2}\beta c_{n+2}+\Bigl(\frac{\beta}{4}\Bigr)^2 nc_n=hc_{n+3}-\Bigl(\frac{\beta h}{4}+\frac{\gamma}{2}\Bigr)c_{n+1}.
\end{equation}

In order to evaluate the determinant giving $[B_{\PPP}(8)]^{\vir}$, we shall use Proposition \ref{prop:powers_of_h} coupled with the formula of Thaddeus \cite{Th} determining all top intersection numbers of tautological classes on $SU_X(2,\omega(p))$. Precisely, for $m+2n+3p=3g-3$, one has
\begin{equation}\label{eq:thaddeus}
\int_{SU_X(2, \omega(p))} \alpha^m \cdot \beta^n \cdot \gamma^p=(-1)^{g-p}\frac{g! m!}{(g-p)! q!}2^{2g-2-p}(2^q-2)B_q,
\end{equation}
where $q=m+p+1-g$ and $B_q$ denotes the Bernoulli number; those appearing in our calculation are:

\[
B_2=\frac{1}{6}, \ B_4=-\frac{1}{30}, \ B_6=\frac{1}{42}, \ B_8=-\frac{1}{30}, B_{10}=\frac{5}{66}, \ B_{12}=-\frac{691}{2730}, \ B_{14}=\frac{7}{6},
\]
\[
 B_{16}=-\frac{3617}{510}, \ B_{18}=\frac{43867}{798}, \ B_{20}=-\frac{174611}{330}, \ B_{22}=\frac{854513}{138},
\ B_{24}=-\frac{236364091}{2730}.
\]


\begin{theorem}\label{thm:3bundles}
For a general curve $X$ of genus $13$, the locus $SU_X(2,\omega, 8)$ consists of three reduced points corresponding to stable vector bundles.
\end{theorem}
\begin{proof} As explained, the Lagrangian degeneracy locus $B_{\PPP}(8)$ is expected to be a curve and we write
\[
[B_{\PPP}(8)]^{\vir}=f(\alpha, \beta, \gamma)+h\cdot u(\alpha, \beta, \gamma),
\]
where $f(\alpha, \beta, \gamma)$  and $u(\alpha, \beta, \gamma)$ are homogeneous polynomials of degree $36=3g-3$ and $35=3g-4,$ respectively.

Observe that if $E\in SU_X(2,\omega,8)$ then necessarily $E$ is a stable bundle. Otherwise $E$ is strictly semistable, in which case
$E=B\oplus (\omega_X\otimes B^{\vee})$, where $B\in W^3_{12}(X)$, which contradicts the Brill-Noether Theorem on $X$.
Since $\rho$ is a $\PP^1$-fibration over the locus of stable vector bundles, it follows that $B_{\PPP}(8)$ is a $\PP^1$-fibration over $SU_X(2,\omega,8)$. Furthermore, applying \cite{Te}, the Mukai-Petri map $\mu_E$ is an isomorphism for each vector bundle $E\in SU_X(2,\omega, 8)$,
therefore $SU_X(2,\omega,8)$ is a reduced zero-dimensional cycle. We denote by $a$ its length, thus we can write
\begin{equation}\label{eq:p1fibr}
[B_{\PPP}(8)]=[B_{\PPP}(8)]^{\vir}=a \rho^*([E_0])=f(\alpha,\beta,\gamma)+h\cdot u(\alpha, \beta,\gamma),
\end{equation}
where $[E_0]\in SU_X(2,\omega)$ is general. Intersecting both sides of (\ref{eq:p1fibr}) with $h$, we obtain
\[
h\cdot f(\alpha, \beta, \gamma)=-h\cdot \alpha u(\alpha, \beta, \gamma).
\]

Next observe that $\rho^*([E_0])\cdot \alpha=2$. Indeed, since $\rho$ is a $\PP^1$-fibration over the open locus of stable bundles and
$\omega_{\PPP}=\rho^*(\L_{\mathrm{ev}})\otimes \pi^*(-\alpha)$, it follows that
\[
-2=\mbox{deg}\bigl(\omega_{\PPP |\rho^*([E_0])}\bigr)=\omega_{\PPP}\cdot \rho^*([E_0])=-\alpha\cdot\rho^*([E_0]).
\]
Intersecting both sides of (\ref{eq:p1fibr}) with $\alpha$, we find
$2a=h\cdot \alpha u(\alpha, \beta, \gamma)=-h\cdot f(\alpha, \beta, \gamma)$, so
\[
a=\bigl|SU_X(2, \omega, 8)\bigr|=\frac{1}{2} \int_{\PPP} hf(\alpha, \beta, \gamma)=\frac{1}{2} \int_{SU_X(2, \omega(p))} f(\alpha, \beta, \gamma).
\]

We are left with the task of computing the degree $36$ polynomial $f(\alpha, \beta, \gamma)$, which is a  long but elementary calculation.
We consider the determinant (\ref{eq:bp8}) computing the class of $B_{\PPP}(8)$. First we substitute for each of the classes $c_1, \ldots, c_{15}$ the expression in terms of $\alpha, \beta, \gamma, h$ given by the recursion (\ref{eq:recursion}), starting with the initial conditions (\ref{eq:recursion_initial}). Evaluating this determinant, we obtain a polynomial of degree $36$ in the classes $\alpha, \beta, \gamma$ and  $h$. We recursively express all the powers $h^n$ with $n\geq 2$ and obtain a formula of the form $[B_{\PPP}(8)]=f(\alpha, \beta, \gamma)+h\cdot u(\alpha, \beta, \gamma)$. We set $h=0$ in this formula and then we evaluate each monomial of degree $36$ in $\alpha, \beta, \gamma$ using Thaddeus' formulas (\ref{eq:thaddeus}). At the end, we obtain $f(\alpha, \beta,\gamma)=-6$, which completes the proof of Theorem \ref{thm:3bundles}\footnote{The \emph{Maple} file describing all the calculations explained here can be found at \emph{https://www.mathematik.hu-berlin.de/farkas/gen13bn.mw}}.
\end{proof}

\section{The non-abelian Brill-Noether divisor on $\mm_{13}$}\label{sec:nonab}

In this section we determine the class of the non-abelian Brill-Noether divisor $\overline{\mathcal{M}\mathcal{P}}_{13}$  and prove Theorem \ref{thm:mp}. The results in this section also lay the groundwork for the proof that $\rr_{13}$ is of general type.

\subsection{Tautological classes on the universal non-abelian Brill-Noether locus}
\hfill

\begin{definition} 
Let $\mathfrak{M}_{13}^{\sharp}$ be the open substack of $\overline{\mathfrak{M}}_{13}$  consisting of (i) smooth curves $X$ of genus $13$ with $SU_X(2,\omega,9)=\emptyset$, or of (ii)  $1$-nodal irreducible curves $[X/y\sim q]$, where $X$ is a $7$-gonal smooth genus $12$ curve,  $y, q\in X$, and  such that  the multiplication map $\phi_L\colon \mathrm{Sym}^2 H^0(X,L)\rightarrow H^0(X, L^{\otimes 2})$ is surjective for each $L\in W^5_{15}(X)$. Let $\cM_{13}^{\sharp}$ be the open subset of $\mm_{13}$ coarsely representing $\mathfrak{M}_{13}^{\sharp}$.
\end{definition}

Note that $\cM_{13}^{\sharp}$ and $\cM_{13}\cup \Delta_0$ agree in codimension one, in particular we identify $CH^1(\tm13)$ with $\mathbb Q\langle \lambda, \delta_0\rangle$.  We let $\rtsu$ be the moduli stack of pairs $[X,E]$, with $[X]\in \tm13$ and $E$ is a semistable rank $2$ vector bundle on $X$ with $\mbox{det}(E)\cong \omega_X$ and $h^0(X,E)\geq 8$. Let $\tsu$ be the coarse moduli space of $\rtsu$. We still denote by $\vartheta \colon \rtsu\rightarrow \mathfrak{M}_{13}^{\sharp}$ the forgetful map.

\begin{proposition}\label{prop:tsu}  The map $\vartheta \colon \rtsu\rightarrow \mathfrak{M}_{13}^{\sharp}$ is proper.  Furthermore, for each $[X,E]\in \tsu$ the corresponding vector bundle $E$ is globally generated.
\end{proposition}

\begin{proof}
Suppose $\cX\rightarrow T$ is a flat family of stable curves of genus $13$, such that its generic fibre $X_{\eta}$ is smooth and the special fibre $X_0$ corresponds to a $1$-nodal curve in $\tm13$. The moduli space $SU_{X_{\eta}}(2,\omega)$ specializes to a moduli space $SU_{X_0}(2,\omega)$ that is a closed subvariety of the moduli space $U_{X_0}(2, 24)$ of $S$-equivalence classes of torsion free sheaves of rank $2$ and degree $24$ on $X_0$. The points in $SU_{X_0}(2,\omega)$ are described  in \cite{Su}.

\vskip 3pt

We claim that if $E\in SU_{X_0}(2,\omega)$ satisfies $h^0(X_0, E)\geq 8$, then necessarily $E$ is locally free, in which case $\bigwedge^2 E\cong \omega_{X_0}$. Suppose $\nu\colon X\rightarrow X_0$ is the normalization map, let $y,q\in X$ denote the inverse images of the node $p$ of $X_0$ and assume $E$ is not locally free at $p$. Denoting by $\mathfrak{m}_p\subseteq \cO_{X_0,p}$ the maximal ideal,  either (i) $E_{p}\cong \mathfrak{m}_p\oplus \mathfrak{m}_p$, or else (ii) $E_{p}\cong \cO_{X_0,p}\oplus \mathfrak{m}_p$. In the first case $E=\nu_*(F)$, where $F$ is a vector bundle of rank $2$ on $X$ with $\mbox{det}(F)\cong \omega_X$, that is, $SU_X(2,\omega,8)\neq \emptyset$. Note that
\[
h^0\big(X,\mbox{det}(F)\bigr)=12\leq 2h^0(X,F)-4,
\]
implying that $F$ has a subpencil $A\hookrightarrow F$\footnote{Use that for dimension reasons the determinant map $d\colon \bigwedge^2 H^0(X,F)\rightarrow H^0\bigl(X, \mbox{det}(F)\bigr)$ must necessarily vanish on a pure element $0\neq s_1\wedge s_2$, with $s_1, s_2\in H^0(X, F)$. The subpencil in question is then generated by the sections $s_1$ and $s_2$.}. Then $A\in W^1_7(X)$ and $L:=\omega_X\otimes A^{\vee}\in W^5_{15}(X)$ is such that $\phi_L\colon \mbox{Sym}^2 H^0(X,L)\rightarrow H^0(X, L^{\otimes 2})$ is not surjective. This is ruled out by the definition of $\tm13$. In case (ii), when $E_p\cong \OO_{X_0,p}\oplus \mathfrak m_p$, one has an exact sequence
\[
0\longrightarrow E\longrightarrow \nu_*(\widetilde{F})\longrightarrow  K(p)\longrightarrow 0,
\]
where $\widetilde{F}=\nu^*(E)/\mathrm{Torsion}$ is a vector bundle on the smooth curve $X$ and satisfies $\mbox{det}(F)=\omega_X(y)$, or $\mbox{det}(F)\cong \omega_X(q)$, see also \cite[1.2]{Su}. Observe that also in this case  $F$ necessarily carries a subpencil, and we argue as before to rule out this possibility.

\vskip 3pt

We now turn out to the last part of Proposition \ref{prop:tsu}. Choose $[X, E]\in \tsu$ and assume for simplicity $X$ is smooth (the case when $X$ is $1$-nodal being similar). Assume $E$ is not globally generated at a point $q\in X$.  Then there exists a vector bundle $F\in SU_X(2, \omega(-q),8)$, obtained from $E$ by an elementary transformation at $q$. Note that $h^0\bigl(X, \mathrm{det}(F)\bigr)\leq 2h^0(X,F)-4$, which forces $F$ to have a subpencil $A\hookrightarrow F$. Necessarily, $\mbox{deg}(A)=7$. Since $h^0(F)=h^0(A)+h^0\bigl(\omega_X\otimes A^{\vee}(-q)\bigr)$, setting $L:=\omega_X\otimes A^{\vee}\in W^6_{17}(X)$, it follows that the multiplication map
\[
H^0(X, L)\otimes H^0(X, L(-q))\rightarrow H^0\bigl(X, L^{\otimes 2}(-q)\bigr)
\]
is not surjective, and in particular the map $\mbox{Sym}^2 H^0(X,L)\rightarrow H^0(X,L^{\otimes 2})$ is not surjective either. Then $X$ possesses a stable rank $2$ vector bundle with canonical determinant and $9=h^0(X,A)+h^0(X,L)$ sections, which is not the case.
\end{proof}

Let us consider the universal genus $13$ curve
\[
\wp \colon \mathfrak{C}_{13}^{\sharp} \rightarrow \rtsu,
\]
then let $\mathfrak{E}$ be the universal rank two bundle over the stack $\rtsu$. Note that we can normalize $\mathfrak{E}$ in such a way that
$\mbox{det}(\mathfrak{E})\cong \omega_{\wp}$.

\begin{definition}\label{def:tautclass}
We define the tautological class $\gamma:=\wp_*\bigl(c_2(\mathfrak{E})\bigr)\in CH^1\bigl(\rtsu\bigr)$.
\end{definition}

We aim to  determine the push-forward to $\tm13$ of the class $\gamma$ in terms of $\lambda$ and $\delta_0$. To that end, we begin with the following:

\begin{proposition}\label{prop:rk8}
The push-forward $\wp_*(\mathfrak{E})$ is a locally free sheaf of rank $8$ and
\[
c_1\bigl(\wp_*(\mathfrak{E})\bigr)=\vartheta^*(\lambda)-\frac{\gamma}{2}\in CH^1\bigl(\rtsu\bigr).
\]
\end{proposition}
\begin{proof}
The fact that $\wp_*(\mathfrak{E})$ is locally free follows from \cite{Ha}. We apply Grothendieck-Riemann-Roch to the curve $\wp\colon \mathfrak{C}_{13}^{\sharp}\rightarrow \rtsu$ and to the vector bundle $\mathfrak{E}$ to obtain:
\[
\mathrm{ch}\bigl(\wp_{!}(\mathfrak{E}\bigr)=\wp_*\Bigl[\Bigl(2+c_1\bigl(\mathfrak{E})+\frac{c_1^2(\mathfrak{E})-2c_2(\mathfrak{E})}{2}+\cdots\Bigr)\cdot
\Bigl(1-\frac{c_1(\Omega^1_{\wp})}{2}+\frac{c_1^2(\Omega^1_{\wp})+c_2(\Omega^1_{\wp})}{12}+\cdots\Bigr)\Bigr].
\]
We consider the degree one terms in this equality. Using \cite[page 49]{HM}, observe that
\[
c_1(\Omega_{\wp}^1)=c_1(\omega_{\wp}) \ \mbox{ and } \ \wp_*\Bigl(\frac{c_1^2(\Omega^1_{\wp})+c_2(\Omega^1_{\wp})}{12}\Bigr)=\vartheta^*(\lambda).
\]
By Serre duality, observe that
 $R^1 \wp_*(\mathfrak{E})\cong \wp_*(\mathfrak{E})^{\vee}$, therefore one can write:
\[
2c_1(\wp_*(\mathfrak{E}))=c_1(\wp_*(\mathfrak{E}))-c_1(R^1\wp_*(\mathfrak{E}))=2\vartheta^*(\lambda)-\frac{1}{2}\wp_*\bigl(c_1^2(\omega_{\wp})\bigr)
+\frac{1}{2}\wp_*\bigl(c_1^2(\omega_{\wp})\bigr)-\gamma,
\]
which leads to the claimed formula.
\end{proof}

In view of our future applications to $\rr_{13}$, we  introduce the rank $6$ vector bundle
\[
\cM_{\mathfrak{E}}:=\mbox{Ker}\bigl\{\wp^*(\wp_*(\mathfrak{E}))\rightarrow \mathfrak{E}\bigr\}.
\]
The fibre $M_E:=\cM_{\mathfrak{E}}[X,E]$ over a point $[X, E]\in \tsu$ sits in an exact sequence

\begin{equation}\label{eq:me}
0\longrightarrow M_E\longrightarrow H^0(X,E)\otimes \OO_X\stackrel{\mathrm{ev}}\longrightarrow E\longrightarrow 0,
\end{equation}
where exactness on the right is a consequence of Proposition \ref{prop:tsu}.
\begin{proposition}\label{prop:me}
The following formulas hold:
$c_1\bigl(\cM_{\mathfrak{E}}\bigr)=\wp^*\Bigl(\vartheta^*(\lambda)-\frac{\gamma}{2}\Bigr)-c_1(\omega_{\wp})$ \ \ \mbox{ and   }
\[
 c_2\bigl(\cM_{\mathfrak{E}}\bigr)=\wp^*c_2(\wp_*\mathfrak{E})-c_2(\mathfrak{E})-c_1(\omega_{\wp})\cdot \wp^*\bigl(\vartheta^*(\lambda)-\frac{\gamma}{2}\bigr)+c_1^2(\omega_{\wp}).
 \]
\end{proposition}
\begin{proof}
This follows from the splitting principle applied to $\cM_{\mathfrak{E}}$, coupled with Proposition \ref{prop:rk8}.
\end{proof}

\subsection{The resonance divisor in genus $13$.}
A general curve $X$ of genus $13$ has $3$ stable vector bundles $E\in SU_X(2,\omega, 8)$. In this case $h^0\bigl(X,\mbox{det}(E)\bigr)=2h^0(X,E)-3$, which implies that requiring  $E$ to carry a subpencil defines a divisorial condition on
the moduli space $\mathcal{SU}_{13}(2,\omega,8)$ and thus on $\cM_{13}$. For a vector bundle $E\in SU_X(2,\omega)$, we denote its determinant map by
\[
d\colon \bigwedge^2 H^0(X,E)\rightarrow H^0(X,\omega_X).
\]

\begin{definition}
The \emph{resonance divisor} $\mathfrak{Res}_{13}^{\sharp}$ is the locus of curves $[X]\in \tm13$ for which
\[
G\bigl(2, H^0(X,E)\bigr)\cap \PP\bigl(\mathrm{Ker}(d)\bigr)\neq \emptyset,
\]
for some vector bundle $E\in SU_X(2, \omega, 8)$.  In other words, $\mathfrak{Res}_{13}^{\sharp}$ is the locus of $[X]$ for which there exists  an element $0\neq s_1\wedge s_2\in \bigwedge^2 H^0(X,E)$ such that $d(s_1\wedge s_2)=0$.
\end{definition}

We set  $\mathfrak{Res}_{13}:=\mathfrak{Res}_{13}^{\sharp}\cap \cM_{13}$. Note that $\mathfrak{Res}_{13}^{\sharp}$ comes with an induced scheme structure under the proper map $\vartheta\colon \rtsu \rightarrow \mathfrak{M}_{13}^{\sharp}$. The points in $\mathfrak{Res}_{13}^{\sharp}$ correspond to those curves $X$ for which a vector bundle $E\in SU_X(2,\omega, 8)$ carries a subpencil (which is  generated by the sections $s_1, s_2\in H^0(X, E)$ with $d(s_1\wedge s_2)=0$). The class $[\mathfrak{Res}_{13}^{\sharp}]$ can be computed in terms of certain tautological classes over $\rtsu$. On the other hand, we have a geometric characterization of points in $\mathfrak{Res}_{13}$, and it turns out that the resonance divisor coincides with $\mathfrak{D}_{13}$ away from the heptagonal locus $\cM_{13,7}^1$.

\vskip 4pt

\begin{proof}[Proof of Theorem~\ref{thm:bn13}.]
We show that one has the following equality of effective divisors
\[
\mathfrak{Res}_{13}=\mathfrak{D}_{13}+3\cdot \cM_{ 13,7}^{1}
\]
on $\cM_{13}$. Indeed, let us assume $[X]\in \mathfrak{Res}_{13}\smallsetminus \cM^1_{13,7}$ and let $E\in SU_X(2,\omega,8)$ be the vector bundle which can be written as an extension
\begin{equation}\label{prop:ext2}
0\longrightarrow A\longrightarrow E\longrightarrow \omega_X\otimes A^{\vee}  \longrightarrow 0,
\end{equation}
where $h^0(X,A)\geq 2$.  Since $\mbox{gon}(X)=8$, and since $8\leq h^0(X,E)\leq h^0(X,A)+h^0(X,\omega_X\otimes A^{\vee})$, it  follows that $A\in W^1_8(X)$ and $L:=\omega_X\otimes A^{\vee} \in W^5_{16}(X)$. If such an extension exists, then the map $\phi_L$ is not surjective, therefore  $[X]\in \mathfrak{D}_{13}$.

\vskip 3pt

Conversely, if $[X]\in \mathfrak{D}_{13}$, there is some  $L\in W^5_{16}(X)$ such that the multiplication map $\phi_L$ 
is not surjective. For $[X]$ a general point of an irreducible component of $\mathfrak{D}_{13}$, we may assume that the multiplication map $\phi_L$ has corank $1$, for else $\varphi_L\colon X\hookrightarrow \PP^5$ lies on a $(2,2,2)$ complete intersection in $\PP^5$, which is a (possibly degenerate) $K3$ surface. But the locus of curves $[X]\in \cM_{13}$ lying on a (possibly degenerate) $K3$ surface cannot exceed $g+19=32<3g-4$, a contradiction. We let
\[
E\in \PP\bigl(\mathrm{Ext}^1(L, \omega_X\otimes L^{\vee})\bigr)
\]
be the \emph{unique} vector bundle with $h^0(X,E)=h^0(X,L)+h^0(X,\omega_X\otimes L^{\vee})=8$. The argument of Proposition \ref{prop:extension} shows that $E$ is stable, otherwise there would exist an effective divisor $M$ of degree $4$ on $X$ such that $L(-M)\in W^3_{12}(X)$. Since $\rho(13,3,12)=-3$,  the locus of curves $[X]\in \cM_{13}$ with $W^3_{12}(X)\neq \emptyset$ has codimension at least three in $\cM_{13}$, hence this situation does not occur along a component of $\mathfrak{D}_{13}$. Summarizing, away from the divisor $\cM_{13,7}^1$, the divisors $\mathfrak{Res}_{13}$ and $\mathfrak{D}_{13}$ coincide.

\vskip 4pt

We now show that  $\cM_{13,7}^1$ appears with multiplicity $3$ inside $\mathfrak{Res}_{13}$. Let $X$ be a general $7$-gonal curve of genus $13$ and let $A\in W^1_7(X)$ denote its (unique)  degree $7$ pencil. Set $L:=\omega_X\otimes A^{\vee}\in W^6_{17}(X)$. Each vector bundle $E\in SU_X(2,\omega,8)$ that has a subpencil appears as an extension
\begin{equation}\label{eq:ext7}
0\longrightarrow A\longrightarrow E\stackrel{j}\longrightarrow L\longrightarrow 0.
\end{equation}
In this case $h^0(X,E)=h^0(X,A)+h^0(X,L)-1$.  That is, $V:=\mbox{Im}\bigl\{H^0(E)\stackrel{j}\rightarrow H^0(L)\bigr\}$ is $6$-dimensional. Furthermore, the multiplication map
\[
\mu_V\colon V\otimes H^0(X,L)\rightarrow H^0(X,L^{\otimes 2})
\]
is not surjective. Conversely, each $6$-dimensional subspace $V\subseteq H^0(X,L)$ such that $\mu_V$ is not surjective leads to a vector bundle $E\in \PP\bigl(\mbox{Ext}^1(L, A)\bigr)$ with $h^0(X,E)=8$. The corresponding bundle $E$ is stable unless $V$ is of the form $H^0(X,L(-p))$ for a point $p\in X$, in which case $E$ can also be realized as an extension
\[
0\longrightarrow L(-p)\longrightarrow E\longrightarrow A(p)\longrightarrow 0.
\]

\vskip 4pt

To determine the number of such subspaces $V\subseteq H^0(X,L)$, we consider the projective space $\PPP^6:=\PP\bigl(H^0(X,L)^{\vee}\bigr)$ and consider the vector bundle $\cA$ on $\PPP^6$ with fibre
\[
\cA(V)=\frac{V\otimes H^0(X,L)}{\bigwedge^2 V}
\]
over a point $[V]\in \PPP^6$. There exists a bundle morphism $\mu \colon \cA \rightarrow H^0(X,L^{\otimes 2})\otimes \OO_{\PPP^6}$ given by multiplication and the
subspaces $[V]\in \PPP^6$ for which $\mu_V$ is not surjective (or, equivalently, $\mu^{\vee}$ is not injective) are precisely those lying in the degeneracy locus of $\mu$, that is, for which $\mbox{rk}(\mu(V))=21$. Applying the Porteous formula we find
\[
[Z_{21}(\mu)]=c_6\Bigl(H^0(X,L^{\otimes 2})^{\vee}\otimes \OO_{\PPP^6}-\cA^{\vee} \Bigr)=c_6(-\cA).
\]

To compute the Chern classes of $\cA$, we recall that via the Euler sequence the rank $6$ vector bundle $M_{\PPP^6}$ on $\PPP^6$ with
$M_{\PPP^6}(V)=V\subseteq H^0(X,L)$ can be identified with $\Omega_{\PPP^6}(1)$. Then $\cA$ is isomorphic to $M_{\PPP^6}\otimes H^0(X,L)/\bigwedge^2 M_{\PPP^6}$.
From the exact sequence
\[
0\longrightarrow \bigwedge^2 M_{\PPP^6} \longrightarrow \bigwedge^2 H^0(X,L)\otimes \OO_{\PPP^7}\longrightarrow M_{\PPP^6}(1)\longrightarrow 0,
\]
recalling that $c_{\mathrm{tot}}(M_{\PPP^6})=\frac{1}{1+h}$, where $h=c_1(\OO_{\PPP^6}(1))$, we find $c_{\mathrm{tot}}\bigl(\bigwedge^2 M_{\PPP^6}\bigr)=\frac{1+2h}{(1+h)^7}$, therefore
\[
[Z_{21}(\mu)]=\Bigl[\frac{1}{(1+h)^7}\cdot \frac{(1+h)^7}{1+2h}\Bigr]_6=\Bigl[\frac{1}{1+2h}\Bigr]_6=2^6\cdot h^6=64.
\]

From this, we subtract the excess contribution corresponding to the  locus $X\stackrel{|L|}\hookrightarrow \PPP^6$, parametrizing the subspaces $V=H^0(X,L(-p))$ corresponding to unstable bundles.  Via the excess Porteous formula \cite[Example 14.4.7]{Fu}, this locus appears in the class $[Z_{21}(\mu)]$ with a contribution of
\[
c_1\Bigl(\mathrm{Ker}(\mu^{\vee})\otimes \mathrm{Coker}(\mu^{\vee})-N_{X/\PPP^6}\Bigr)=-5c_1\bigl(\mathrm{Ker}(\mu^{\vee})\bigr)+c_1\bigl(\cA^{\vee}_{|X}\bigr)-c_1(N_{X/\PPP^6}).
\]
The restriction to $X\subseteq \PPP^6$ of the kernel bundle of $\mu^{\vee}$ can be identified with  $L^{\vee}$, whereas
$c_1\bigl(\cA^{\vee}_{|X}\bigr)=-2c_1\bigl(M_{\PPP^6|X}\bigr)=2\mbox{ deg}(L)$.  Furthermore $c_1\bigl(N_{X/\PPP^6}\bigr)=7\mbox{deg}(L)+2g(X)-2$. All in all, the excess contribution to $[Z_{21}(\mu)]$ coming from $X$ equals
\[
10\ \mbox{deg}(L)+2\ \mbox{deg}(L)-7\ \mbox{deg}(L)-2g(X)-2=5\cdot 17-24=61.
\]
Therefore, for a general curve $[X]\in \cM_{13,7}^1$, there are $3=64-61$ vector bundles $E\in SU_X(2,\omega,8)$ having $A$ as a subpencil, which finishes the proof.
\end{proof}

\vskip 4pt

We are now in a position to explain how Theorems \ref{thm:main13} and \ref{thm:bn13} provide enough geometric information to determine the push-forward to $\mathfrak{M}_{13}^{\sharp}$ of the  class $\gamma$.

\begin{proposition}\label{prop:gamma}
One has $\vartheta_*(\gamma)=\frac{11288}{143}\lambda-\frac{1582}{143}\delta_0\in CH^1(\tm13)$.
\end{proposition}

\begin{proof}
The divisor $\mathfrak{Res}_{13}^{\sharp}$ is defined as the push-forward under $\vartheta\colon \rtsu\rightarrow \mathfrak{M}_{13}^{\sharp}$ of the locus where the fibers of the morphism of vector bundles
\[
d\colon \bigwedge^2 \wp_*(\mathfrak{E})\rightarrow \wp_*(\omega_{\wp})
\]
contain a rank two tensor in their kernel. To compute the class of this locus, we use Proposition \ref{prop:rk8} in combination with \cite[Theorem 1.1]{FR}\footnote{The result in \cite{FR} is stated for a morphism of vector bundles of the form $\mbox{Sym}^2(\E)\rightarrow \F$. An immediate inspection of the proof shows though that the \emph{same formula} applies also in the setting of a morphism of the form $\bigwedge^2 (\E)\rightarrow \F$.}:
\[
[\mathfrak{Res}^{\sharp}_{13}]=132\Bigl(c_1\bigl(\wp_*(\omega_{\wp}))-\frac{13}{4} c_1\bigl(\wp_*(\mathfrak{E})\bigr)\Bigr)=132\Bigl(-\frac{9}{4}\vartheta^*(\lambda)+\frac{13}{8}\gamma\Bigr).
\]
Using \cite{HM} we write $[\mm_{13,7}^1]=6\cdot (48\lambda-7\delta_0-\cdots)$ for the class of the heptagonal locus, while the class $[\widetilde{\mathfrak{D}}_{13}]$ is computed by Theorem \ref{rho1virtual}. Since $\mbox{deg}(\vartheta)=3$,  we then find
\[
\vartheta_*(\gamma)=\frac{48}{13}\Bigl(\frac{5059}{264}\lambda-\frac{749}{264}\delta_0+\frac{9}{8}\lambda+\frac{3}{132}(48\lambda-7\delta_0)\Bigr)=
\frac{1128}{143}\lambda-\frac{1582}{143}\delta_0.
\]
\end{proof}

\subsection{The class of the non-abelian Brill-Noether divisor on $\mm_{13}$.}\hfill

In the introduction, we defined the non-abelian Brill-Noether divisor $\mathcal{M}\mathcal{P}^{\sharp}_{13}$ as the locus of curves $[X]\in \tm13$ for which there exists $E\in SU_X(2,\omega,8)$ such that the map
\[
\mu_E\colon \mbox{Sym}^2 H^0(X,E)\rightarrow H^0(X, \mathrm{Sym}^2 E)
\]
is not an isomorphism, or equivalently, the scheme $SU_X(2,\omega, 8)$ is not reduced. We now compute the class of this divisor.

\vskip 4pt

\begin{proof}[Proof of Theorem~\ref{thm:mp}.]
The locus $\mathcal{M}\mathcal{P}_{13}^{\sharp}$ is the push-forward under the proper map $\vartheta$ of the degeneracy locus of the following map of vector bundles over $\rtsu$:
\[
\mathrm{Sym}^2 \wp_*(\mathfrak{E})\rightarrow \wp_*\bigl(\mathrm{Sym}^2\mathfrak{E}\bigr).
\]

Using Grothendieck-Riemann-Roch for $\wp\colon \mathfrak{C}_{13}^{\sharp} \rightarrow \rtsu$, we compute
\[
c_1\Bigl(p_{*}(\mathrm{Sym}^2\mathfrak{E})\Bigr)=\wp_*\Bigl[\Bigl(3+3c_1\bigl(\mathfrak{E})+\frac{5c_1^2(\mathfrak{E})-8c_2(\mathfrak{E})}{2}\Bigr)\cdot
\Bigl(1-\frac{c_1(\Omega^1_{\wp})}{2}+\frac{c_1^2(\Omega^1_{\wp})+c_2(\Omega^1_{\wp})}{12}\Bigr)\Bigr]_2 .
\]
Using again that \
$12 \wp_*\Bigl(c_1^2(\Omega^1_{\wp})+c_2(\Omega^1_{\wp})\Bigr)=\vartheta^*(\lambda)$,  we conclude that
\[
c_1\Bigl(\wp_*(\mbox{Sym}^2\mathfrak{E})\Bigr)=3\vartheta^*(\lambda)+\wp_*\bigl(c_1^2(\omega_{\wp})\bigr)-4\gamma
=\vartheta^*\bigl(15\lambda-\delta_0\bigr)-4\gamma.
\]
Via Proposition \ref{prop:rk8}, we have $c_1\bigl(\mbox{Sym}^2 \wp_*(\mathfrak{E})\bigr)=9c_1\bigl(\wp_*(\mathfrak{E})\bigr)=9\bigl(\vartheta^*(\lambda)-\frac{\gamma}{2}\bigr)$, yielding
\[
[\mathcal{M}\mathcal{P}_{13}^{\sharp}]=\vartheta_*\Bigl(c_1\bigl(\wp_*(\mbox{Sym}^2\mathfrak{E})-\mbox{Sym}^2 \wp_*(\mathfrak{E})\bigr)\Bigr)=3(6\lambda-\delta_0)+\frac{\vartheta_*(\gamma)}{2}.
\]
Substituting via Proposition \ref{prop:gamma}, we find
$[\mathcal{M}\mathcal{P}_{13}^{\sharp}]=\frac{1}{143}\bigl(8218\ \lambda-1220\ \delta_0\bigr)$.
\end{proof}

\section{The Kodaira dimension of $\rr_{13}$.}

We turn our attention to showing that the Prym moduli space $\rr_{13}$ is a variety of general type.  We begin by recalling basics on the geometry of the moduli of Prym variety, referring to \cite{FL} for details. We denote by $\overline{\mathfrak{R}}_g:=\MM_g\bigl(\mathcal{B}\mathbb Z_2\bigr)$ the Deligne-Mumford stack of \emph{Prym curves} of genus $g$ classifying triples
$[Y, \eta, \beta]$, where $Y$ is a nodal curve of genus $g$ such that each of its rational components meets the rest of the curve in at least two points, $\eta\in \mathrm{Pic}^0(Y)$ is a line bundle
of total degree $0$ such that $\eta_{|R}=\OO_R(1)$ for every rational
component $R\subseteq Y$ with $\bigl|R\cap \overline{Y\smallsetminus R}\bigr|=2$ (such a component is called \emph{exceptional}), and $\beta\colon \eta^{\otimes 2}\rightarrow \OO_Y$ is a morphism generically non-zero along
each non-exceptional component of $Y$. Let $\rr_g$ be the coarse moduli space of $\mathfrak{R}_g$.  One has a finite cover
\[
\pi \colon \rr_g\rightarrow \mm_g.
\]

\subsection{The boundary divisors of $\rr_g$.} The geometry of the boundary of $\rr_g$ is described in \cite{FL} and we recall some facts. If $[X_{yq}=X/y\sim q]\in \Delta_0\subseteq \mm_g$ is such that $[X, y, q]\in \cM_{g-1, 2}$, denoting by $\nu\colon X\rightarrow X_{yq}$ the normalization map, there are three types of Prym curves in the fibre $\pi^{-1}\bigl([X_{yq}]\bigr)$. First, one can choose a non-trivial $2$-torsion point $\eta\in \mbox{Pic}^0(X_{yq})$. If $\nu^*(\eta)\neq \OO_X$, this amounts to choosing a $2$-torsion point $\eta_X\in \mbox{Pic}^0(X)[2]\smallsetminus \{\OO_X\}$ together with an identification of the fibers $\eta_X(y)$ and $\eta_X(q)$ at the points $y$ and $q$ respectively. As we vary $[X, y,q]$, points of this type fill-up the boundary divisor $\Delta_0^{'}$ in $\rr_g$. The Prym curves corresponding to the situation $\nu^*(\eta)\cong \OO_X$ fill-up the boundary divisor $\Delta_0^{''}$. Finally, choosing a line bundle $\eta_X$ on $X$ with $\eta_X^{\otimes 2}\cong \OO_X(-y-q)$ leads to a Prym curve  $[Y:=X\cup_{y,q} R, \eta, \beta]$, where $R$ is a smooth rational curve meeting $X$ at $y$ and $q$ and $\eta\in \mbox{Pic}^0(Y)$ is a line bundle such that $\eta_{|X}=\eta_X$ and $\eta_{|R}=\OO_R(1)$. Points of this type fill-up the boundary divisor $\Delta_{0}^{\mathrm{ram}}$ of $\rr_g$, which is the ramification divisor of the morphism $\pi$.

\vskip 3pt

Denoting by $\delta_0^{'}:=[\Delta_0^{'}]$, $\delta_0^{''}:=[\Delta_0^{''}]$ and $\delta_0^{\mathrm{ram}}:=[\delta_0^{\mathrm{ram}}]$ the corresponding divisor classes, one has the following relation in $CH^1(\rr_g)\cong CH^1(\overline{\mathfrak{R}}_g)$, see \cite{FL}:
\[
\pi^*(\delta_0)=\delta_0^{'}+\delta_0^{''}+2\delta_0^{\mathrm{ram}}.
\]
The finite morphism $\pi\colon \rr_g\rightarrow \mm_g$ being ramified only along the divisor $\Delta_0^{\mathrm{ram}}$, one has
\begin{equation}\label{eq:can}
K_{\rr_g}=13\lambda-2(\delta_0^{'}+\delta_0^{''})-3\delta_0^{\mathrm{ram}}-2\sum_{i=1}^{\lfloor \frac{g}{2}\rfloor}(\delta_i+\delta_{g-i}+\delta_{i:g-i})-(\delta_1+\delta_{g-1}+\delta_{1:g-1}),
\end{equation} where $\pi^*(\delta_i)=\delta_{i}+\delta_{g-i}+\delta_{i:g-i}$, see \cite[Theorem 1.5]{FL} for details.

\subsection{The universal theta divisor on $\rr_{13}$.} \hfill

\vskip 3pt

For a semistable vector bundle $E\in SU_X(2,\omega)$ on a smooth curve $X$ of genus $g$, its \emph{Raynaud theta divisor}
$\Theta_E:=\bigl\{\xi\in \mbox{Pic}^0(X): H^0(X, E\otimes \xi)\neq 0\bigr\}$ is a $2\theta$-divisor inside the Jacobian of $X$, see \cite{Ray}.

\begin{definition}
The universal theta divisor $\Theta_{13}$ on $\cR_{13}$ is defined as the locus of smooth Prym curves $[X, \eta]\in \cR_{13}$ for which there exists a vector bundle $E\in SU_X(2, \omega, 8)$ such that $H^0(X, E\otimes \eta)\neq 0$.
\end{definition}

We first show that, as expected,  this definition gives rise to a divisor on $\cR_{13}$.

\begin{proposition}\label{prop:transv_raynaud}
For a general Prym curve $[X, \eta]\in \cR_{13}$ one has $H^0(X, E\otimes \eta)=0$ for all $E\in SU_X(2,\omega, 8)$.
It follows that $\Theta_{13}$ is an effective divisor on $\cR_{13}$.
\end{proposition}
\begin{proof}
Consider the subvariety of $\cR_{13}\times_{\cM_{13}} \mathcal{SU}_{13}(2,\omega,8)$
\[
\cZ := \bigl\{ [X, \eta, E] :H^0(X, E\otimes \eta)\neq 0\bigr\}.
\]
Assume for contradiction that $\cZ$ surjects onto $\cR_{13}$. Then $\cZ$ is a union of \emph{irreducible} components of
$\cR_{13}\times_{\cM_{13}} \mathcal{SU}_{13}(2, \omega,8)$. In particular $\cZ$ surjects onto the irreducible variety $\mathcal{SU}_{13}(2, \omega,8)$; see Corollary \ref{cor:unirational}. Therefore, for every pair $[X, E]\in \mathcal{SU}_{13}(2,\omega, 8)$, there exists  a $2$-torsion point $\eta$ on $X$ with $H^0(X, E\otimes \eta)\neq 0$.

\vskip 3pt

We now specialize to the case when $E$ is a strictly semistable vector bundle of the type
\[
E=A^{\otimes 3} \oplus (\omega_X\otimes A^{\otimes (-3)}),
\]
where $[X, A]$ is a general tetragonal curve of genus $13$. Note that  $h^0(X, A^{\otimes 3})=4,$ by \cite[Proposition~2.1]{CM}. In particular $h^0(X, E)=8$. Using \cite{BF2} the space $\cR_{13}\times _{\cM_{13}} \cM_{13,4}^1$ parametrizing Prym curves over tetragonal curves of genus $13$ is irreducible, therefore $H^0\bigl(X, A^{\otimes 3}\otimes \eta)\neq 0$ for \emph{every} triple $[X, \eta, A]\in \cR_{13}\times_{\cM_{13}} \cM_{13,4}^1$. We now further specialize the tetragonal curve $X$ to a hyperelliptic curve and $A=A_0(x+y)$, where $A_0\in W^1_2(X)$ and $x,y\in X$ are general points, whereas
\[
\eta=\OO_X(p_1+p_2+p_3+p_4-q_1-q_2-q_3-q_4)\in \mbox{Pic}^0(X)[2],
\]
with $p_1, \ldots, p_4, q_1, \ldots, q_4$ being mutually distinct Weierstrass points of $X$. It immediately follows that for these choices $H^0\bigl(X, A^{\otimes 3}\otimes \eta\bigr)=0$, which is a contradiction.
\end{proof}

\vskip 3pt

We consider the open substack $\trs13:=\pi^{-1}(\mathfrak{M}_{13}^{\sharp})$ of $\overline{\mathfrak{R}}_{13}$ and let $\tr13$ its associated coarse moduli space. We identify $CH^1(\tr13)$ with the space $\mathbb Q\langle \lambda, \delta_0^{'}, \delta_0^{''}, \delta_0^{\mathrm{ram}}\rangle$.
In what follows we extend the structure on the universal theta divisor $\Theta_{13}$ to $\tr13$ and realize it as the push-forward of the degeneracy locus of a map of vector bundles of the same rank over the fibre product
\[
\rsu:=\trs13\times_{\mathfrak{M}_{13}^{\sharp}} \rtsu.
\]

We start with a triple $[X, \eta, E]\in \rsu$. Via Proposition \ref{prop:tsu} the vector bundle $E$ is  globally generated and we let $M_E:=\mbox{Ker}\bigl\{H^0(X,E)\otimes \OO_X\rightarrow E\bigr\}$. By tensoring with $\eta$ and taking cohomology in the exact sequence (\ref{eq:me}), we observe that $H^0(X, E\otimes \eta)\neq 0$ if and only if the coboundary map

\begin{equation}\label{eq:coboundary}
\upsilon\colon H^1\bigl(X, M_E\otimes \eta\bigr)\rightarrow H^0(X,E)\otimes H^0(X, \omega_X\otimes \eta)^{\vee}
\end{equation}
is not injective. Since clearly $H^0(X, M_E\otimes \eta)=0$, it follows that
\[
h^1(X, M_E\otimes \eta)=-\mbox{deg}(M_E)+6(g-1)=96=8\cdot 12=h^0(X,E)\cdot h^0(X, \omega_X\otimes \eta).
\]
That is, $\upsilon$ is a map between vector space of the same dimension.

By slightly abusing notation, we still denote by
\[
\wp \colon \mathfrak{RC}_{13}^{\sharp} \rightarrow \rsu
\]
the universal curve of genus $13$ over $\rsu$. It comes equipped with a universal rank $2$ vector bundle $\mathfrak{E}$ such that
$\bigwedge^2\mathfrak{E}\cong \omega_{\wp}$ and $\wp_*(\mathfrak{E})$ is locally free of rank $8$ (cf. Proposition \ref{prop:rk8}), as well as with a universal Prym line bundle $\mathcal{L}$ with
$\mathcal{L}_{|\wp^{-1}([X, \eta, E])}\cong \eta$, for any point $[X, \eta, E]\in \rsu$.

We consider the rank $6$ vector bundle $\cM_{\mathfrak{E}}$ on $\tcr13$ defined by the exact sequence
\[
0\longrightarrow \cM_{\mathfrak{E}}\longrightarrow \wp^*\bigl(\wp_*\mathfrak{E}\bigr) \longrightarrow \mathfrak{E}\longrightarrow 0,
\]
then introduce the following sheaves over $\rsu$:
\[
\cA:=R^1\wp_*\bigl(\cM_{\mathfrak{E}}\otimes \mathcal{L}\bigr) \ \ \mbox{ and } \
 \ \cB:=\wp_*(\mathfrak{E})\otimes \wp_*\bigl(\omega_{\wp}\otimes \mathcal{L}\bigr)^{\vee}.
 \]
Using the fact that the map $v$ defined in (\ref{eq:coboundary}) is a morphism between two vector space of the same dimension for every point $[X, \eta, E]\in \rsu$, via Grauert's Theorem we conclude that both $\cA$ and $\cB$ are locally free of the same rank $96$, and there exists a morphism
\begin{equation}\label{eq:upsilon}
\upsilon \colon\cA\rightarrow \cB
\end{equation}
whose fibre restrictions are the maps (\ref{eq:coboundary}). Recall that the
forgetful map $\vartheta \colon \rsu\rightarrow \mathfrak{R}_{13}^{\sharp}$ is generically finite of degree $3$. We denote by $\Theta_{13}^{\sharp}$ the push-forward to $\tr13$ of the degeneracy locus of the morphism $\upsilon$ given by (\ref{eq:upsilon}). Observe that $\Theta_{13}^{\sharp}\cap \cM_{13}=\Theta_{13}$.

\begin{theorem}\label{thm:univtheta13}
The class of the universal theta divisor $\Theta_{13}^{\sharp}$ on $\cR_{13}$ is given by
\[
\bigl[\Theta_{13}^{\sharp}\bigr]=\frac{1}{143}\Bigl(10430 \ \lambda-1582 \ (\delta_0^{'}+\delta_0^{''})-\frac{5899}{2}\delta_0^{\mathrm{ram}}\Bigr)\in CH^1\bigl(\tr13\bigr).
\]
\end{theorem}

\begin{proof} From Proposition \ref{prop:transv_raynaud} it follows that $\upsilon$ is generically non-degenerate, therefore
\[
[\Theta_{13}^{\sharp}]=c_1\bigl(\cB-\cA).
\]
Computing the class $c_1(\cB)$ is straightforward. We find that  $c_1\bigl(\wp_*(\omega_{\wp}\otimes \mathcal{L}\bigr)\bigr)=\vartheta^*\bigl(\lambda-\frac{\delta_0^{\mathrm{ram}}}{4}\bigr)$, using \cite[Proposition 1.7]{FL}. Then via Proposition \ref{prop:rk8}, we  compute
\begin{align*}
c_1(\cB)=12c_1\bigl(\wp_*\mathfrak{E}\bigr)-8c_1\bigl(\wp_*(\omega_{\wp}\otimes \mathcal{L})\bigr)& =12\Bigl(\vartheta^*(\lambda)-\frac{\gamma}{2}\Bigr)-8\Bigl(\vartheta^*\bigl(\lambda-\frac{\delta_0^{\mathrm{ram}}}{4}\Bigr)\Bigr)
\\& =\vartheta^*\bigl(4\lambda+2\delta_0^{\mathrm{ram}}\bigr)-6\gamma.
\end{align*}

To determine $c_1(\cA)$ we apply Grothendieck-Riemann-Roch to the morphism $\wp$:

\begin{align*}
\mathrm{ch}\bigl(\wp_{!}(\cM_{\mathfrak{E}}\otimes \mathcal{L})\bigr)=\wp_{*}\Bigl[\Bigl(6+c_1(\cM_{\mathfrak{E}}\otimes \mathcal{L})+\frac{c_1^2(\cM_{\mathfrak{E}}\otimes \mathcal{L})-2c_2(\cM_{\mathfrak{E}}\otimes \mathcal{L})}{2}+\cdots\Bigr)\\
\numberthis \label{eq:grr} \cdot \Bigl(1-\frac{c_1(\Omega^1_{\wp})}{2}+\frac{c_1^2(\Omega^1_{\wp})+c_2(\Omega^1_{\wp})}{12}+\cdots\Bigr)\Bigr].
\end{align*}

Observe by direct calculation that the following formulas hold:
\[
c_1(\cM_{\mathfrak{E}}\otimes \mathcal{L})=c_1(\cM_{\mathfrak{E}})+6c_1(\mathfrak{L}), \ \ \ c_2(\cM_{\mathfrak{E}}\otimes \mathcal{L})=c_2(\cM_{\mathfrak{E}})+5c_1(\cM_{\mathfrak{E}})\cdot c_1(\mathcal{L})+15c_1^2(\mathcal{L}),
\]
therefore
\begin{align*}\label{eq:twist2}
\wp_*\Bigl(\frac{c_1^2(\cM_{\mathfrak{E}}\otimes \mathcal{L})-2c_2(\cM_{\mathfrak{E}}\otimes \mathcal{L})}{2}\Bigr) &=\wp_*\Bigl(\frac{c_1^2(\cM_{\mathfrak{E}})-2c_2(\cM_{\mathfrak{E}})}{2}+c_1(\cM_{\mathfrak{E}})\cdot c_1(\mathcal{L})+3c_1^2(\mathcal{L})\Bigr)\\
&=\gamma-\frac{1}{2}\wp_*(c_1^2(\omega_{\wp}))=\gamma-\frac{1}{2}\Bigl(\vartheta^*\bigl(12\lambda-\delta_0^{'}-\delta_0^{''}-2\delta_0^{\mathrm{ram}}\bigr)
\Bigr),
\end{align*}
where in the last formula we have used Proposition \ref{prop:me}, Mumford's formula \cite{HM} for the class $\wp_*(c_1^2(\omega_{\wp}))$, and $2\wp_*\bigl(c_1^2(\mathcal{L})\bigr)=-\vartheta^*(\delta_0^{\mathrm{ram}})$; see \cite[Proposition 1.6]{FL}.

\vskip 3pt

Substituting in the equation  (\ref{eq:grr}), coupled with Proposition \ref{prop:me} and also using that via the push-pull formula  $\wp_*\bigl(\wp^*(\vartheta^*(\lambda)-\frac{\gamma}{2})\cdot c_1(\omega_{\wp})\bigr)=(g-1)\cdot \bigl(\vartheta^*\bigl(\lambda)-\frac{\gamma}{2}\bigr)$, we obtain
\[
c_1(\cA)=-7\gamma+\vartheta^*\bigl(6\lambda+\frac{3}{2}\delta_0^{\mathrm{ram}}\bigr).
\]
Putting everything together we find
\[
[\Theta_{13}^{\sharp}]=\vartheta_*c_1(\cB-\cA)=\vartheta_*\Bigl(\gamma-2\lambda+\frac{\delta_0^{\mathrm{ram}}}{2}\Bigr)=2\vartheta_*(\gamma)-6\lambda+
\frac{3}{2}\delta_0^{\mathrm{ram}}.
\]
Finally Proposition \ref{prop:gamma} gives $143\ \vartheta_*(\gamma)=11288\lambda-1582(\delta_0^{'}+\delta_0^{''}+2\delta_0^{\mathrm{ram}})$ and the conclusion follows.
\end{proof}

We can now complete the proof that $\rr_{13}$ is of general type.

\vskip 4pt

\begin{proof}[Proof of Theorem \ref{thm:r13}.]
It is shown in \cite[Theorem 6.1]{FL} that any $g$ pluricanonical forms defined on $\rr_g$ automatically extend to any resolution of singularities, therefore $\rr_g$ is of general type if and only if the canonical class $K_{\rr_g}$ is big, that is, can be expressed as a positive rational combination of an ample and an effective class on $\rr_g$. To that end, we shall use apart from the closure $\overline{\Theta}_{13}$ in $\rr_{13}$ of the universal theta divisor $\Theta_{13}$, the divisor $D_{13:2}$ on $\cR_{13}$ consisting of pairs $[X,\eta]$ where the $2$-torsion point $\eta$ lies in the divisorial \emph{difference variety}
\[
X_6-X_6=\Bigl\{\OO_X(D-E): D, E\in X_6\Bigr\}\subseteq \mbox{Pic}^0(X).
\]
It is shown in \cite[Theorem 0.2]{FL} that up to a positive rational constant, the closure of $D_{13:2}$
inside$\rr_{13}$ is given by $[\overline{D}_{13:2}]=19\lambda-3(\delta_0^{'}+\delta_0^{''})-\frac{13}{4}\delta_0^{\mathrm{ram}}-\cdots\in CH^1(\rr_{13})$. Observe that by construction $\Theta_{13}^{\sharp}$ differs from the restriction of $\overline{\Theta}_{13}$ to $\cM_{13}^{\sharp}$ by a (possibly empty) \emph{effective} combination of the divisors $\Delta_0^{'}, \Delta_0^{''}$ and $\Delta_0^{\mathrm{ram}}$, hence using Theorem \ref{thm:univtheta13} we can write
$$[\overline{\Theta}_{13}]=\frac{1}{143}\Bigl(10430 \lambda-b_0^{'} \delta_0^{'}-b_0^{''} \delta_0^{''}-b_0^{\mathrm{ram}} \delta_0^{\mathrm{ram}}-\cdots\Bigr)\in CH^1(\rr_{13}),$$
where $b_0^{'}\geq 1582$, $b_0^{''}\geq 1582$ and $b_0^{\mathrm{ram}}\geq \frac{5899}{2}$. We consider the following effective divisor on $\rr_{13}$:

\[
\cD:=\frac{65}{674}[\overline{\Theta}_{13}]+\frac{1153}{3707}[\overline{D}_{13:2}]=a\lambda-a_0^{'}\delta_0^{'}-a_0^{''}\delta_0^{''}-
a_0^{\mathrm{ram}}\delta_0^{\mathrm{ram}}-\sum_{i=1}^{12} a_i\delta_i- \sum_{i=1}^6 a_{i,13-i}\delta_{i:13-i},
\]
where $a=\frac{4362}{337}$, $a_0^{'}\geq 2$, $a_0^{''}\geq 2$ and $a_0^{\mathrm{ram}}\geq 3$. By an argument using pencils on $K3$ surfaces, one can show that each of the coefficients $a_1, \ldots, a_{12}$ or
$a_{1,12}, \ldots, a_{6,7}$ is at least equal to $3$. Indeed, each boundary divisor $\Delta_i$ or $\Delta_{i:13-i}$ of $\rr_{13}$ is covered by pencils of reducible Prym curves, consisting of two components, of which one moves in a suitable Lefschetz pencil on a \emph{fixed} $K3$ surface. The intersection numbers of these pencils with the generators of $CH^1(\rr_g)$ have been in computed in \cite[Proposition 1.8]{FL}. Since $\cD$ is the closure in $\rr_{13}$ of an effective divisor on $\cR_{13}$,  the intersection number of each such pencil with $\cD$ is non-negative.  For instance, for $1\leq i\leq 6$ we obtain in this way the inequality,
$$a_{13-i}\geq a_0^{'}(6i+18)-a(i+1)\geq 2(6i+18)-\frac{4362}{337}(i+1)\geq 3.$$ The inequalities for the remaining coefficients of $\cD$ can be handled similarly, see also \cite[Proposition 1.9]{FL}.   Since $a=12.943...<13$, comparing the class of $\cD$ to that of $K_{\rr_{13}}$ given in (\ref{eq:can}), we conclude that $K_{\rr_{13}}$ can be written as a positive combination of $[\cD]$ and a multiple of $\lambda$, hence it is big.
\end{proof}

\vskip 4pt

\subsection{The Kodaira dimension of $\mm_{13,n}$.} We indicate how our results on divisors on $\mm_{13}$ can be used to determine the Kodaira dimension of the moduli space $\mm_{13,n}$.

\noindent
\begin{proof}[Proof of Theorem~\ref{m13n}.]
It suffices to show that $\mm_{13,9}$ is of general type to conclude that the same holds for $\mm_{13,n}$ when $n\geq 10$. We use the divisor $\cD_{13:2^4,1^5}$ considered by Logan \cite{Log} and defined as the $\mathfrak{S}_9$-orbit (under the action permuting the marked points) of the locus of  pointed curves
$[X, p_1,\ldots, p_9]\in \cM_{13,9}$ such that
\[
h^0\bigl(X, \OO_X(2p_1+\cdots+2p_4+p_5+\cdots+p_9)\bigr)\geq 2.
\]
Up to a positive constant the class of the closure in $\mm_{13,9}$ of $D_{13:2^4,1^5}$ equals
\[
[\overline{\cD}_{13:2^4,1^5}]=-\lambda+\frac{17}{9}\sum_{i=1}^9\psi_i-\frac{25}{6}\delta_{0:2}-\cdots  \in CH^1(\mm_{13,9}).
\]
(See \cite{F} or \cite{Log} for the standard notation on the generators of $CH^1(\mm_{g,n})$.)
If $\pi\colon \mm_{13,9}\rightarrow \mm_{13}$ is the map forgetting the marked points, a routine calculation shows that the canonical class
$K_{\mm_{13,9}}$ can be expressed as a positive linear combination of $[\overline{\cD}_{13:2^4, 1^5}]$ and $\pi^*([D])$, where $D\in \mbox{Eff}(\mm_{13})$ if and only if $2s(D)-\frac{9}{17}<13$. Observe that the class of the non-abelian Brill-Noether divisor $[\overline{\mathcal{M}\mathcal{P}}_{13}]$ verifies this inequality,  and the result follows.
\end{proof}


\end{document}